\documentclass[10pt,twocolumn,paper=letter]{scrartcl}

\usepackage[bf,format=plain,singlelinecheck=false]{caption}
\usepackage{graphicx}
\usepackage{amsmath}
\usepackage{amsfonts}
\usepackage{amssymb}
\usepackage{subcaption}
\usepackage{authblk}
\usepackage[utf8]{inputenc}
\usepackage[english]{babel}
\usepackage{amsmath}
\usepackage{amsfonts}
\usepackage{amssymb}
\usepackage{enumerate}
\usepackage{graphicx}
\usepackage{cite}
\usepackage[noend]{algpseudocode}
\usepackage{caption}
\usepackage[dvipsnames]{xcolor}
\usepackage{placeins}
\usepackage{caption}
\usepackage{subcaption}
\usepackage{abstract}

\newcommand{\X}{\mathcal{X}}
\newcommand{\I}{\mathcal{I}}

\newcommand{\E}{\mathbb{E}}
\renewcommand{\P}{\mathbb{P}}
\newcommand{\Yobs}{Y}

\newcommand{\var}{\mbox{Var}}

\newcommand{\cov}{\mbox{Cov}}
\newcommand{\R}{\mathbb{R}}
\newcommand{\Dt}{\tilde D}
\newcommand{\taun}{\hat{\tau}^{mn}}
\newcommand{\theoremfont}{\itshape}

\newcommand{\D}{\mathcal{D}}

\newcommand{\EMSE}{\mbox{EMSE}}

\newcommand{\Pcal}{\mathcal{P}}
\newcommand{\mae}{\gamma_{\mbox{\scriptsize max}}}
\newcommand{\mie}{\gamma_{\mbox{\scriptsize min}}}

\newcommand{\emin}{e_{\min}}
\newcommand{\emax}{e_{\max}}
\DeclareMathOperator{\N}{\mathbb N}

\renewcommand{\eqref}[1]{\textup{\bf \ref{#1}}}
\newcommand{\Sb}{{\bar S}}

\newtheorem{theorem}{Theorem}

\newtheorem{lemma}{Lemma}
\newtheorem{conjecture}{Conjecture}
\newtheorem{simulation}{Simulation}
\newtheorem{remark}{Remark}

\newtheorem{condition}{Condition}

\newtheorem{definition}{Def\,inition}
\newtheorem{example}{Example}
\newenvironment{proof}{%
	\trivlist\item[\hskip \labelsep{\theoremfont Proof.}]%
}{\endtrivlist}

\renewcommand{\theoremfont}{\itshape}

\renewcommand{\eqref}[1]{\textup{\bf \ref{#1}}}

\usepackage[round, sort, numbers]{natbib}

\makeatletter  
\def\BState{\State\hskip-\ALG@thistlm}
\makeatother

\usepackage{comment}
\newcommand{\dropcap}{}
\graphicspath{{figures/}}
\usepackage{enumerate}
\usepackage{enumitem}

\usepackage{comment}
\usepackage{graphicx}
\usepackage{amsmath}
\usepackage{amsfonts}
\usepackage{amssymb}
\usepackage{authblk}
\usepackage[utf8]{inputenc}
\usepackage[english]{babel}
\usepackage{amsmath}
\usepackage{amsfonts}
\usepackage{amssymb}
\usepackage{enumitem}
\usepackage{graphicx}
\usepackage[noend]{algpseudocode}
\usepackage{placeins}
\usepackage{algorithm}

\usepackage{etoolbox}
\usepackage{xr}

\makeatletter  
\def\BState{\State\hskip-\ALG@thistlm}
\makeatother

\makeatletter
\AfterEndEnvironment{algorithm}{\let\@algcomment\relax}
\AtEndEnvironment{algorithm}{\kern2pt\hrule\relax\vskip3pt\@algcomment}
\let\@algcomment\relax
\newcommand\algcomment[1]{\def\@algcomment{\footnotesize#1}}

\renewcommand\fs@ruled{\def\@fs@cfont{\bfseries}\let\@fs@capt\floatc@ruled
	\def\@fs@pre{\hrule height.8pt depth0pt \kern2pt}%
	\def\@fs@post{}%
	\def\@fs@mid{\kern2pt\hrule\kern2pt}%
	\let\@fs@iftopcapt\iftrue}
\makeatother

\allowdisplaybreaks[4]

\author[a]{Sören R. Künzel}
\author[a,b]{Jasjeet S. Sekhon} 
\author[a]{Peter J. Bickel} 
\author[a,c]{Bin Yu}
\affil[a]{{Department of Statistics, University of California, Berkeley, CA 94720}}
\affil[b]{Department of Political Science, University of California, Berkeley, CA 94720}
\affil[c]{Department of Electrical Engineering and Computer Science, University of California, Berkeley, CA 94720}

\title{Meta-learners for Estimating  Heterogeneous Treatment Effects using Machine Learning}

\newcommand{\refb}{\ref}
\begin{document}
	 \defcitealias{gyorfi2006distribution}{(20)}
	 \defcitealias{bickel2015mathematical}{(22)}
	 \defcitealias{breiman2001random}{(34)}
	 \defcitealias{hill2011bayesian}{(5)}
	 \defcitealias{Lewandowski2009}{(36)}
	 \defcitealias{wager2015estimation}{(7)}
	 \defcitealias{heckman1997making}{(35)}
	 \defcitealias{stone1982optimal}{(18)}
	 \defcitealias{Tsybakov2009}{(21)}
	 \defcitealias{KuenzelHTE}{(15)}
	 
	 \maketitle

	 \begin{abstract}
	 	\noindent  
	 	\footnotesize
	 	There is growing interest in estimating and analyzing heterogeneous treatment effects in experimental and observational studies. We describe a number of meta-algorithms that can take advantage of any supervised learning or regression method in machine learning and statistics to estimate the Conditional Average Treatment Effect (CATE) function. Meta-algorithms build on base algorithms---such as Random Forests (RF), Bayesian Additive Regression Trees (BART) or neural networks---to estimate the CATE, a function that the base algorithms are not designed to estimate directly.  We introduce a new meta-algorithm, the X-learner, that is provably efficient when the number of units in one treatment group is much larger than in the other, and can exploit structural properties of the CATE function. For example, if the CATE function is linear and the response functions in treatment and control are Lipschitz continuous, the X-learner can still achieve the parametric rate under regularity conditions. We then introduce versions of the X-learner that use RF and BART as base learners. In extensive simulation studies, the X-learner performs favorably, although none of the meta-learners is uniformly the best.  In two persuasion field experiments from political science, we demonstrate how our new X-learner can be used to target treatment regimes and to shed light on underlying mechanisms.  A software package is provided that implements our methods.

	 \end{abstract}
	 \noindent
\dropcap{W}ith the rise of large data sets containing fine-grained information about humans and their behavior, researchers, businesses, and policymakers are increasingly interested in how treatment effects vary across individuals and contexts. They wish to go beyond the information provided by estimating the Average Treatment Effect (ATE) in randomized experiments and observational studies. Instead, they often seek to estimate the Conditional Average Treatment Effect (CATE) to personalize treatment regimes and to better understand causal mechanisms. We introduce a new estimator called the X-learner, and we characterize it and many other CATE estimators within a unified meta-learner framework.  Their performance is compared using broad simulations, theory, and two data sets from randomized field experiments in political science.


In the first randomized experiment, we estimate the effect of a mailer on voter turnout \cite{GerberGreenLarimer} and, in the second, we measure the effect of door-to-door conversations on prejudice against gender-nonconforming individuals \cite{transphobia}.
In both experiments, the treatment effect is found to be non-constant, and we quantify this heterogeneity by estimating the CATE. We obtain insights into the underlying mechanisms, and the results allow us to better target the treatment.  


To estimate the CATE, we build on regression or supervised learning methods in statistics and machine learning, which are successfully used in a wide range of applications. Specifically, we study meta-algorithms (or meta-learners) for estimating the CATE { in a binary treatment setting. Meta-algorithms} decompose estimating the CATE into several sub-regression problems that can be solved with any regression or supervised learning method.

The most common meta-algorithm for estimating heterogeneous treatment effects takes two steps. {First, it uses so-called base learners to estimate the conditional expectations of the outcomes separately for units under control and those under treatment}. Second, it takes the difference between these estimates. This approach has been analyzed when the base learners are linear regression \cite{foster2013subgroup} or tree-based methods \cite{athey2015machine}. When used with trees, this has been called the \textit{Two-Tree} estimator and we will therefore  refer to the general mechanism of estimating the response functions separately as the \emph{T-learner}, ``T’’ being short for ``\textit{two}.’’ 

Closely related to the T-learner is the idea of estimating the outcome using all of the features and the treatment indicator,  without giving the treatment indicator a special role. The predicted CATE for an individual unit is then the difference between the predicted values when the treatment assignment indicator is changed from control to treatment, with all other features held fixed. This meta-algorithm has been studied with BART \cite{hill2011bayesian, green2012modeling} and regression trees \cite{athey2015machine} as the base learners. We refer to this meta-algorithm as the \emph{S-learner}, since it uses a ``\textit{single}’’ estimator.

Not all methods that aim to capture the heterogeneity of treatment effects fall in the class of meta-algorithms. 
 For example, some researchers analyze heterogeneity by estimating average treatment effects for meaningful subgroups \cite{hansen2009attributing}.  
Another example is causal forests \cite{wager2015estimation}. {Since causal forests are RF-based estimators, they can be compared to meta-learners with RFs in simulation studies. We will see that causal forests and the meta-learners used with RFs perform comparably well, but the meta-learners with other base learners can significantly outperform causal forests.}

The main contribution of this paper is the introduction of a new meta-algorithm: the \emph{X-learner}, which builds on the T-learner and uses each observation in the training set in an ``X”--like shape. Suppose that we could observe the individual treatment effects directly. We could then estimate the CATE function by regressing the difference of individual treatment effects on the covariates. Structural knowledge about the CATE function (e.g., linearity, sparsity, or smoothness) could be taken into account by either picking a particular regression estimator for CATE or using an adaptive estimator that could learn these structural features.
Obviously, we do not observe individual treatment effects because we observe the outcome either under control or under treatment, but never both. The X-learner uses the observed outcomes to estimate the unobserved individual treatment effects. It then estimates the CATE function in a second step as if the individual treatment effects were observed. 

The X-learner has two key advantages over other estimators of the CATE. First, it can provably adapt to structural properties such as the sparsity or smoothness of the CATE. This is particularly useful since the CATE is often zero or approximately linear \cite{Kalla2017, SekhonSATO17}.
Secondly, it is particularly effective when the number of units in one treatment group (usually the control group) is much larger than in the other. This occurs because (control) outcomes and covariates are easy to obtain using data collected by administrative agencies, electronic medical record systems, or online platforms. This is the case in our first data example, where election turnout decisions in the U.S. are recorded by local election administrators for all registered individuals.

The rest of the paper is organized as follows. We start with a formal introduction of the meta-learners and provide intuitions for why we can expect the X-learner to perform well when the CATE is smoother than the response outcome functions and when the sample sizes between treatment and control are unequal.
We then present the results of an extensive simulation study and provide advice for practitioners before we present theoretical results on the convergence rate for different meta-learners. Finally, we examine two field experiments using several meta-algorithms and illustrate how the X-learner can find useful heterogeneity with fewer observations. 

\subsection*{Framework and Definitions} \label{notation}
We employ the Neyman--Rubin potential outcome framework \citep{rubin1974estimating, splawa1990application}, and assume a superpopulation or distribution $\Pcal$ from which a realization of $N$ independent random variables is given as the training data. That is, $(Y_i(0), Y_i(1), X_i, W_i) \sim \Pcal$, where $X_i \in \mathbb{R}^d$ is a $d$-dimensional covariate or feature vector, $W_i \in \{0, 1\}$ is the treatment assignment indicator (to be defined precisely later), $Y_i(0) \in \mathbb{R}$ is the potential outcome of unit $i$ when $i$ is assigned to the control group, and $Y_i(1)$ is the potential outcome when $i$ is assigned to the treatment group.
With this definition, the Average Treatment Effect is defined as
$$
\mbox{ATE} := \E[Y(1) - Y(0)].
$$
It is also useful to define the response under control, $\mu_0,$ and the response under treatment, $\mu_1,$ as
$$
\mu_0(x) := \E[Y(0) | X = x] ~~~~\mbox{and} ~~~~ \mu_1(x) := \E[Y(1) | X = x].
$$
Furthermore, we use the following representation of $\Pcal$:
\begin{equation}\label{model:basic}
\begin{aligned}
X &\sim \Lambda,\\
W&\sim \mbox{Bern}(e(X)), \\
Y(0) &= \mu_0(X) + \varepsilon(0), \\
Y(1) &= \mu_1(X) + \varepsilon(1), 
\end{aligned} 
\end{equation}
where $\Lambda$ is the marginal distribution of $X$, $\varepsilon(0)$ and $\varepsilon(1)$ are zero-mean random variables and independent of $X$ and $W$, and $e(x) = \P(W=1 | X= x)$ is the propensity score.

The fundamental problem of causal inference is that for each unit in the training data set, we observe either the potential outcome under control ($W_i = 0$), or the potential outcome under treatment ($W_i = 1$) but never both. Hence we denote the observed data as
$$
\D = (\Yobs_i, X_i, W_i)_{1 \le i \le N},
$$
with $\Yobs_i = Y_i(W_i)$. Note that the distribution of $\D$ is specified by $\Pcal$.
To avoid the problem that with a small but non-zero probability all units are under control or under treatment, we will analyze the behavior of different estimators conditional on the number of treated units.
That is, for a fixed $n$ with $0 < n < N$, we condition on the event that
$$
\sum_{i = 1}^N W_i = n.
$$
This will enable us to state the performance of an estimator in terms of the number of treated units $n$ and the number of control units $m = N - n$.

For a new unit $i$ with covariate vector $x_i$, in order to decide whether to give the unit the treatment, we wish to estimate the Individual Treatment Effect (ITE) of unit $i$, $D_i$, which is defined as 
$$ 
D_i := Y_i(1) - Y_i(0).
$$
However, we do not observe $D_i$ for any unit, and $D_i$ is not identifiable without strong additional assumptions { in the sense that one} can construct data-generating processes with the same distribution of the observed data, but a different $D_i$ (Example \ref{Exa:ITEnotIdfbl}).
Instead, we will estimate the CATE function, which is defined as
$$
\tau(x) := \E\Big[ D \Big| X = x \Big] = \E\Big[Y(1)- Y(0)\Big|X = x\Big],
$$
and we note that the best estimator for the CATE is also the best estimator for the ITE in terms of the MSE. 
To see that, let $\hat \tau_i$ be an estimator for $D_i$ and decompose the MSE at $x_i$ 
\begin{equation}\label{EMSE}
\begin{aligned}
\E &\left[(D_i - \hat \tau_i )^2 | X_i = x_i\right] \\
= &
\E \left[(D_i - \tau(x_i) )^2 | X_i = x_i\right] + 
\E \left[( \tau(x_i)  - \hat \tau_i)^2 \right]. 
\end{aligned}
\end{equation}
Since we cannot influence the first term in the last expression, the estimator that minimizes the MSE for the ITE of $i$ also minimizes the MSE for the CATE at $x_i$.

In this paper, we are interested in estimators with a small Expected Mean Squared Error (EMSE) for estimating the CATE,
\begin{equation*}
\EMSE(\Pcal, \hat \tau) = \E\left[(\tau(\X) - \hat \tau(\X))^2\right].
\end{equation*}
The expectation is here taken over $\hat \tau$ and $\X \sim \Lambda$, where $\X$ is independent of $\hat \tau$.

To aid our ability to estimate $\tau$, we need to assume that there are no hidden confounders \cite{rosenbaum1983central}:
\begin{condition}~
	\label{equation:conditionalUnconf}
	\begin{equation*}
	\left(\varepsilon(0), \varepsilon(1) \right) \perp W | X.
	\end{equation*}
\end{condition}
This assumption is, however, not sufficient to identify the CATE. One additional assumption that is often made to obtain identifiability of the CATE in the support of $X$ is to assume that the propensity score is bounded away from 0 and 1:
\begin{condition}
	\label{cond:overlab}
	There exists $\emin$ and $\emax$ such that for all $x$ in the support of $X$,
	\begin{equation*}
	0 < \emin < e(x)  < \emax < 1.
	\end{equation*}
\end{condition}


	\section*{Meta-algorithms}\label{section:Xlearner}    

In this section, we formally define a meta-algorithm (or meta-learner) for the CATE as the result of combining supervised learning or regression estimators (i.e., base learners) in a specific manner while allowing the base learners to take any form. Meta-algorithms thus have the flexibility to appropriately leverage different sources of prior information in separate sub-problems of the CATE estimation problem: they can be chosen to fit a particular type of data, and they can directly take advantage of existing data analysis pipelines. 

We first review both S- and T-learners, and we then propose the X-learner, which is a new meta-algorithm that can take advantage of unbalanced designs (i.e., the control or the treated group is much larger than the other group) and existing structures of the CATE (e.g., smoothness or sparsity).  Obviously, flexibility is a gain only if the base learners in the meta-algorithm match the features of the data and the underlying model well.

The T-learner takes two steps. First, the
control response function,
\begin{align*}
\mu_0(x) &= \E[Y(0)|X=x],
\end{align*}
is estimated by a base learner, which could be any supervised learning or regression estimator using the observations in the control group, $\{(X_i, Y_i)\}_{W_i=0}$. We denote the estimated function as $\hat \mu_0$.
Second, we estimate the treatment response function, 
\begin{align*}
\mu_1(x) &= \E[Y(1)|X=x],
\end{align*}
with a potentially different base learner, using the treated observations and denoting the estimator by $\hat \mu_1$. A T-learner is then obtained as 
\begin{equation} \label{Tlearner}
\hat\tau_T(x) = \hat \mu_1(x) - \hat \mu_0(x).
\end{equation}
Pseudocode for this T-learner can be found in  Algorithm \ref{algo:Tlearner}.

In the S-learner, the treatment indicator is included as a feature similar to all the other features without the indicator being given any special role. We thus
estimate the combined response function,
\begin{align*}
\mu(x,w) &:= \E[Y^{obs}|X=x, W= w],
\end{align*}
using any base learner (supervised machine learning or regression algorithm) on the entire data set. We denote the estimator as $\hat \mu$.
The CATE estimator is then given by
\begin{equation} \label{Slearner}
\hat\tau_S(x) = \hat \mu(x, 1) - \hat \mu(x, 0),
\end{equation}
and pseudocode is provided in Algorithm \ref{algo:Slearner}.

There are other meta-algorithms in the literature, but we do not discuss them here in detail because of limited space. For example, one may transform the outcomes so that any regression method can estimate the CATE directly (Algorithm \ref{algo:Flearner}) \cite{athey2015machine, tian2014simple, Powers2017}. In our simulations, this algorithm performs poorly, and we do not discuss it further, but it may do well in other settings.

\subsection*{X-learner}

We propose the X-learner, and provide an illustrative example to highlight its motivations. The basic idea of the X-learner can be described in three stages:
\begin{enumerate}
	\item
	Estimate the response functions
	\begin{align}
	\mu_0(x) &= \E[Y(0)|X=x], \mbox{ and}\\
	\mu_1(x) &= \E[Y(1)|X=x],
	\end{align}
	using any supervised learning or regression algorithm and denote the estimated
	functions $\hat \mu_0$ and $\hat \mu_1$. The algorithms used are referred to as the base
	learners for the first stage.
	
	\item
	Impute the treatment effects for the individuals in the treated group, based on the control outcome estimator, and the treatment effects for the individuals in the control group, based on the treatment outcome estimator, that is,
	\begin{align}
	\Dt^1_{i} &:=         Y^1_{i}                         -     \hat \mu_0(X^{1}_{i}), \mbox{ and}\\
	\Dt^0_{i} &:=     \hat \mu_1(X^{0}_{i})       -     Y^{0}_{i},
	\end{align}
	and call these the imputed treatment effects.
	{ Note that if $\hat \mu_0 = \mu_0$ and $\hat \mu_1 = \mu_1$, then $\tau(x)= \E[\Dt^1|X=x] = \E[\Dt^0|X=x]$.}
	
	{ Employ any supervised learning or regression method(s)} to estimate $\tau(x)$ in two ways: using the imputed treatment effects as the response variable in the treatment group to obtain $\hat \tau_1(x)$, and similarly in the control group to obtain $\hat \tau_0(x)$.
	Call the supervised learning or regression algorithms base learners of the second stage. 
	
	\item 
	Define the CATE estimate by a weighted average of the two estimates
	in Stage 2:
	\begin{equation} \label{equ:3rdStep}
	\hat \tau(x)     = g(x) \hat \tau_0(x) + (1 - g(x))    \hat \tau_1(x),     
	\end{equation}
	where $g \in [0,1]$ is a weight function.
	
\end{enumerate}
See Algorithm \ref{aglo:Xlearner} for pseudocode.

\begin{remark} 
	$\hat \tau_0$ and $\hat \tau_1$ are both estimators for $\tau$, while $g$ is chosen to combine these estimators to one improved estimator $\hat \tau$.
	
	Based on our experience, we observe that it is good to use an estimate of the propensity score for $g$, so that $g = \hat e$, but it also makes sense to choose $g = 1$ or $0$, if the number of treated units is very large or small compared to the number of control units.
	
	For some estimators, it might even be possible to estimate the covariance matrix of $\hat \tau_1$ and $\hat \tau_0$. One may then wish to choose $g$ to minimize the variance of $\hat \tau$. 
\end{remark}

\subsection*{Intuition behind the meta-learners}\label{section:XlearnerMotivation}
%
%
%
%

\begin{figure}
	\centering
	\includegraphics[width=1\linewidth]{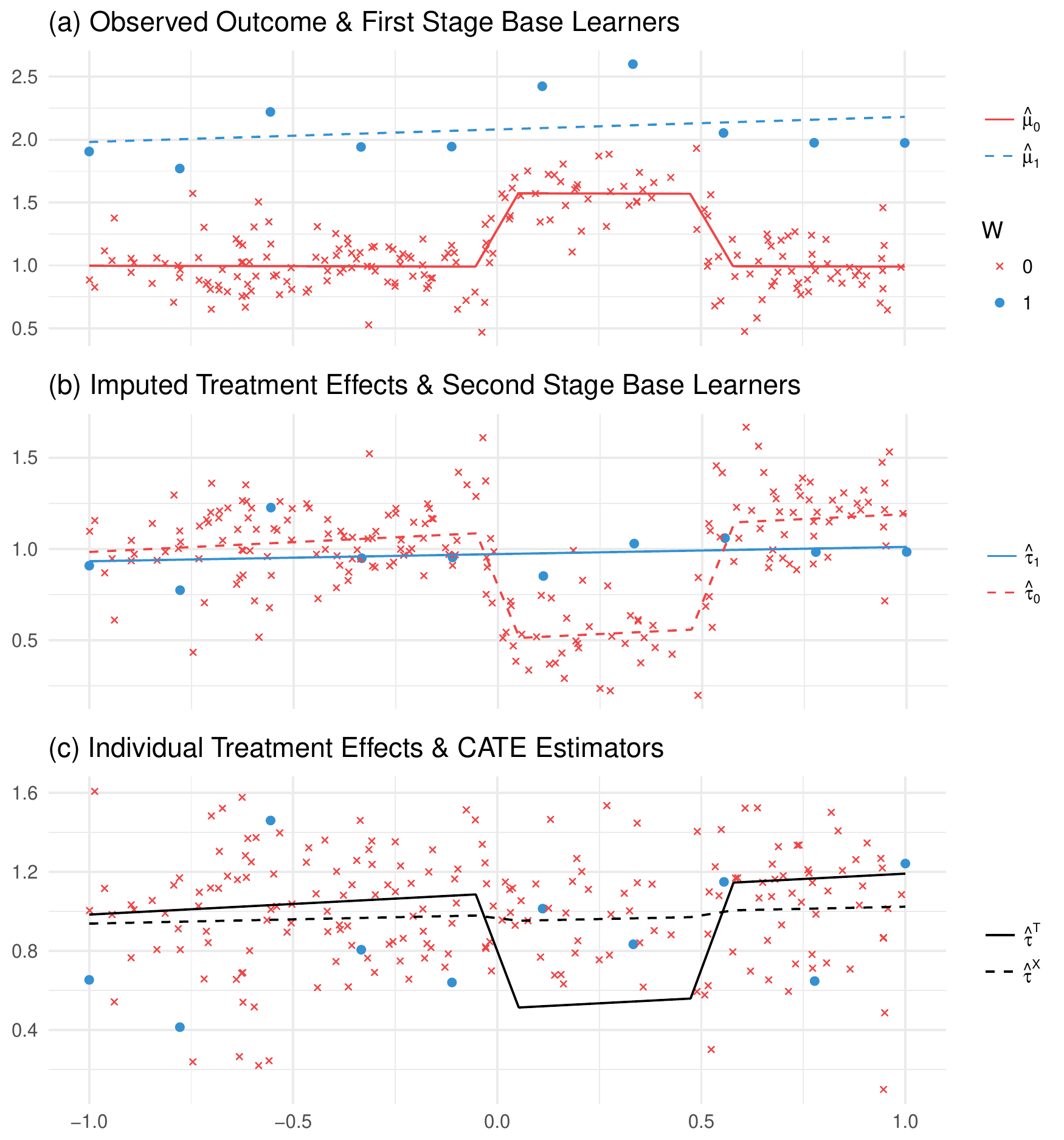}
	\caption{Intuition behind the X-learner with an unbalanced design.} \label{XIntuition}
\end{figure}

The X-learner can use information from the control group to derive better estimators for the treatment group and vice versa.  We will illustrate this using a simple example. Suppose that we want to study a treatment, and we are interested in estimating the CATE as a function of one covariate $x$. We observe, however, very few  units in the treatment group and many units in the control group. This situation often arises with the growth of administrative and online data sources: data on control units is often far more plentiful than data on treated units. Figure \ref{XIntuition}(a) shows the outcome for units in the treatment group (circles) and the outcome of unit in the untreated group (crosses). In this example, the CATE is constant and equal to one. 

For the moment, let us look only at the treated outcome. When we estimate $\mu_1(x) = \E[Y(1) | X= x]$, we must be careful not to overfit the data since we observe only 10 data points. We might decide to use a linear model, $\hat \mu_1(x)$ (dashed line), to estimate $\mu_1$. 
For the control group, we notice that observations with $x \in [0, 0.5]$ seem to be different, and we end up modeling $\mu_0(x) = \E[Y(0)|X=x]$ with a piecewise linear function with jumps at 0 and 0.5 (solid line).  This is a relatively complex function, but we are not worried about overfitting since we observe many data points.

The \emph{T-learner} would now use estimator $\hat \tau_T(x) = \hat \mu_1(x) - \hat \mu_0(x)$ (see Figure \ref{XIntuition}(c), solid line), which is a relatively complicated function with jumps at 0 and 0.5, while the true $\tau(x)$ is a constant.
This is, however, problematic because we are estimating a complex CATE function, based on ten observations in the treated group.

When choosing an estimator for the treatment group, we correctly avoided overfitting, and we found a good estimator for the treatment response function and, as a result, we chose a relatively complex estimator for the CATE, namely, the quantity of interest.
We could have selected a piecewise linear function with jumps at 0 and 0.5, but this, of course, would have been unreasonable when looking only at the treated group. If, however, we were to also take the control group into account, this function would be a natural choice.
In other words, we should change our objective for $\hat \mu_1$ and $\hat \mu_0$. We want to estimate $\hat \mu_1$ and $\hat \mu_0$ in such a way that their difference is a good estimator for $\tau$.

The \emph{X-learner} enables us to do exactly that. It allows us to use structural information about the CATE to make efficient use of an unbalanced design.  The first stage of the X-learner is the same as the first stage of the T-learner, but in its second stage, the estimator for the controls is subtracted from the observed treated outcomes and similarly the observed control outcomes are subtracted from estimated treatment outcomes to obtain the imputed treatment effects, 
\begin{equation*}
	\begin{aligned}
		\Dt^1_{i} &:=         Y^{1}_{i}                       -     \hat \mu_0(X^{1}_{i}), \\
		\Dt^0_{i} &:=     \hat \mu_1(X^{0}_{i})       -     Y^{0}_{i}.
	\end{aligned}
\end{equation*}
Here we use the notation that $Y^{0}_{i}$ and $Y^{1}_{i}$ are the $i$th observed outcome of the control and the treated group, respectively. $X^{1}_{i}$, $X^{0}_{i}$ are the corresponding feature vectors. 
Figure \ref{XIntuition}(b) shows the imputed treatment effects, $\Dt$.
By choosing a simple---here linear---function to estimate $\tau_1(x) = \E[\Dt^1 | X^1 = x]$ we effectively estimate a model for $\mu_1(x) = \E[Y^1 | X^1 = x]$, which has a similar shape to $\hat \mu_0$.
{By choosing a relatively poor model for $\mu_1(x)$, $\Dt^0$ (the red crosses in Figure \ref{XIntuition}(b)) are relatively far away from $\tau(x)$, which is constant and equal to 1.}
The model for $\tau_0(x) = \E[\Dt^0|X=x]$ will thus be relatively poor. However, our final estimator combines these two estimators according to 
\begin{equation*} 
\hat \tau(x)     = g(x) \hat \tau_0(x) + (1 - g(x))    \hat \tau_1(x).
\end{equation*}
If we choose $g(x) = \hat e(x)$, an estimator for the propensity score,
$\hat \tau$ will be very similar to $\hat \tau_1(x)$, since we
have many more observations in the control group; i.e., $\hat e(x)$ is small. 
Figure \ref{XIntuition}(c) shows the T-learner and the X-learner. 

It is difficult to assess the general behavior of the S-learner in this example because we must choose a base learner. For example, when we use RF as the base learner for this data set, the S-learner's first split is on the treatment indicator in 97.5\% of all trees in our simulations because the treatment assignment is very predictive of the observed outcome, $\Yobs$ (see also Figure \ref{fig:snosplit}). From there on, the S-learner and the T-learner are the same, and we observe them to perform similarly poorly in this example.


	\section*{Simulation Results}
In this section, we conduct a broad simulation study to compare the different meta-learners. In particular, we summarize our findings and provide general remarks on the strengths and weaknesses of the S-, T-, and X-learners, while deferring the details to the Supporting Information (SI).  The simulations are key to providing an understanding of the performance of the methods we consider for model classes that are not covered by our theoretical results.

Our simulation study is designed to consider a range of situations. We include conditions under which the S-learner or the T-learner is likely to perform the best, as well as simulation setups proposed by previous researchers \cite{wager2015estimation}. We consider cases where the treatment effect is zero for all units (and so pooling the treatment and control groups would be beneficial), and cases where the treatment and control response functions are completely different (and so pooling would be harmful). We consider cases with and without confounding,\footnote{{Confounding here refers to the existence of an unobserved covariate that influences both the treatment variable, $W$, and at least one of the portential outcomes $Y(0), Y(1)$.}} and cases with equal and unequal sample sizes across treatment conditions. All simulations discussed in this section are based on synthetic data. For details, please see Section \ref{section:simulation}. We provide additional simulations based on actual data when we discuss our applications. 

We compare the S-, T-, and X-learners with RF and BART as base learners. We implement a version of RF for which the tree structure is independent of the leaf prediction given the observed features, the so-called honest RF in an R package called \texttt{hte} \cite{KuenzelHTE}. This version of RF is particularly accessible from a theoretical point of view, it performs well in noisy settings, and it is better suited for inference \cite{scornet2015consistency, wager2015estimation}. For BART, our software uses the \texttt{dbarts} \cite{Hughbart} implementation for the base learner.

Comparing different base learners enables us to demonstrate two things. On the one hand, it shows that the conclusions we draw about the S-, T-, and X-learner are not specific to a particular base learner and, on the other hand, it demonstrates that the choice of base learners can make a large difference in prediction accuracy. The latter is an important advantage of meta-learners since subject knowledge can be used to choose base learners that perform well. For example, in Simulations \ref{sim:complexlin} and \ref{sim:globallin} the response functions are globally linear, and we observe that estimators that act globally such as BART have a significant advantage in these situations or when the data set is small. If, however, there is no global structure or when the data set is large, then more local estimators such as RF seem to have an advantage (Simulations \ref{sim:complexnonlin} and \ref{sim:piecwlin}).

We observe that the choice of meta-learner can make a large difference, and for each meta-learner there exist cases where it is the best-performing estimator. 

The S-learner treats the treatment indicator like any other predictor. For some base learners such as $k$-nearest neighbors it is not a sensible estimator, but for others it can perform well. { Since the treatment indicator is given no special role, algorithms such as the lasso and RFs can completely ignore the treatment assignment by not choosing/splitting on it.} This is beneficial if the CATE is in many places 0 (Simulations \ref{sim:globallin} and \ref{sim:piecwlin}), but---as we will see in our second data example---the S-learner can be biased toward 0.

The T-learner, on the other hand, does not combine the treated and control groups. This can be a disadvantage when the treatment effect is simple because by not pooling the data, it is more difficult for the T-learner to mimic a behavior that appears in both the control and treatment response functions (e.g., Simulation \ref{sim:globallin}).  If, however, the treatment effect is very complicated, and there are no common trends in $\mu_0$ and $\mu_1$, then the T-learner performs especially well (Simulations \ref{sim:complexlin} and \ref{sim:complexnonlin}).

The X-learner performs particularly well when there are structural assumptions on the CATE or when one of the treatment groups is much larger than the other (Simulation \ref{sim:unbalanced} and \ref{sim:complexnonlin}). In the case where the CATE is 0, it usually does not perform as well as the S-learner, but it is significantly better than the T-learner (Simulations \ref{sim:globallin}, \ref{sim:piecwlin}, and \ref{sim:betacon}), and in the case of a very complex CATE, it performs better than the S-learner and it often outperforms even the T-learner (Simulations \ref{sim:complexlin} and \ref{sim:complexnonlin}).  These simulation results lead us to the conclusion that unless one has a strong belief that the CATE is mostly 0, then, as a rule of thumb, one should use the X-learner with BART for small data sets and RF for bigger ones. In the sequel, we will further support these claims with additional theoretical results and empirical evidence from real data and data-inspired simulations.


	\section*{Comparison of Convergence Rates} 
\label{section:ConvergenceRate}

In this section, we provide conditions under which the X-learner can be proven to outperform the T-learner in terms of pointwise estimation rate.
These results can be viewed as attempts to rigorously formulate intuitions regarding when the X-learner is desirable. They corroborate our intuition that the X-learner outperforms the T-learner when one group is much larger than the other group or when the CATE function has a simpler form than those of the underlying response functions themselves.

Let us start by reviewing some of the basic results in the field of minimax nonparametric regression estimation \cite{stone1982optimal, birge1983approximation, gyorfi2006distribution, bickel2015mathematical}. In the standard regression problem, one observes $N$ independent and identically distributed tuples $(X_i, Y_i)_i \in  \R^{d\times N} \times \R^N$ generated from some distribution $\mathcal P$ and one is interested in estimating the conditional expectation of $Y$ given some feature vector $x$, $\mu(x) = \E[Y|X=x]$. 
The error of an estimator $\hat \mu_N$ can be evaluated by the Expected Mean Squared Error (EMSE),
$$
\EMSE(\mathcal{P}, \hat \mu_N) 
=
\E[(\hat \mu_N(\X) - \mu(\X) )^2].
$$
For a fixed $\Pcal$, there are always estimators which have a very small EMSE. For example, choosing $\hat \mu_N \equiv \mu$ would have no error. 
However, $\Pcal$ and thus $\mu$ is unknown. 
Instead, one usually wants to find an estimator which achieves a small EMSE for a relevant set of distributions (such a set is relevant if it captures domain knowledge or prior information of the problem). 
To make this problem feasible, a typical approach is the minimax approach where one analyzes the worst performance of an estimator over a class or family, $F$, of distributions  \cite{Tsybakov2009}. The goal is to find an estimator which has a small EMSE for all distributions in this family. 
For example, if $F_0$ is the family of distributions $\mathcal{P}$ such that $X \sim \mbox{Unif}[0,1]$, $Y = \beta X + \varepsilon$, $\varepsilon \sim N(0,1)$, and $\beta \in \R$, then it is well known that the OLS estimator achieves the optimal parametric rate. That is, there exists a constant $C \in \R$ such that for all $\mathcal{P} \in F_0$,
$$
\EMSE(\mathcal{P}, \hat \mu_N^{\mbox{\tiny{OLS}}})  \le C N^{-1}.
$$
If, however, $F_1$ is the family of all distributions $\mathcal{P}$ such that $X \sim \mbox{Unif}[0,1]$, $Y \sim \mu(X) + \varepsilon$ and $\mu$ is a Lipschitz continuous function with bounded Lipschitz constant, then there exists no estimator that achieves the parametric rate uniformly for all possible distributions in $F_1$. 
To be precise, we can at most expect to find an estimator that achieves a rate of $N^{-2/3}$ and there exists a constant $C'$, such that
$$
\liminf_{N \rightarrow \infty} \inf_{\hat \mu_N} \sup_{\Pcal \in F_1}
\frac{\EMSE(\Pcal, \hat \mu_N)}{N^{-2/3}}
> C' > 0.
$$
Estimators such as the Nadaraya--Watson and $k$-nearest neighbors can achieve this optimal rate \citep{bickel2015mathematical,gyorfi2006distribution}. 

Crucially, the fastest rate of convergence that holds uniformly for a family $F$ is a property of the family to which the underlying data-generating distribution belongs. 
It will be useful for us to define sets of families for which particular rates are achieved. 
\begin{definition}[Families with bounded minimax rate] \label{def:muclasses} 
	For $a \in (0,1]$, we define $S(a)$ to be the set of all families, $F$, with a minimax rate of at most $N^{-a}$. 
\end{definition}
	Note that for any family $F \in S(a)$ there exists an estimator $\hat \mu$ and a constant $C$ such that for all $N \ge 1$,
	$$
		\sup_{\Pcal \in F}
		\EMSE(\Pcal, \hat \mu_N) 
		\le 
		C N^{-a}.	
	$$
From the examples above, it is clear that $F_0 \in S(1)$ and $F_1 \in S(2/3)$.

Even though the minimax rate of the EMSE is not very practical since one rarely knows that the true data-generating process is in some reasonable family of distributions,  it is nevertheless one of the very few useful theoretical tools to compare different nonparametric estimators. If for a big class of distributions, the worst EMSE of an estimator $\hat \mu^A$ is smaller than the worst EMSE of $\hat \mu^B$, then one might prefer estimator $\hat \mu^A$ over estimator $\hat \mu^B$. Furthermore, if the estimator of choice does not have a small error for a family that we believe based on domain information could be relevant in practice, then we might expect $\hat \mu$ to have a large EMSE in real data.

\subsubsection*{Implication for CATE estimation}
Let us now employ the minimax approach to the problem of estimating the CATE. 
Recall that we assume a super--population of random variables $(Y(0), Y(1), X, W)$ according to some distribution $\Pcal$. We observe $n$ treated and $m$ control units from this super-population,
and our goal is to find an estimator $\hat \tau_{mn}$ which has a small EMSE,
$$
\EMSE(\Pcal, \hat \tau_{mn})
=
\E[(\tau(\X) - \hat \tau_{mn}(\X))^2].
$$
Similar to the regression case, we can study the performance of estimators when $\Pcal$ lies in some family of distributions. In the following we will introduce families for which estimators based on the X-learner achieve provably a smaller EMSE than estimators based on the T-learner. 

Similar to Definition \ref{def:muclasses}, we define sets of families of super-populations. 
\begin{definition}[Superpopulations with given rates] \label{def:Samat}
	Recall the general characterization of a superpopulation given in \refb{model:basic}.
	For $a_\mu, a_\tau \in (0,1]$, we define $S(a_\mu, a_\tau)$ to be the set of all families of distributions $\Pcal$ of $(Y(0), Y(1), X, W)$ such that
	\begin{enumerate}
		\item ignorability holds, 
		\item the distribution of $(X, Y(0))$ given $W = 0$ is in a class $F_0 \in S(a_\mu)$,
		\label{Assumption_res_0}
		\item the distribution of $(X, Y(1))$ given $W = 1$ is in a class $F_1 \in S(a_\mu)$, and
		\label{Assumption_res_1}
		\item the distribution of $(X, \mu_1(X) - Y(0))$ given $W = 0$ is in a class $F_{\tau 0} \in S(a_\tau)$. 
		\label{Assumption_tau_0}
		\item the distribution of $(X, Y(1) - \mu_0(X))$ given $W = 1$ is in a class $F_{\tau 1} \in S(a_\tau)$.         
		\label{Assumption_tau_1}
	\end{enumerate}
\end{definition}
A simple example of a family in $S(2/3, 1)$, would be the set of distributions $\Pcal$ for which $X \sim \mbox{Unif}[0,1]$, $W \sim \mbox{Bern}(1/2)$, $\mu_0$ is any Lipschitz continuous function, $\tau$ is linear, and $\varepsilon(0), \varepsilon(1) \sim N(0,1)$.

The difference between the T and X learner is that the T-learner estimates the response functions separately, and does not benefit from the possible smoothness of the CATE. Hence, we can conclude the following theorem.
\begin{theorem}[Minimax rates of the T-learner]
	For $\Pcal \in S(a_\mu, a_\tau)$, there exist base learners to be used in the T-learner so that the corresponding T-learner estimates the CATE at a rate of 
	$$\mathcal{O}(m^{-a_\mu} + n^{-a_\mu}),$$ 
	but, in general, we cannot expect it to be any faster. 
\end{theorem}

The X-learner, on the other hand, can be seen as a weighted average of the two estimators, $\hat \tau_0$ and $\hat \tau_1$ (Eq.~\refb{equ:3rdStep}). Take for the moment, $\hat \tau_1$. It consists of an estimator for the outcome under control which achieves a rate of at most $a_\mu$, and an estimator for the imputed treatment effects which should intuitively achieve a rate of at most $a_\tau$. We therefore expect the following conjecture.
\begin{conjecture}[Minimax rates of the X-learner] \label{theo:Xrate}
	Under some conditions on $\mathcal{P} \in S(a_\mu, a_\tau)$, there exist base learners such that
	$\hat{\tau}_0$ and $\hat{\tau}_1$ in the X-learner achieve the rates,  
	\begin{equation*}
	\mathcal{O}(m^{-a_\tau} + n^{-a_\mu}) 
	\quad
	\mbox{and}
	\quad
	\mathcal{O}(m^{-a_\mu} + n^{-a_{\tau}}),
	\end{equation*}
	respectively.
\end{conjecture}

It turns out to be mathematically very challenging to give a satisfying statement of the extra conditions needed on $\mathcal{P}$.
We will therefore discuss two cases where we do not need any extra conditions on $\mathcal{P}$, and we emphasize that we believe the conjecture to hold in much greater generality. 
In the first case (Theorem \ref{theorem:lineartau}), we discuss all families of distributions in $S(a_\mu, a_\tau )$ where the CATE is linear. This implies that $a_\tau = 1$, and we achieve the parametric rate in $n$. This is in particular important when the number of control units, $m$, is large. 
In Section \ref{section:balanced}, we discuss the other extreme where we don't have any assumptions on the CATE. In this case, there is nothing to be inferred from the control units about the treated units and vice versa. Consequently, the T-learner is in some sense the best strategy and achieves the minimax optimal rate of $\mathcal{O}(m^{-a_\mu} + n^{-a_\mu})$ and we show that, for example, under Lipschitz continuity of the response functions, the X-learner will achieve the same rate and is therefore minimax optimal as well. 

We also conduct an extensive simulation study (Section \ref{section:simulation}) in which we compare the different meta-learners combined with Random Forests and BART for many different situations. We find that neither learner is uniformly the best, but the X-learner is never the worst, and it performs particularly well, when the group sizes are very unbalanced, or the CATE function satisfies some regularity conditions. 
\subsection*{Smoothness conditions for the CATE} \label{section:unbalanced}
Even though it is theoretically possible that $a_\tau$ is similar to $a_\mu$, our experience with real data suggests that it is often larger (or the treatment effect is \textit{simpler} to estimate than the potential outcomes). 
Let us intuitively understand the difference between the T- and X-learners for a class $F \in S(a_\mu, a_\tau)$ with $a_\tau > a_\mu$. 
The T-learner splits the problem of estimating the CATE into the two subproblems of estimating $\mu_0$ and $\mu_1$ separately. By appropriately choosing the base learners, we can expect to achieve the minimax optimal rates of $a_\mu$,
\begin{equation}\label{rates:Tlearner}
	\begin{aligned} 
	\sup_{\Pcal_0 \in F_0} \EMSE(\Pcal_0,  \hat \mu_0^m) 
	\le C m^{-a_\mu},
	&&
	\mbox{and}
	\\
	\sup_{\Pcal_1 \in F_1} \EMSE(\Pcal_1,  \hat \mu_1^n) 
	\le C n^{-a_\mu},
	\end{aligned}
\end{equation}
where $C$ is some constant.
Those rates translate immediately to rates for estimating $\tau$, since
\begin{align*} \label{rates:Tlearnertau}
\sup_{\Pcal \in F} &\EMSE(\Pcal,  \hat \tau^T_{nm}) 
\\&\le
2
\sup_{\Pcal_0 \in F_0} \EMSE(\Pcal_0,  \hat \mu_0^m) 
+~
2
\sup_{\Pcal_1 \in F_1} \EMSE(\Pcal_1,  \hat \mu_1^n) \notag
\\&=
2C \left( m^{-a_\mu} +n^{-a_\mu} \right).
\end{align*}
In general, we cannot expect to do better than this, when using an estimation strategy that falls in the class of T-learners, because the subproblems in Equation~\refb{rates:Tlearner} are treated completely independently and there is nothing to be learned from the treatment group about the control group and vice versa.

In some cases, we observe that the number of control units is much larger than the number of treated units, $ m \gg n$. This happens for example if we test a new treatment and we have a large number of previous (untreated) observations that can be used as the control group. In that case, the bound on the EMSE of the T-learner for the CATE will be dominated by the regression problem for the treated group,
\begin{equation} \label{rates:Tlearnertau+bigm}
\sup_{\Pcal \in F} \EMSE(\Pcal,  \hat \tau^T_{nm}) 
= 
\sup_{\Pcal_1 \in F_1} \EMSE(\Pcal_1,  \hat \mu_1^n)
\le
C n^{-a_\mu}.
\end{equation}
This is an improvement, but it still does not lead to the fast rate, $a_\tau$.
The X-learner, however, can achieve the fast rate $a_\tau$. An expansion of the EMSE into two squared error terms and also a cross term involving biases can be used to show that the T-learner cannot achieve this fast rate in general in the unbalanced case of $m>>n$.
To see the faster rate for the X-learner, recall that the number of control units is assumed so large that $\mu_0$ can be predicted almost perfectly and choose the weighing function $g$ equal to 0 in Equation \refb{equ:3rdStep}.
It follows that the error of the first stage of the X-learner is negligible and the imputed treatment effects for the treated group satisfy $D^1_i = \tau(X_i(1)) + \varepsilon_i$. 
Per Assumption \ref{Assumption_tau_1} in Definition \ref{def:Samat}, $\E [D^1 |X =x]$ can now be estimated using an estimator achieving the desired rate of $a_\tau$,
$$
\sup_{\Pcal \in F} \EMSE(\Pcal,  \hat \tau^X_{nm}) 
\le
C n^{-a_\tau}.
$$
This is a substantial improvement over \refb{rates:Tlearnertau+bigm} when $a_\tau > a_\mu$ and intuitively demonstrates that, in contrast to the T-learner, the X-learner can exploit structural assumptions on the treatment effect. 
However, even for large $m$, we cannot expect to perfectly estimate $\mu_0$. The following theorem deals carefully with this estimation error when $\tau$ is linear, but the response functions can be estimated at any nonparametric rate.

\begin{theorem} \label{theorem:lineartau}
	Assume we observe $m$ control units and $n$ treated units from some super population of independent and identically distributed observations $(Y(0), Y(1), X, W)$ coming from a distribution $\Pcal$ given in \refb{model:basic} that satisfies the following assumptions:
	\begin{enumerate}[label=A\arabic*]
		\item The error terms $\varepsilon_i$ are independent given $X$, with $\E[\varepsilon_i|X = x] = 0$ and $\var[\varepsilon_i|X = x] \le \sigma^2$.
		\item $X$ has finite second moments, $$\E[\|X\|_2^2] \le C_X.$$
		\item Ignorability holds.
		\item There exists an estimator $\hat \mu^m_0$ and $a >0$ with 
		$$
		\EMSE(\Pcal,  \hat \mu^m_0)   = \E[(\mu_0(X) - \hat \mu^m_0(X))^2] \le C_0 m^{-a}.
		$$ \label{condition:controlrate}
		\item The treatment effect is parametrically linear, $\tau(x) = x^T \beta$, with $\beta \in \R^d$.
		\item \label{theorem:lineartau:goodEV} The eigenvalues of the sample covariance matrix of $X^1$ are well conditioned, in the sense that there exists an $n_0$, such that 
		\begin{align}
		\sup_{n >n_0}
		\mie^{-1}(\hat \Sigma_n) \le C_\Sigma.
		\end{align}
	\end{enumerate}
	Then the X-learner with $\hat \mu_0^m$ in the first stage, OLS in the second stage and weigh function $g\equiv 0$ has the following upper bound on its EMSE: for all $n> n_0$,
	$$
	\EMSE(\Pcal,  \hat \tau^{mn})
	=
	\E\left[\left\|\tau(X) - \hat  \tau^{mn}(X)\right\|^2\right] \le C (m^{-a} + n^{-1})
	$$
	with $C = \max(C_0, \sigma^2 d)  C_X C_\Sigma$.
	In particular, if there are a lot of control units, such that $m \geq c_3 n^{1/a}$, then the X-learner achieves the parametric rate in $n$,
	$$
	\EMSE(\Pcal,  \hat \tau^{mn}) \le (1+c_3)Cn^{-1}.
	$$
\end{theorem}

It is symmetric that similar results hold if $n$ (the size of the treatment group) is much larger than $m$ (the size of the control group). Furthermore, we note that an equivalent statement also holds for the pointwise MSE (Theorem \ref{theorem:unbalanced_ptwise}).


	\section*{Applications}
\label{sec:applications}

In this section, we consider two data examples. In the first example,
we consider a large Get-Out-The-Vote (GOTV) experiment that explored
if social pressure can be used to increase voter turnout in elections
in the United States \citep{GerberGreenLarimer}. In the second
example, we consider an experiment that explored if door-to-door
canvassing can be used to durably reduce transphobia in Miami
\citep{transphobia}. In both examples, the original authors failed to
find evidence of heterogeneous treatment effects when using simple linear
models without basis expansion, and subsequent researchers and policy
makers have been acutely interested in treatment effect heterogeneity that
could be used to better target the interventions. 
We use our honest random forest implementation \cite{KuenzelHTE} because of the importance of obtaining useful confidence intervals in these applications. Confidence intervals are obtained using a bootstrap procedure (Algorithm \ref{alg:computeCI}). 
 We have evaluated several bootstrap procedures, and we have found that the results for all of them were very similar. We explain this particular bootstrap choice in detail in SI.3.
\subsection*{Social pressure and voter turnout}

In a large field experiment, Gerber et al. show that
substantially higher turnout was observed among registered voters who
received a mailing promising to publicize their turnout to their
neighbors \cite{GerberGreenLarimer}. In the United States, whether someone is registered to
vote and their past voting turnout are a matter of public
record. Of course, \textit{how} individuals voted is private. The experiment has been highly influential both
in the scholarly literature and in political practice. In our
reanalysis, we focus on two treatment conditions: the control group,
which was assigned to 191,243 individuals, and the ``neighbor's'' treatment group,
which was assigned to 38,218 individuals. Note the unequal sample
sizes. The experiment was conducted in Michigan before the August
2006 primary election, which was a statewide election with a wide
range of offices and proposals on the ballot. The authors randomly
assigned households with registered voters to receive mailers. The
outcome, whether someone voted, was observed in the primary
election. The ``neighbors'' mailing opens with a message that states
``DO YOUR CIVIC DUTY---VOTE!'' It then continues by not only listing
the household's voting records but also the voting records of those
living nearby. The mailer informed individuals that ``we intend to
mail an updated chart'' after the primary. 

The study consists of seven key individual-level covariates, most of
which are discrete: gender, age, and whether the registered individual
voted in the primary elections in 2000, 2002, and 2004 or the general
election in 2000 and 2002. The sample was restricted to voters who had
voted in the 2004 general election.  The outcome of interest is turnout in the 2006 primary election, which is an indicator variable. Because compliance is not observed, all estimates are of the Intention-to-Treat (ITT) effect, which is identified by the randomization. The average treatment effect estimated by the authors is 0.081 with a standard error of (0.003). Increasing voter turnout by 8.1\% using a simple mailer is a substantive effect, especially considering that many individuals may never have seen the mailer.  

\begin{figure}
	\centering
		\includegraphics[width=1\linewidth]{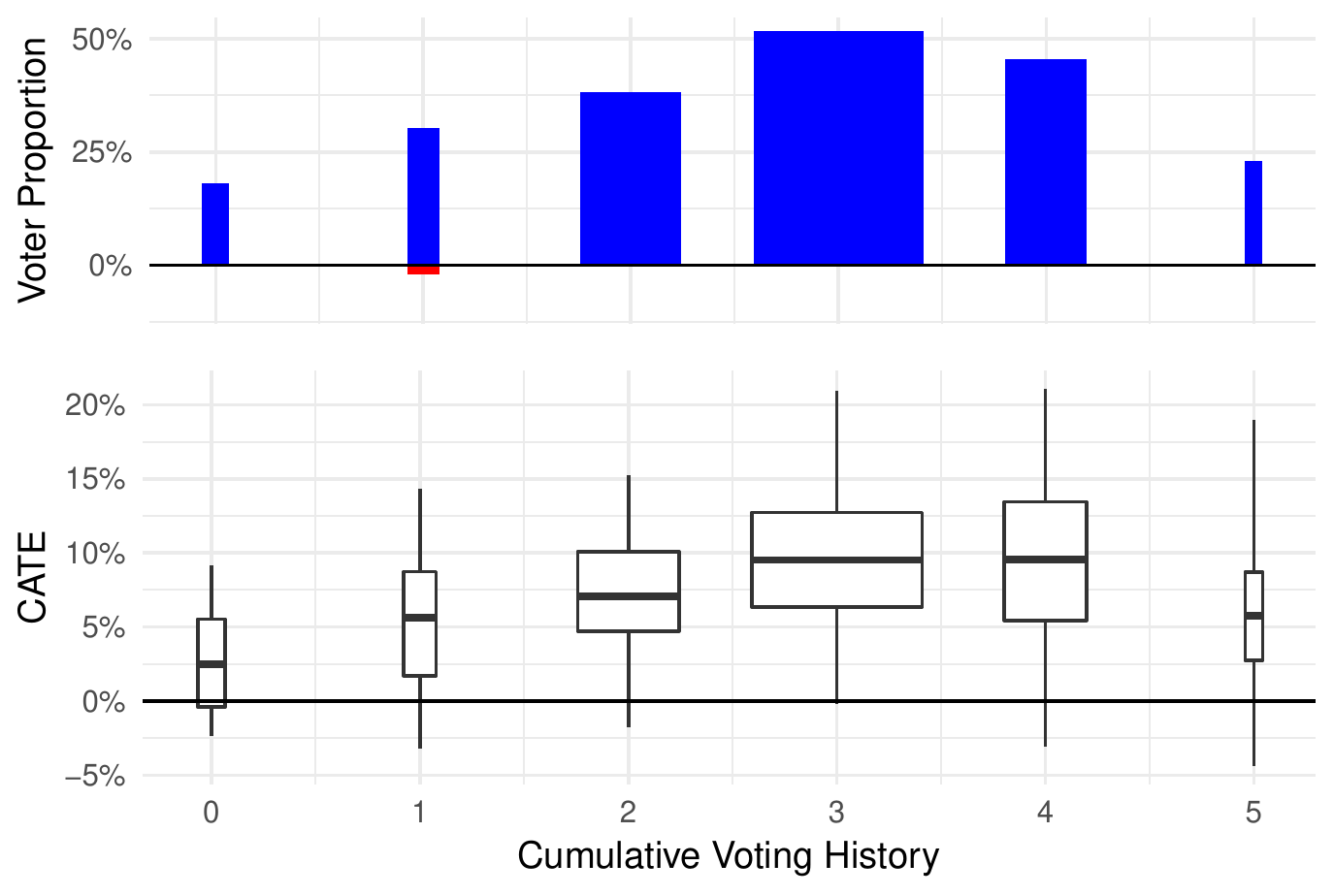}
	\caption{
		Social pressure and voter turnout. Potential voters are grouped by the number of elections they participated in, ranging from 0 (potential voters who did not vote during the past five elections) to 5 (voters who participated in all five past elections). The width of each group is proportional to the size of the group.
		Positive values in the first plot correspond to the percentage of voters for which the predicted CATE is significantly positive, while negative values correspond to the percentage of voters for which the predicted CATE is significantly negative. 
		The second plot shows the CATE estimate distribution for each bin.
	}
	\label{fig:gg2xlfeaturevstecmb}
	\label{fig:gotv}
\end{figure}

Figure \ref{fig:gotv} presents the estimated treatment effects, using X--RF where the potential voters are grouped by their voting history. 
The upper panel of the figure shows the proportion of voters with a significant positive (blue) and a significant negative (red) CATE estimate. We can see that there is evidence of a negative backlash among a small number of people who voted only once in the past five elections prior to the general election in 2004. Applied researchers have observed a backlash from these mailers; e.g., some recipients called their Secretary of States office or local election registrar to complain \citep{mann2010there,michelson2016risk}. 
The lower panel shows the distribution of CATE estimates for each of the subgroups. Having estimates of the heterogeneity enables campaigns to better target the mailers in the future. For example, if the number of mailers is limited, one should target potential voters who voted three times during the past five elections, since this group has the highest average treatment effect and it is a very big group of potential voters.\footnote{In praxis, it is not necessary to identify a particular subgroup. Instead, one can simply target units for which the predicted CATE is large. { If the goal of our analysis were to find subgroups with different treatment effects, one should validate those subgroup estimates. We suggest either splitting the data and letting the X-learner use part of the data to find subgroups and the other part to validate the subgroup estimates, or to use the suggested subgroups to conduct further experiments.}}

S--RF, T--RF, and X--RF all provide similar CATE estimates. This is unsurprising given the very large sample size, the small number of covariates, and their distributions. For example, the correlation between the CATE estimates of S--RF and T--RF is 0.99 (results for S--RF and T--RF can be found in Figure \ref{fig:gg2rslfeaturevstecmb}).

We conducted a data-inspired simulation study to see how these estimators would behave in smaller samples. We take the CATE estimates produced by T--RF, and we assume that they are the truth. We can then impute the potential outcomes under both treatment and control for every observation. We then sample training data from the complete data and predict the CATE estimates for
the test data using, S--, T--, and X--RF. We keep the unequal treatment proportion observed in the full data fixed, i.e., $\P(W = 1) = 0.167$. Figure \ref{fig:gotv.sim} presents the results of this simulation. They show that in small samples both X--RF and S--RF outperform T--RF, with X--RF performing the best, as one may conjecture given the unequal sample sizes.

\begin{figure}
	\centering
	\includegraphics[width=1\linewidth]{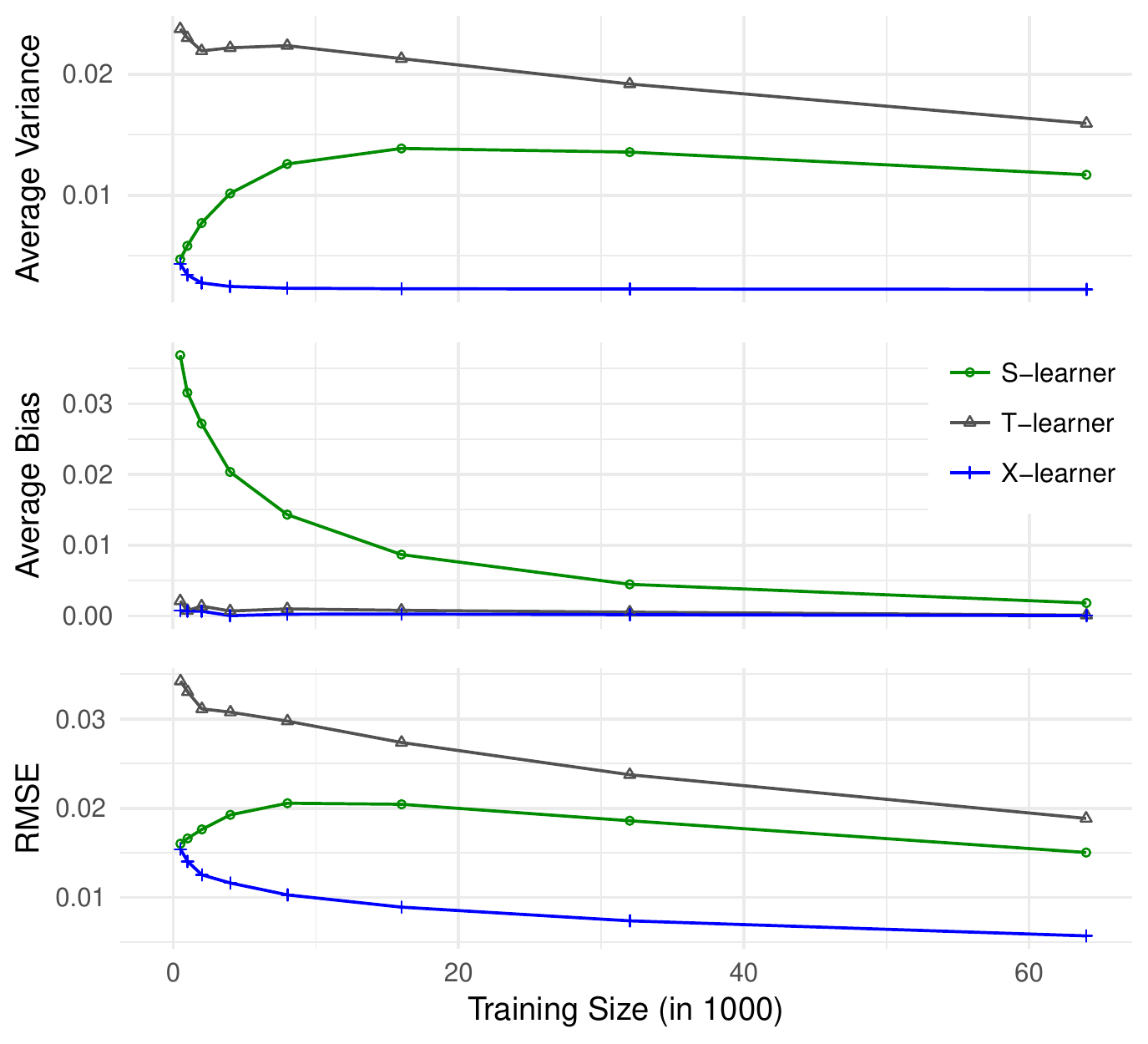}
	\caption{
		 RMSE, bias, and variance for a simulation based on the social pressure and voter turnout experiment.
	}
	\label{fig:gotv.sim}
\end{figure}

\subsection*{Reducing transphobia: A field experiment on door-to-door canvassing}

In an experiment that received widespread media attention, Broockman et al. show that brief (10 minutes) but high-quality door-to-door conversations can markedly reduce prejudice against gender-nonconforming individuals for at least three months \cite{transphobia}. This experiment was published in \textit{Science} after the journal retracted an earlier article claiming to show the same in an experiment about gay rights \citep{ScienceReal}. Broockman et al. showed that the earlier published study was fraudulent, and they conducted the new one to determine if the pioneering behavioral intervention of encouraging people to actively take the perspective of others was effective in decreasing prejudice \cite{broockman2015irregularities}.

There are important methodological differences between this example and our previous one. The experiment is a placebo-controlled experiment with a parallel survey that measures attitudes, which are the outcomes of interest. The authors follow the design of \citep{BKS2017design}. The authors first recruited registered voters ($n=68,378$) via mail for an unrelated online survey to measure baseline outcomes. They then randomly assigned respondents of the baseline survey to either the treatment group ($n=913$) or the placebo group that was targeted with a conversation about recycling ($n=912$). Randomization was conducted at the household level ($n=1295$), and because the design employs a placebo control, the estimand of interest is the complier-average-treatment effect. Outcomes were measured by the online survey three days, three weeks, six weeks, and three months after the door-to-door conversations. We analyze results for the first follow-up.

The final experimental sample consists of only 501 observations. The experiment was well powered despite its small sample size because it includes a baseline survey of respondents as well as post-treatment surveys. The survey questions were designed to have high over-time stability. The R$^2$ of regressing the outcomes of the placebo-control group on baseline covariates using OLS is 0.77. Therefore, covariate adjustment greatly reduces sampling variation. There are 26 baseline covariates that include basic demographics (gender, age, ethnicity) and baseline measures of political and social attitudes and opinions about prejudice in general and Miami's nondiscrimination law in particular. 

The authors find an average treatment effect of 0.22 (SE: 0.072, t-stat: 3.1) on their transgender tolerance scale.\footnote{The authors' transgender tolerance scale is the first principal component of combining five $-3$ to $+3$ Likert scales. See \cite{transphobia} for details.}
The scale is coded so that a larger number implies greater tolerance. The variance of the scale is 1.14, with a minimum observed value of -2.3 and a maximum of 2. This is a large effect given the scale. For example, the estimated decrease in transgender prejudice is greater than Americans' average decrease in homophobia from 1998 to 2012, when both are measured as changes in standard deviations of their respective scales. 

The authors report finding no evidence of heterogeneity in the treatment effect that can be explained by the observed covariates. Their analysis is based on linear models (OLS, lasso, and elastic net) without basis expansions.\footnote{\cite{transphobia} estimates the CATE using Algorithm \ref{algo:Flearner}.}
Figure \ref{fig:trans}(a) presents our results for estimating the CATE, using X--RF. We find that there is strong evidence that the positive effect that the authors find is only found among a subset of respondents that can be targeted based on observed covariates. The average of our CATE estimates is within half a standard deviation of the ATE that the authors report.
%
%
%
%
%

\begin{figure}
	\centering
	\includegraphics[width=1\linewidth]{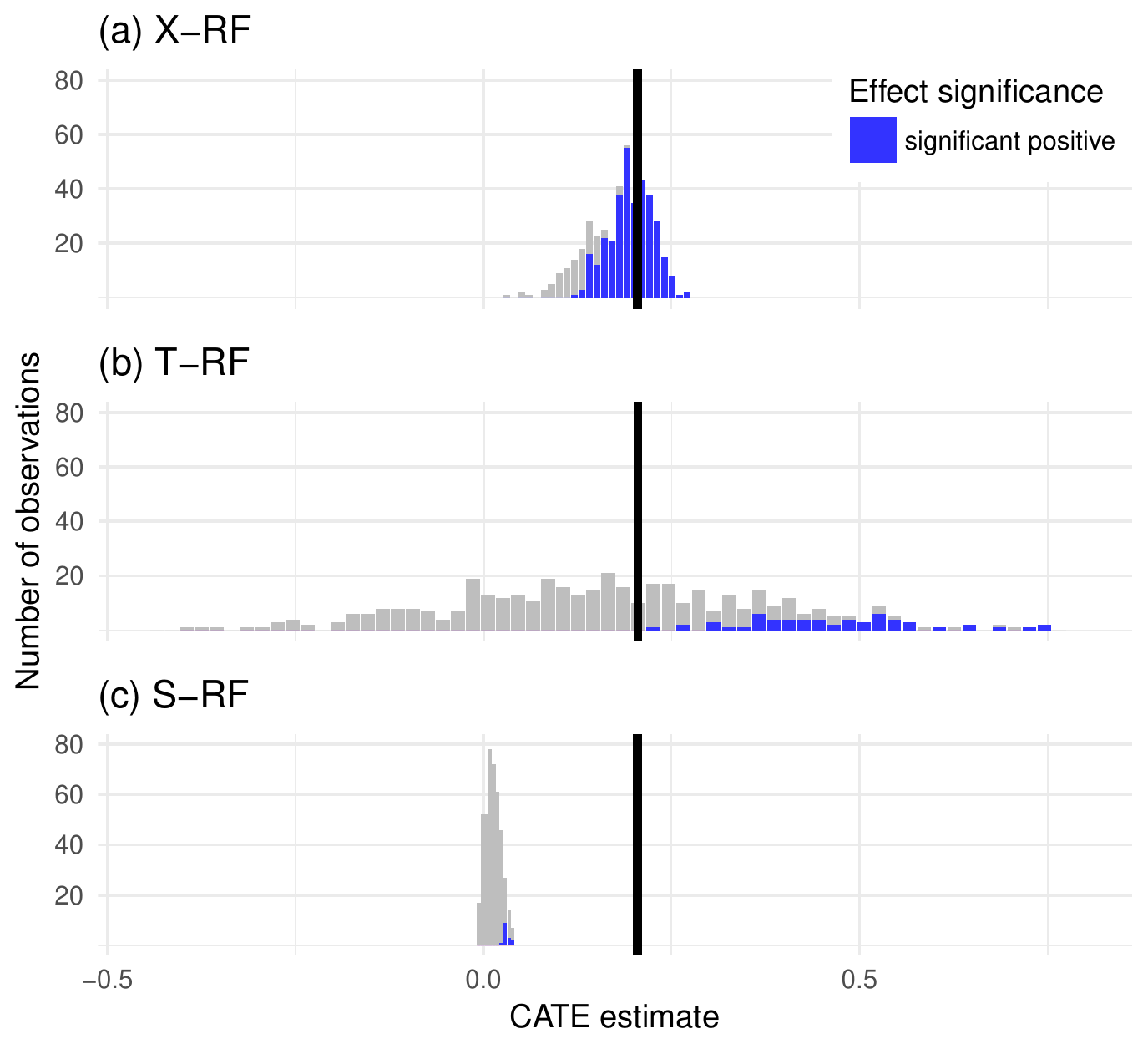}
	\caption{Histograms for the distribution of the CATE estimates in the Reducing Transphobia study. The horizontal line shows the position of the estimated ATE.}
	\label{fig:trans}
\end{figure}

Unlike in our previous data example, there are marked differences in the treatment effects estimated by our three learners. Figure \ref{fig:trans}(b) presents the estimates from T--RF. These estimates are similar to those of X--RF, but with a larger spread. Figure \ref{fig:trans}(c) presents the estimates from S--RF. Note that the average CATE estimate of S--RF is much lower than the ATE reported by the original authors and the average CATE estimates of the other two learners. Almost none of the CATE estimates are significantly different from zero. Recall that the ATE in the experiment was estimated with precision, and was large both substantively and statistically (t-stat=3.1). 

In this data, S--RF shrinks the treatment estimates toward zero. The ordering of the estimates we see in this data application is what we have often observed in simulations: the S-learner has the least spread around zero, the T-learner has the largest spread, and the X-learner is somewhere in between. 
Unlike in the previous example, the covariates are strongly predictive of the outcomes, and the splits in the S--RF are mostly on the features rather than the treatment indicator, because they are more predictive of the observed outcomes than the treatment assignment (cf., Figure \ref{fig:snosplit}).


	\section*{Conclusion}

This paper reviewed meta-algorithms for CATE estimation including the S- and T-learners. It then introduced a new meta-algorithm, the X-learner, that can translate any supervised learning or regression algorithm or a combination of such algorithms into a CATE estimator. The X-learner is adaptive to various settings. For example, both theory and data examples show that it performs particularly well when one of the treatment groups is much larger than the other or when the separate parts of the X-learner are able to exploit the structural properties of the response and treatment effect functions.
Specifically, if the CATE function is linear, but the response functions in the treatment and control group satisfy only the Lipschitz-continuity condition, the X-learner can still achieve the parametric rate if one of the groups is much larger than the other (Theorem \ref{theorem:lineartau}). If there are no regularity conditions on the CATE function and the response functions are Lipschitz continuous, then both the X-learner and the T-learner obtain the same minimax optimal rate (Theorem \ref{Theorem:Lipschitz_Convergence}). We conjecture that these results hold for more general model classes than those in our theorems.

We have presented a broad set of simulations to understand the finite sample behaviors of different implementations of these learners, especially for model classes that are not covered by our theoretical results. We have also examined two data applications. Although none of the meta-algorithms is always the best, the X-learner performs well overall, especially in the real-data examples. In practice, in finite samples, there will always be gains to be had if one accurately judges the underlying data-generating process. For example, if the treatment effect is simple, or even zero, then pooling the data across treatment and control conditions will be beneficial when estimating the response model (i.e., the S-learner will perform well). However, if the treatment effect is strongly heterogeneous and the response surfaces of the outcomes under treatment and control are very different, pooling the data will lead to worse finite sample performance (i.e., the T-learner will perform well). Other situations are possible and lead to different preferred estimators. For example, one could slightly change the S-learner so that it shrinks to the estimated ATE instead of zero, and it would then be preferred when the treatment effect is constant and non-zero.  One hopes that the X-learner can adapt to these different settings. The simulations and real-data studies presented have demonstrated the X-learner's adaptivity.  However, further studies and experience with more real data sets are necessary.
{ To enable practitioners to benchmark these learners on their own data sets, we have created an easy-to-use software library called \texttt{hte}. It implements several methods of selecting the best CATE estimator for a particular data set, and it implements confidence interval estimators for the CATE.}

In ongoing research, we are investigating using other supervised learning algorithms. For example, we are creating a deep learning architecture for estimating the CATE that is based on the X-learner with a particular focus on transferring information between different data sets and treatment groups.
Furthermore, we are concerned with finding better confidence intervals for the CATE. This might enable practitioners to better design experiments, and determine the required sample size before an experiment is conducted.



	\section*{Acknowledgement}
	We thank Rebecca Barter, David Broockman, Peng Ding, Avi Feller, Steve Howard, Josh Kalla, Fredrik S\"avje, Yotam Shem-Tov, Allen Tang, Simon Walter and seminar participants at Adobe, Columbia, MIT and Stanford for helpful discussions. We also thank Allen Tang for help with software development. We are responsible for all errors. The authors thank Office of Naval Research (ONR) Grants N00014-17-1-2176 (joint), N00014-15-1-2367 (Sekhon), N00014-16-1-2664 (Yu), ARO grant W911NF-17-10005, and the Center for Science of Information (CSoI), an NSF Science and Technology Center, under grant agreement CCF-0939370 (Yu).

	\appendix 
	\onecolumn
\section{Simulation  Studies}\label{section:simulation}
In this section, we compare the S-, T-, and X-learners in several simulation studies. We examine prototypical situations where one learner is preferred to the others. In practice, we recommend choosing powerful machine-learning algorithms such as BART \cite{hill2011bayesian}, Neural Networks, or RFs \cite{breiman2001random} for the base learners, since such methods perform well for a large variety of data sets. In what follows, we choose all the base learners to be either BART or honest RF algorithms---as implemented in the \texttt{hte} R package \cite{KuenzelHTE}---and we refer to these meta-learners as S--RF, T--RF, X--RF, S--BART, T--BART, and X--BART, respectively. Using two machine-learning algorithms as base learners helps us to demonstrate that our conclusions about the performance of the different meta learners is often independent of the particular base learner.  
For example, for all our simulation results we observe that if X--RF outperforms T--RF, then X--BART also outperforms T--BART.

\begin{remark}[BART and RF]
	BART and RF are regression tree-based algorithms that use all observations for each prediction, and they are in that sense global methods. However, BART seems to use global information more seriously than RF, and it performs particularly well when the data-generating process exhibits some global structures (e.g., global sparsity or linearity). RF, on the other hand, is relatively better when the data has some local structure that does not necessarily generalize to the entire space.  
\end{remark}

\subsection*{Causal Forests}
An estimator closely related to T--RF and S--RF is Causal Forests (CF) \cite{wager2015estimation}, because all three of these estimators can be defined as
$$
\hat\tau(x) = \hat \mu(x, w = 1) - \hat \mu(x, w = 0),
$$
where $\hat \mu(x, w)$ is a form of random forest with different constraints on the split on the treatment assignment, $W$. To be precise, in the S-learner the standard squared error loss function will decide where to split on $W$, and it can therefore happen anywhere in the tree. In the T-learner the split on $W$ must occur at the very beginning.\footnote{In the original statement of the algorithm we train separate RF estimators for each of the treatment groups, but they are equivalent.} For CF the split on $W$ is always made to be the split right before the terminal leaves. To obtain such splits, the splitting criterion has to be changed, and we refer to \cite{wager2015estimation} for a precise explanation of the algorithm. 

Figure \ref{fig:STCF_basetree} shows the differences between these learners for full trees with 16 leaves.

CF is not a meta-learner since the random forests algoirthm has to be changed. However, its similarity to T--RF and S--RF makes it interesting to evaluate its performance. Furthermore, one could conceivably generalize CF to other tree-based learners such as BART. However, this has not been done yet, and we will therefore compare CF in the following simulations to S--, T--, and X--RF.

\begin{figure}[!htb]
	\minipage{0.32\textwidth}
	\includegraphics[width=\linewidth]{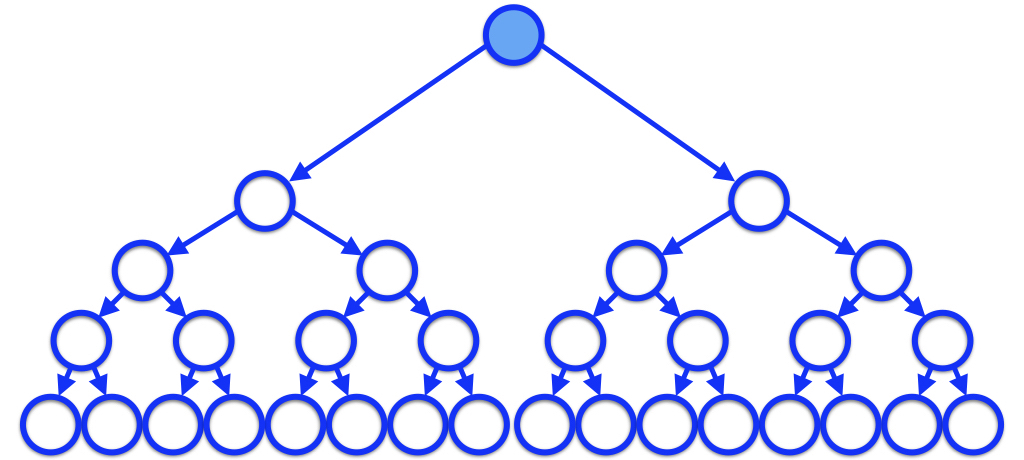}
	\caption*{T-learner}
	\endminipage\hfill
	\minipage{0.32\textwidth}
	\includegraphics[width=\linewidth]{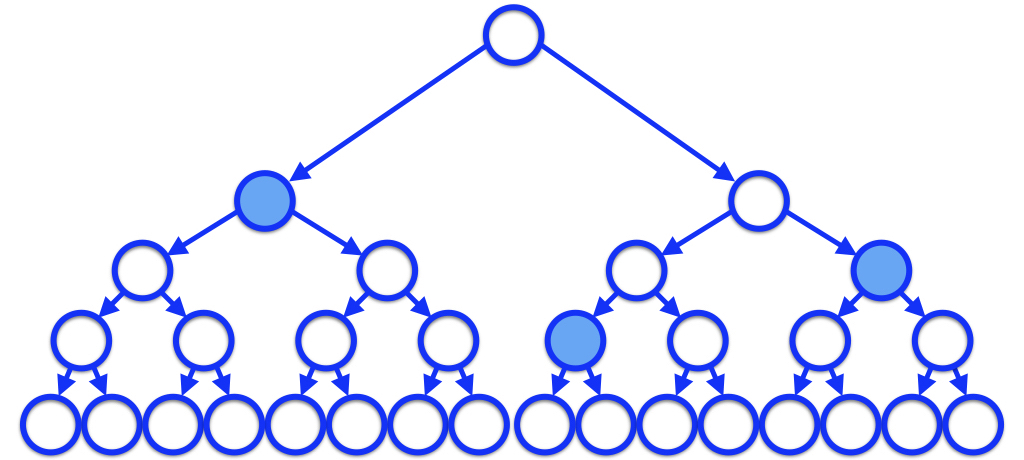}
	\caption*{S-learner}
	\endminipage\hfill
	\minipage{0.32\textwidth}
	\includegraphics[width=\linewidth]{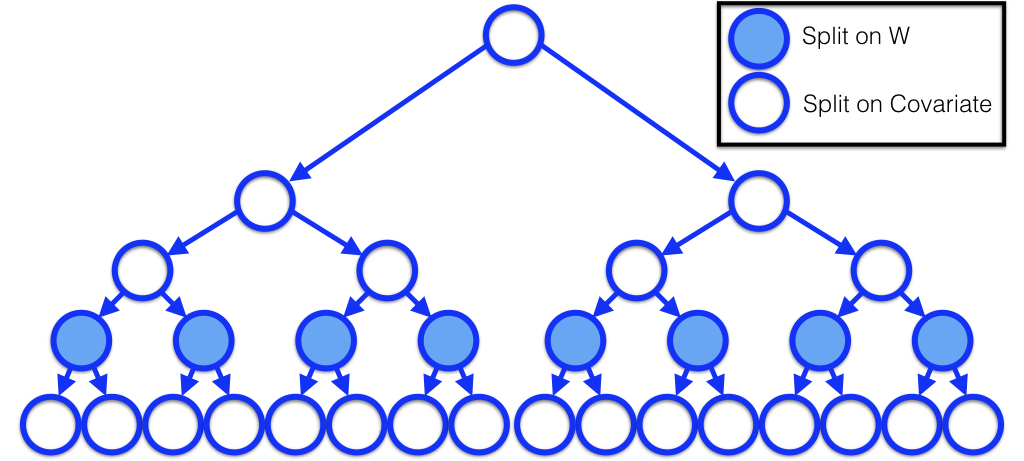}
	\caption*{Causal Forests}
	\endminipage
	\caption{Illustration of the structural form of the trees in T--RF, S--RF, and CF.} \label{fig:STCF_basetree}
\end{figure}

\subsection*{Simulation setup}
Let us here introduce the general framework of the following simulations.
For each simulation, we specify: the propensity score, $e$; the response functions, $\mu_0$ and $\mu_1$; the dimension, $d \in \N$, of the feature space; and a parameter, $\alpha$, which specifies the amount of confounding between features. 
To simulate an observation, $i$, in the training set, we simulate its feature vector, $X_i$, its treatment assignment, $W_i$, and its observed outcome, $\Yobs_i$, independently in the following way:
\begin{enumerate}
	\item
	First, we simulate a $d$-dimensional feature vector,
	\begin{equation} \label{equ:covariateCorrelation}
	X_i \overset{iid}{\sim} \mathcal N(0, \Sigma),
	\end{equation}
	where $\Sigma$ is a correlation matrix that is created using the \texttt{vine} method \cite{Lewandowski2009}.
	\item
	Next, we create the potential outcomes according to 
	\begin{align*}
	Y_i(1) = \mu_1(X_i) + \varepsilon_i(1),\\                
	Y_i(0) = \mu_0(X_i) + \varepsilon_i(0),
	\end{align*} 
	where $\varepsilon_i(1), \varepsilon_i(0) \overset{iid}{\sim} \mathcal{N}(0,1)$ and independent of $X_i$.
	\item
	Finally, we simulate the treatment assignment according to
	$$
	W_i \sim \mbox{Bern}(e(X_i)),
	$$
	we set $\Yobs_i = Y(W_i),$
	and we obtain $(X_i, W_i, \Yobs_i)$.\footnote{This is slightly different from the DGP we were considering for our theoretical results, because here $m$, the number of control units, and $n$, the number of treated units, are both random. The difference is, however, very small, since in our setups $N =m+n$ is very large.}
\end{enumerate}    

We train each CATE estimator on a training set of $N$ units, and we evaluate its performance against a test set of  $10^5$ units for which we know the true CATE. We repeat each experiment 30 times, and we report the averages. 

\subsection{The unbalanced case with a simple CATE}
We have already seen in Theorem \ref{theorem:lineartau} that the X-learner performs particularly well when the treatment group sizes are very unbalanced. We verify this effect as follows. 
We choose the propensity score to be constant and very small, $e(x) = 0.01$, such that on average only one percent of the units receive treatment. Furthermore, we choose the response functions in such a way that the CATE function is comparatively simple to estimate. 

\begin{simulation}[unbalanced treatment assignment] \label{sim:unbalanced}
	\begin{align*}
	e(x) &= 0.01,~~d = 20,\\
	\mu_0(x) &= x^T \beta + 5 ~ \mathbb{I}(x_1 > 0.5), ~~ \mbox{with} ~~ \beta \sim \mbox{Unif}\left([-5,5]^{20}\right),\\
	\mu_1(x) &= \mu_0(x) + 8 ~ \mathbb{I}(x_2 > 0.1).
	\end{align*}
\end{simulation}
The CATE function $\tau(x)= 8~\mathbb{I}(x_2 > 0.1)$ is a one-dimensional indicator function, and thus simpler than the 20-dim function for the response functions $\mu_0 (\cdot)$ and $\mu_1 (\cdot)$.
We can see in Figure \ref{fig:mseratesrare1d20a} that the X-learner indeed performs much better in this unbalanced setting with both BART and RF as base learners. 

\begin{figure}[ht]
	\centering
	\includegraphics[width=.6\linewidth]{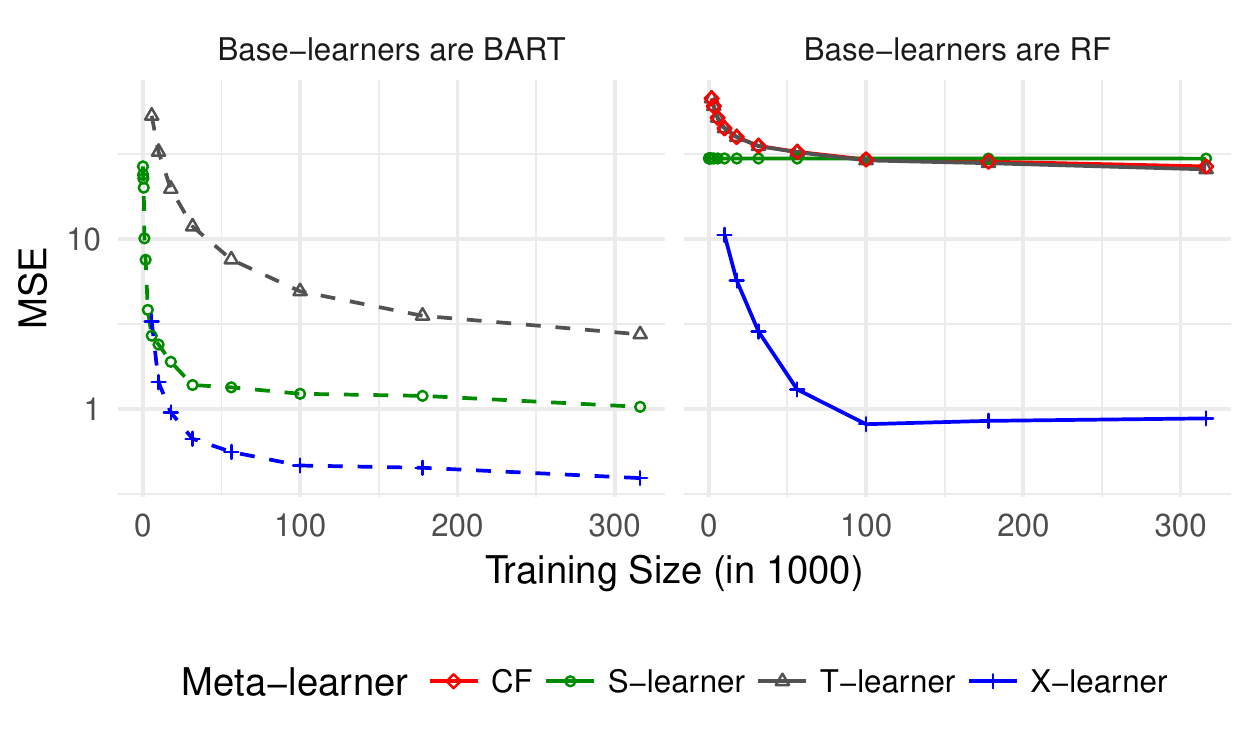}
	\caption{
		Comparison of S--, T--, and X--BART (left) and S--, T--, and X--RF and CF (right) for Simulation \ref{sim:unbalanced}. 
	}
	\label{fig:mseratesrare1d20a}
\end{figure}

\subsection{Balanced cases without confounding}

Next, let us analyze two extreme cases: In one of them the CATE function is very complex and in the other one the CATE function is equal to zero. We will show that for the case of no treatment effect, the S-learner performs very well since it sometimes does not split on the treatment indicator at all and it tends to be biased toward zero. 
On the other hand, for the complex CATE case simulation we have chosen, there is nothing to be learned from the treatment group about the control group and vice versa. Here the T-learner performs very well, while the S-learner is often biased toward zero.  
Unlike the T-learner, the X-learner pools the data, and it therefore performs well in the simple CATE case. And unlike the S-learner, the X-learner is not biased toward zero. It therefore performs well in both cases.

\subsubsection{Complex CATE}
Let us first consider the case where the treatment effect is as complex as the response functions in the sense that it does not satisfy regularity conditions (such as sparsity or linearity) that the response functions do not satisfy.
We study two simulations here, and we choose for both the feature dimension to be $d = 20$, and the propensity score to be $e(x) = 0.5$. In the first setup (complex linear) the response functions are different linear functions of the entire feature space.
\begin{simulation}[complex linear] \label{sim:complexlin}
	
	\begin{align*}
	e(x) &= 0.5,~~d = 20,\\
	\mu_1(x) &= x^T \beta_1, ~\mbox{with} ~ \beta_1 \sim \mbox{Unif}([1, 30]^{20}),\\
	\mu_0(x)  &= x^T \beta_0, ~\mbox{with} ~ \beta_0 \sim \mbox{Unif}([1, 30]^{20}).    
	\end{align*}
	
\end{simulation}
The second setup (complex non-linear) is motivated by \cite{wager2015estimation}. Here the response function are non-linear functions.

\begin{simulation}[complex non-linear]\label{sim:complexnonlin}
	
	\begin{align*}
	e(x) &= 0.5,~~d = 20,\\
	\mu_1(x) &= \frac12\varsigma(x_1)\varsigma(x_2),\\
	\mu_0(x)  &= -\frac12 \varsigma(x_1)\varsigma(x_2)\notag    
	\end{align*}
	with 
	$$
	\varsigma(x) = \frac{2}{1 + e^{-12 (x - 1/2)}}.
	$$
	
\end{simulation}

Figure \ref{fig:mseratescomplexcate} shows the MSE performance of the different learners. 
In this case, it is best to separate the CATE estimation problem into the two problems of estimating $\mu_0$ and $\mu_1$ since there is nothing one can learn from the other assignment group. The T-learner follows exactly this strategy and should perform very well.
The S-learner, on the other hand, pools the data and needs to learn that the response function for the treatment and the response function for the control group are very different. 
However, in the simulations we study here, the difference seems to matter only very little. 

Another interesting insight is that choosing BART or RF as the base learner can matter a great deal. BART performs very well when the response surfaces satisfy global properties such as being globally linear, as in Simulation \ref{sim:complexlin}. 
However, in Simulation \ref{sim:complexnonlin}, the response surfaces do not satisfy such global properties.
Here the optimal splitting policy differs throughout the space and this non-global behavior is harmful to BART. Thus, choosing RF as the base learners results in a better performance here. 
Researchers should use their subject knowledge when choosing the right base learner. 
\begin{figure}
	\centering
	\includegraphics[width=.6\linewidth]{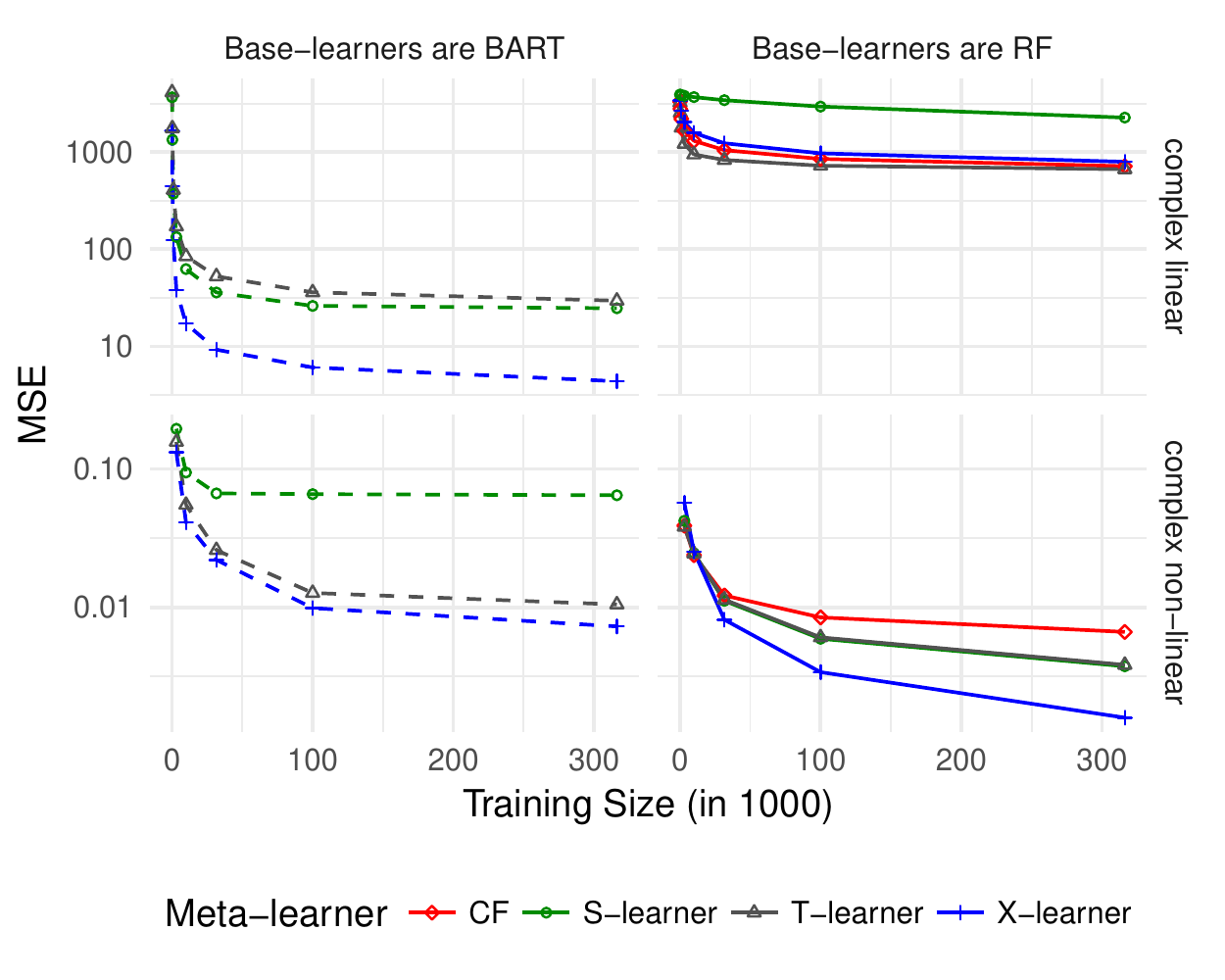}
	\caption{
		Comparison of the S-, T-, and X-learners with BART (left) and RF (right) as base learners for Simulation \ref{sim:complexlin} (top) and Simulation \ref{sim:complexnonlin} (bottom).
	}
	\label{fig:mseratescomplexcate}
\end{figure}

\subsubsection{No treatment effect}
Let us now consider the other extreme where we choose the response functions to be equal. This leads to a zero treatment effect, which is very favorable for the S-learner.
We will again consider two simulations where the feature dimension is 20, and the propensity score is constant and 0.5.

We start with a global linear model (Simulation \ref{sim:globallin}) for both response functions. In Simulation \ref{sim:piecwlin}, we simulate some interaction by slicing the space into three parts, $\{x : x_{20} < -0.4\}$, $\{x : -0.4< x_{20} < 0.4 \}$, and $\{x : 0.4< x_{20} \}$, where for each of the three parts of the space a different linear response function holds. We do this because we believe that in many data sets there is a local structure, that appears only in some parts of the space. 

\begin{simulation}[global linear] \label{sim:globallin}
	\begin{align*}
	e(x) &= 0.5,~~d = 5,\\
	\mu_0(x) &= x^T \beta, ~\mbox{with} ~ \beta \sim \mbox{Unif}([1, 30]^5),\\
	\mu_1(x) &= \mu_0(x).
	\end{align*}
\end{simulation}

\begin{simulation}[piecewise linear] \label{sim:piecwlin}
	\begin{align*}
	e(x) &= 0.5,~~d = 20,\\
	\mu_0(x) &=
	\begin{cases}
	x^T \beta_l    &\mbox{if} ~~ x_{20} < -0.4\\
	x^T \beta_m &\mbox{if} ~~ -0.4 \le x_{20} \le 0.4\\
	x^T \beta_u   &\mbox{if} ~~  0.4 < x_{20}, 
	\end{cases}\\
	\mu_1(x) &= \mu_0(x),
	\end{align*}
	with 
	\begin{align*}
	&\beta_l(i) = \begin{cases}
	\beta(i)  &\mbox{if} ~~ i \le 5\\
	0 &\mbox{otherwise}
	\end{cases}&                    
	&\beta_m(i) = \begin{cases}
	\beta(i)  &\mbox{if} ~~ 6 \le i \le 10\\
	0 &\mbox{otherwise}
	\end{cases}&                    
	&\beta_u(i) = \begin{cases}
	\beta(i)  &\mbox{if} ~~ 11 \le i \le 15\\
	0 &\mbox{otherwise}
	\end{cases}&                    
	\end{align*}
	and 
	$$
	\beta \sim \mbox{Unif}([-15, 15]^d).
	$$
\end{simulation}

Figure \ref{fig:mseratessimplecate} shows the outcome of these simulations. For both simulations, the CATE is globally 0. As expected, the S-learner performs very well, since the treatment assignment has no predictive power for the combined response surface. The S-learner thus often ignores the variable encoding the treatment assignment, and the S-learner correctly predicts a zero treatment effect. 
We can again see that the global property of the BART harms its performance in the piecewise linear case since here the importance of the features is different in different parts of the space. 

\begin{figure}
	\centering
	\includegraphics[width=.6\linewidth]{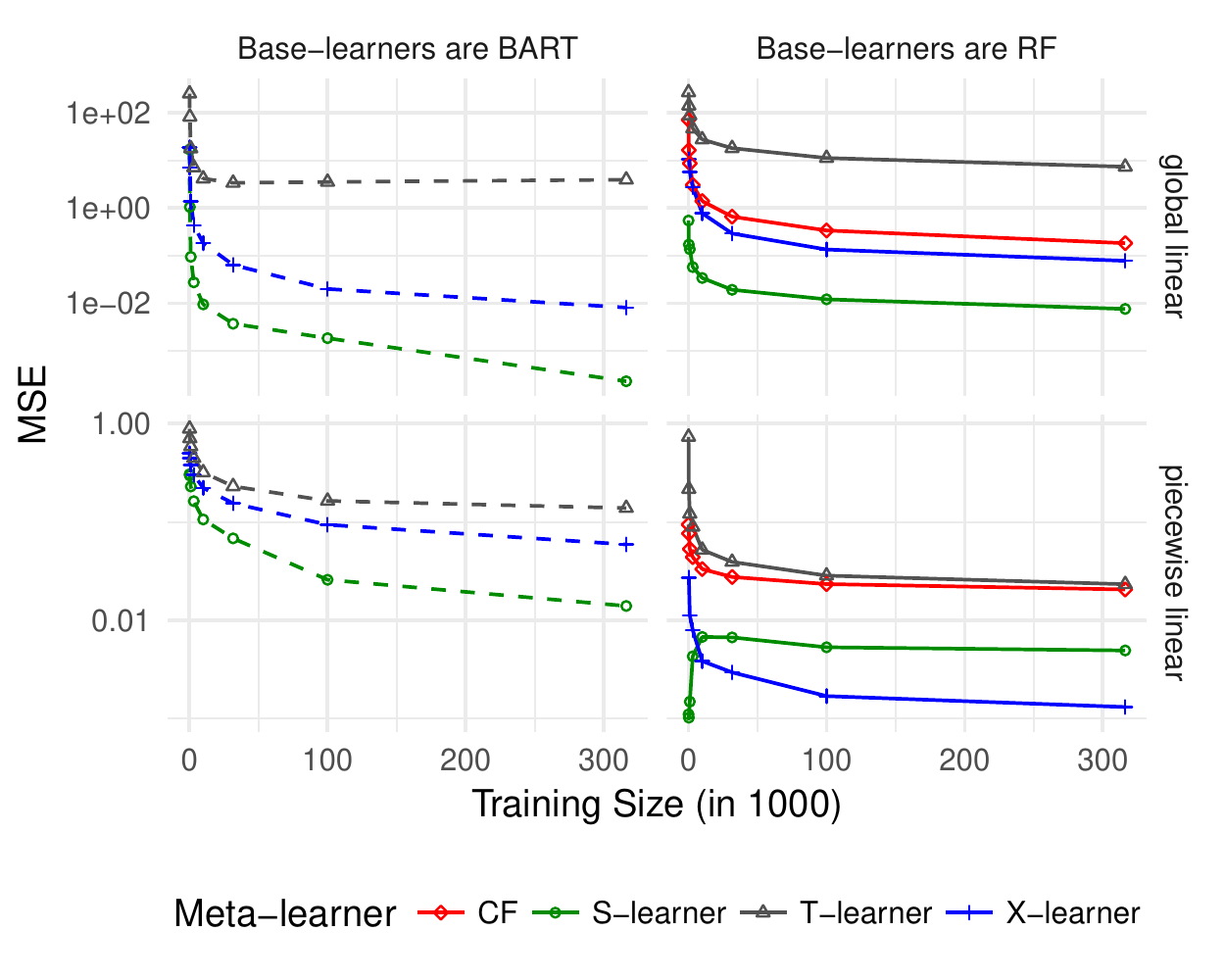}
	\caption{
		Comparison of S-, T-, and X-learners with BART (left) and RF (right) as base learners for Simulation \ref{sim:globallin} (top) and Simulation \ref{sim:piecwlin} (bottom).
	}
	\label{fig:mseratessimplecate}
\end{figure}

\subsection{Confounding}
In the preceding examples, the propensity score was globally equal to some constant. This is a special case, and in many observational studies, we cannot assume this to be true. All of the meta-learners we discuss can handle confounding, as long as the ignorability assumption holds. We test this in a setting that has also been studied in \cite{wager2015estimation}. For this setting we choose $x \sim Unif([0,1]^{n\times 20})$ and we use the notation that $\beta(x_1, 2, 4)$ is the $\beta$ distribution with parameters 2 and 4.
\begin{simulation}[beta confounded] \label{sim:betacon}
	\begin{align*}
	e(x) &= \frac14(1 + \beta(x_1, 2, 4) ),\\
	\mu_0(x) &= 2 x_1  - 1,\\
	\mu_1(x) &= \mu_0(x).
	\end{align*}
\end{simulation}

Figure \ref{fig:simulationobservationalstudies} shows that none of the algorithms performs significantly worse under confounding. We do not show the performance of causal forests, because---as noted by the authors---it is not designed for observational studies with only conditional unconfoundedness and it would not be fair to compare it here  \cite{wager2015estimation}.

\begin{figure}
	\centering
	\includegraphics[width=.6\linewidth]{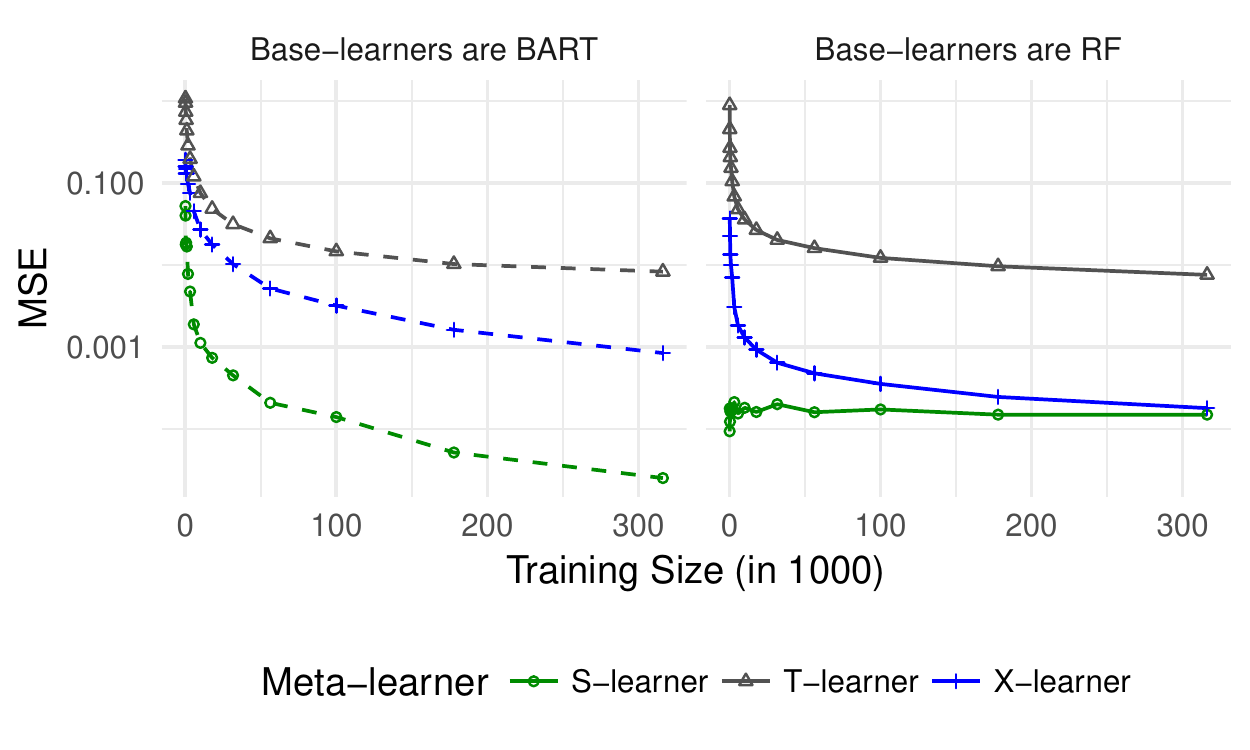}
	\caption{
				Comparison of S--, T--, and X--BART (left) and S--, T--, and X--RF (right) for Simulation \ref{sim:betacon}.
	}
	\label{fig:simulationobservationalstudies}
\end{figure}


\section{Notes on the ITE}
We provide an example that demonstrates that the ITE is not identifiable without further assumptions. Similar arguments and examples have been given before \cite{heckman1997making}, and we list it here only for completeness. 
\begin{example}[$D_i$ is not identifiable] \label{Exa:ITEnotIdfbl}
	Assume that we observe a one-dimensional and uniformly distributed feature between 0 and 1, 
    $ X\sim \mbox{Unif}([0,1]),$
    a treatment assignment that is independent of the feature and Bernoulli distributed, 
    $W \sim \mbox{Bern}(0.5),$
    and a Rademacher-distributed outcome under control that is independent of the features and the treatment assignment,
    $$P(Y(0) = 1) = P(Y(0) = -1) = 0.5.$$
    Now consider two Data-Generating Processes (DGP) identified by the distribution of the outcomes under treatment: 
    \begin{enumerate}
    	\item In the first DGP, the outcome under treatment is equal to the outcome under control:
        	$$ Y(1) = Y(0).$$
         \item In the second DGP, the outcome under treatment is the negative of the outcome under control:
         	$$ Y(1) = - Y(0).$$ 
     \end{enumerate}
     Note that the observed data, $\D = (\Yobs_j, X_j, W_j)_{1\le j \le N}$, has the same distribution for both DGPs, but $D_i = 0$ for all $i$ in  DGP 1, and $D_i \in \{-2, 2\}$ for all i in DGP 2. 
     Thus, no estimator based on the observed data $\D$ can be consistent for the ITEs, $(D_i)_{1\le i \le n}$. The CATE, $\tau(X_i)$, is, however, equal to 0 in both DGPs. $\hat \tau \equiv 0$, for example, is a consistent estimator for the CATE.
\end{example}

\section{Confidence Intervals for the Social Pressure Analysis} \label{sec:CI}
In this paper, we study general meta-learners without making any parametric assumptions on the CATE. This generality makes it very difficult to provide confidence intervals with formal guarantees. In the GOTV section of the main paper, we used bootstrap confidence intervals; in this section, we explain why we choose the bootstrap and details of the variant of the bootstrap, we selected.

The bootstrap has been proven to perform well in many situations \cite{liu2013asymptotic} and it is straightforward to apply to any estimator that can be written as a function of iid data. 
There are, however, many ways to obtain bootstrap confidence intervals. We have decided to use Algorithm \ref{alg:computeCI}, because it performed well for X--RF in the Atlantic Causal Inference Conference (ACIC) challenge \cite{dorie2017automated}, where one of the goals was to create confidence intervals for a wide variety of CATE estimation problems. We refer to these confidence intervals as normal approximated CIs.

It was seen in the ACIC challenge that constructing confidence intervals for the CATE that achieve their nominal coverage is extremely difficult, and no method always provides the correct coverage. 
To argue that the conclusions we draw in this paper are not specific to a single bootstrap method, we implement another version of the bootstrap to estimate confidence intervals due to \cite{putter2012resampling} and \cite{efron2014estimation}. We refer to it as the smoothed bootstrap, and we call the corresponding confidence intervals smoothed CIs. Pseudocode for this method can be found in Algorithm \ref{alg:computeCI2}.

There are many other versions of the bootstrap that could have been chosen, but we focus on two that performed well in the ACIC challenge. 
To compare these methods, we use the GOTV data, and we estimate confidence intervals for $2,000$ test points based on $50,000$ training points. We have to use this much smaller subset of the data for computational reasons.

For both methods, we use $B = 10,000$ bootstrap samples. This is a large number of replications, but it is necessary because the smoothed CIs (Algorithm \ref{alg:computeCI2}) are unstable for a smaller $B$. 
Figure \ref{fig:smoothedvsnormalcomparison} compares the center and the length of the confidence intervals of the two methods for T--RF.
We can see that the two methods lead to almost the same confidence intervals. The normal approximated CIs are slightly larger, but the difference is not substantial.
This is not surprising given the size of the data, and it confirms that our analysis of the GOTV data would have come to the same conclusion had we used smoothed CIs (Algorithm \ref{alg:computeCI2}). However, normal approximated CIs (Algorithm \ref{alg:computeCI}) are computationally much less expensive and they are therefore our default method. 

\begin{figure}
	\centering
	\includegraphics[width=\linewidth]{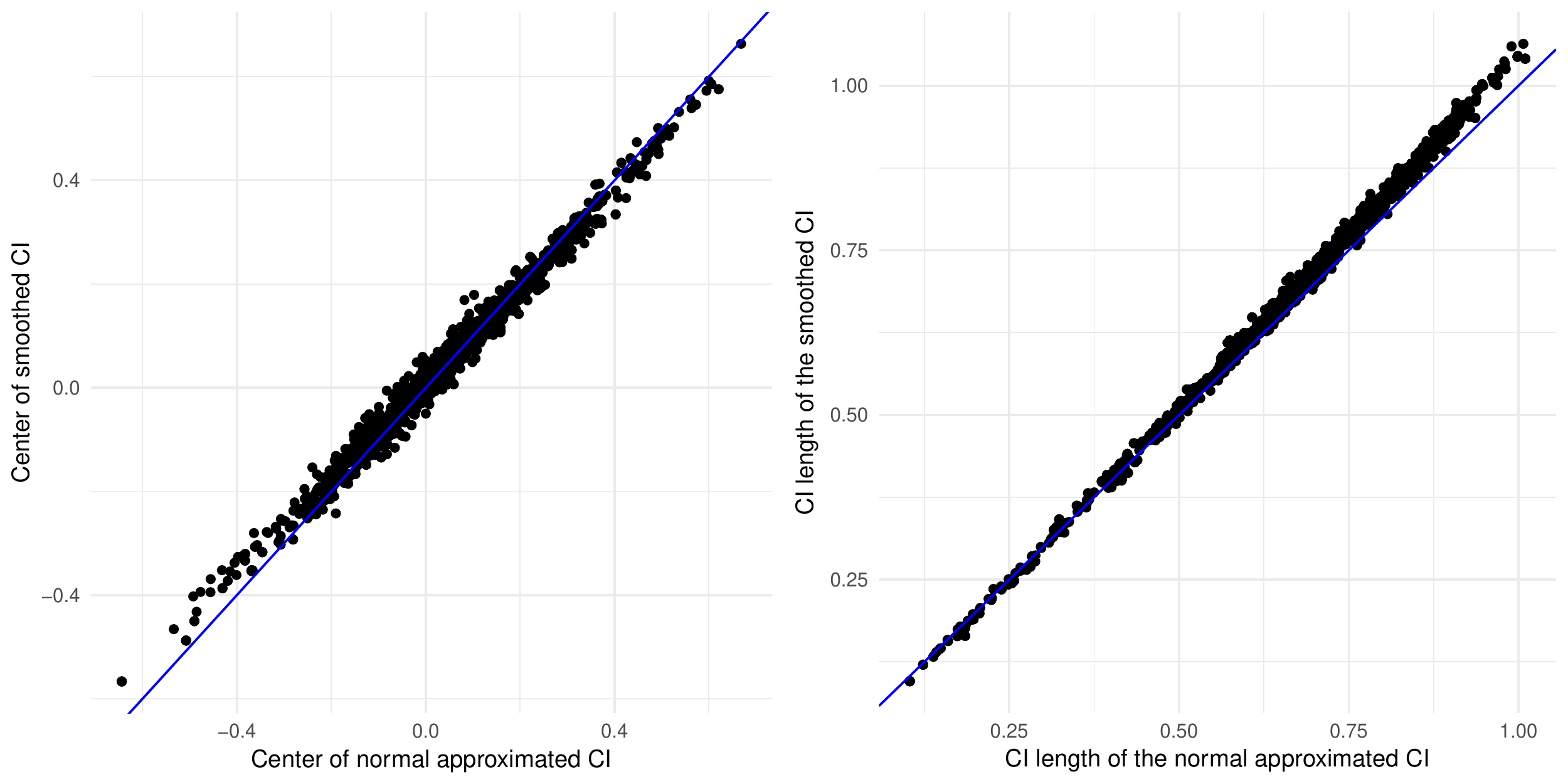}
	\caption{Comparison of normal approximated CI (Algorithm \ref{alg:computeCI}) and smoothed CI (Algorithm \ref{alg:computeCI2}). The blue line is the identity function.}
	\label{fig:smoothedvsnormalcomparison}
\end{figure}

\subsection{CI-Simulation 1: Comparison of the coverage of the CI estimation methods} \label{section:CIcomparison}
To analyze the coverage of the different bootstrap methods, we use a simulation study informed by the GOTV data. 
We generate the data in the following way:
\subsubsection*{CI-Simulation 1}

\begin{enumerate} 
	\item
	We start by training the T-learner with random forests on the entire GOTV data set to receive CATE estimates. We take this estimate as the ground truth and call it $\tau(x)$. 
	\item 
	We then compute for each unit $i$ the missing potential outcome.
	That is, for a unit in the control group, we add $\tau(x_i)$ to the observed outcome to obtain the outcome under treatment, and for each unit in the treatment group, we subtract $\tau(x_i)$ from the observed outcome to obtain the outcome under control.
	\item Next, we create a new treatment assignment by permuting the original one.
	This also determines our new observed outcome. 
	\item Finally, we sample uniformly and without replacement a test set of  $2,000$ observations and a training set of $50,000$ observations. 
\end{enumerate}

We then compute 95\% confidence intervals for each point in the test set using the the normal and smoothed bootstrap combined with the S, T, and X-learner. 
The left part of Figure \ref{fig:coveragevslength} shows a comparison of the six methods. 
 We find that none of the methods provide the correct coverage. 
The coverage of the smooth bootstrap intervals is slightly higher than the coverage of the normal approximated confidence intervals, but the difference is within 1\%. 
It also appears that the T-learner provides the best coverage, but it also has the largest confidence interval length.

Based on this simulation, we believe that the smooth CIs have a slightly higher coverage but the intervals are also slightly longer. However, the smooth CIs are computationally much more expensive and need a lot of bootstrap samples to be stable. They are therefore unfeasible for our data.
Hence we prefer the normal approximated CIs.

\begin{figure}
	\centering
	\includegraphics[width=1\linewidth]{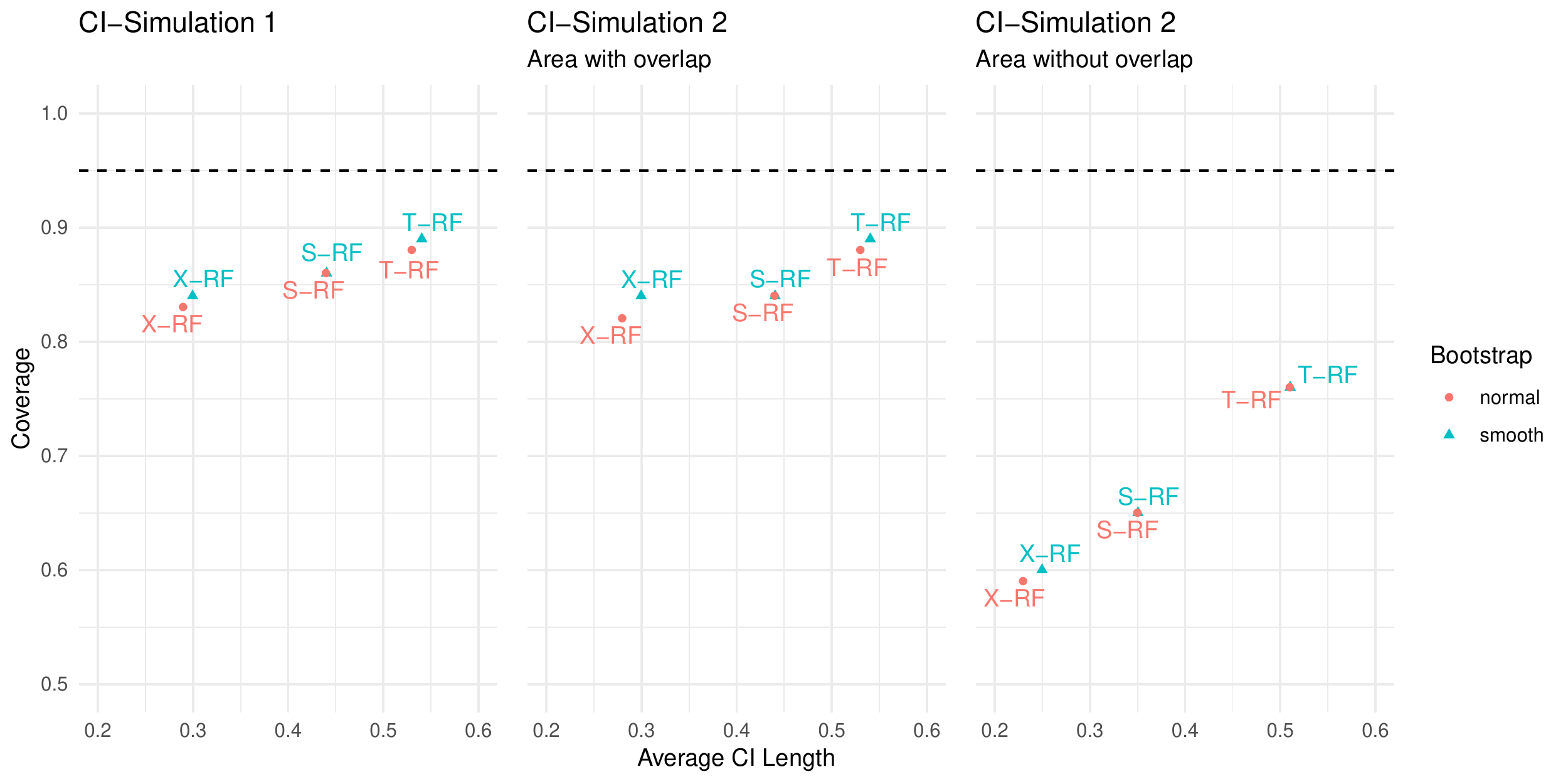}
	\caption{ Coverage and average confidence interval length of the three meta-learners for normal approximated CI (Algorithm \ref{alg:computeCI}) and smoothed CI (Algorithm \ref{alg:computeCI2}). The left figure corresponds to Simulation 3.1; the middle figure corresponds to units in an area with overlap in Simulation SI 3.2, and the right figure corresponds to units in an area without overlap in Simulation SI 3.2. The dotted line corresponds to the target 95\% confidence interval.}
	\label{fig:coveragevslength}
\end{figure}

In general, we observe that none of the methods achieve the anticipated 95\% coverage and we suspect that this is the case, because the CATE estimators are biased and the bootstrap is not adjusting for the bias terms. 
To analyze this, we approximated the bias using a Monte Carlos simulation for each of the 2,000 test points using Algorithm \ref{alg:computeBIAS}.
The density plot in Figure \ref{fig:biasestimation} shows that the bias of X--RF in our sample is substantial and in particular of the same order as the size of the confidence intervals of X--RF. For example, more than 11\% of all units had  bias bigger than 0.15.

This raises the question whether it is possible to correct for the bias. 
We tried to use the bootstrap again to estimate the bias. Specifically, we used Algorithm \ref{alg:estimateBIAS} to estimate it. 
The upper subfigure in Figure \ref{fig:biasestimation} is a scatter plot of the Monte-Carlo-approximated bias versus the bootstrap-estimated bias. 
We can see that the bootstrap does not correctly estimate the bias.

\begin{figure}
			\centering
			\includegraphics[width=0.4\linewidth]{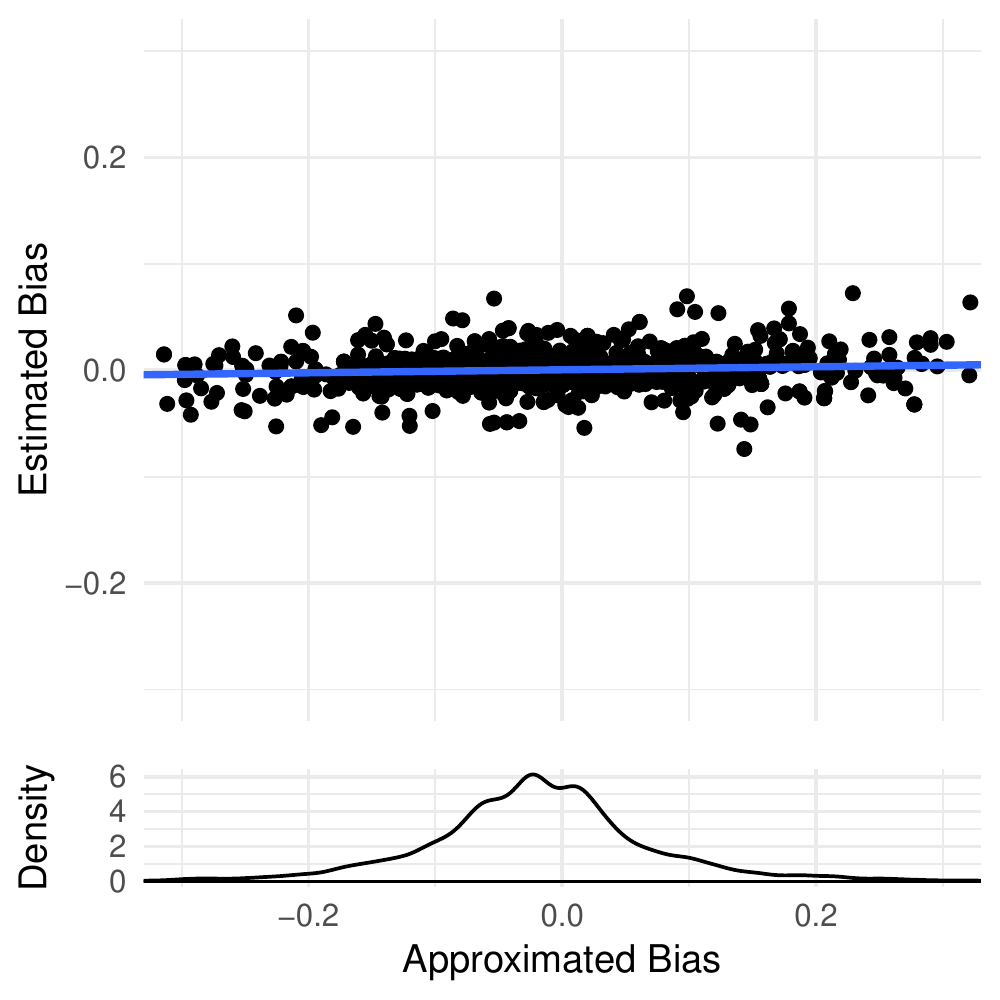}
			\caption{ Approximated bias using Algorithm \ref{alg:computeBIAS} versus estimated bias using Algorithm \ref{alg:estimateBIAS} and X--RF.}
			\label{fig:biasestimation}
\end{figure}

\subsection{CI-Simulation 2: Confounding without overlap} \label{sec:confwo}
In observational studies, researchers have no control over the treatment assignment process and, in some cases, even the overlap condition may be violated. That is, there exists a subgroup of units that is identifiable by observed features for which the propensity score is 0 or 1. Consequently, all units are either treated or not and estimating the CATE is impossible without very strong assumptions. We generally advise researchers to be very cautious when using these methods on observational data. 
In this section, we want to study how well one can estimate confidence intervals in observational studies where the overlap condition is violated. Ideally, we would hope that the confidence intervals in areas with no overlap are extremely wide. 

To test the behavior of the different confidence interval estimation methods, we set up another simulation based on real data. In this simulation we intentionally violate the overlap condition by assigning all units between 30 and 40 years to the control group. We then compared the confidence intervals for this subgroup with the other units where the overlap condition is not violated. 
For our simulation, we follow the same steps as in Section \ref{section:CIcomparison}, but we modified Step 3 to ensure that all units between 30 and 40 years of age are in the control group. 
Specifically, we construct the data in the following way: 
\begin{enumerate} 
	\item
	We start by training the T-learner with random forests on the entire GOTV data set to construct CATE estimates. We take this estimate as the ground truth and call it $\tau(x)$. 
	\item 
	We then use $\tau(x)$ to impute the missing potential outcomes.
	That is, for a unit in the control group, we add $\tau(x_i)$ to the observed outcome to obtain the outcome under treatment, and for each unit in the treatment group, we subtract $\tau(x_i)$ from the observed outcome to obtain the outcome under control.
	\item Next, we create a new treatment assignment by permuting the original treatment assignment vector and  assigning all entries for  units  between 30 and 40 years old to the control group.
	This also determines our new observed outcome. 
	\item Finally, we sample uniformly and without replacement two test sets and one training set. 
	We first sample the training set of $50,000$ observations. 
	Next, we sample the first test set of $20,000$ units out of all units that are not in the 30 to 40-year-old age group. This test set is called the \textbf{overlap test set}.
	Finally, we sample the second test set of $20,000$ units out of all units in the 30 to 40-year-old age group and we call this test set the  \textbf{non-overlap test set}.  
\end{enumerate}
Note that by construction the overlap condition is violated for the subgroup of units between 30 and 40 years and satisfied for units outside of that age group.

We trained each method on the training set and estimated the confidence intervals for the CATE in both test sets. 
The middle and the right part of Figure \ref{fig:coveragevslength} shows the results for the overlap test set and the non-overlap test set, respectively.
We find that the coverage and the average confidence interval length for the overlap test set is very similar to that of the previous simulation study, CI-Simulation 1. This is not surprising, because the two setups are very similar and the overlap condition is satisfied in both. 

The coverage and the average length of the confidence intervals for the non-overlap test set is, however, very different. 
For this subgroup, we do not have overlap.
We should be very cautious when estimating the CATE or confidence intervals of the CATE, and we would hope to see this reflected by very wide confidence intervals.
This is unfortunately not the case. 
We observe that the confidence intervals are tighter and the coverage is much lower than on the data where we have overlap. 
This is a problematic finding and suggests that confidence interval estimation in observational data is extremely difficult and that a violation of the overlap condition can lead to invalid inferences.

\section{Stability of the Social Pressure Analysis across Meta-learners}
In Figure \ref{fig:gotv}, we present how the CATE varies with the observed covariates. We find a very interesting behavior in the fact that the largest treatment effect can be observed for potential voters who voted three or four times before the 2004 general election. The treatment effect for potential voters who voted in none or all five of the observed elections was much smaller.
We concluded this based on the output of the X-learner.
To show that a similar conclusion can be drawn using different meta-learners, we repeated our analysis with the S and T learner (cf. Figure \ref{fig:gg2rslfeaturevstecmb}). We find that the output is almost identical to the output of the X-learner. This is not surprising since the data set is very large and most of the covariates are discrete. 
\begin{figure}
	\centering
	\includegraphics[width=0.9\linewidth]{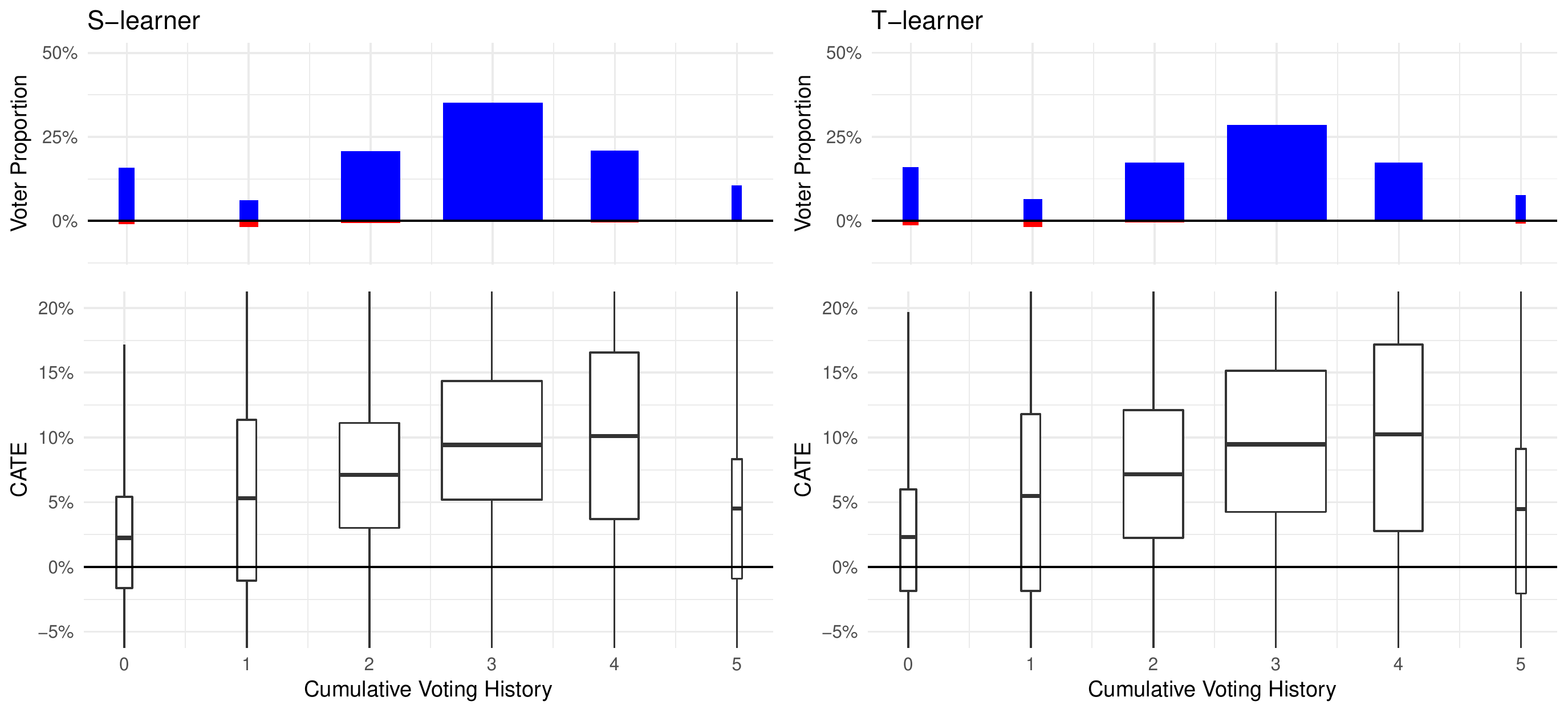}
	\caption{Results for the S-learner (left) and the T-learner (right) for the get-out-the-vote experiment.}
	\label{fig:gg2rslfeaturevstecmb}
\end{figure}

\section{The Bias of the S-learner in the Reducing Transphobia Study}
For many base learners, the S-learner can completely ignore the treatment assignment and thus predict a 0 treatment effect. This often leads to a bias toward 0, as we can see in Figure \ref{fig:trans}. To further analyze this behavior, we trained a random forest estimator on the transphobia data set with 100,000 trees, and we explored how often the individual trees predict a 0 treatment effect by not splitting on the treatment assignment. 
Figure \ref{fig:snosplit} shows that the trees very rarely split on the treatment assignment. This is not surprising for this data set since the covariates are very predictive of the control response function and the treatment assignment is a relatively weak predictor. 
\begin{figure}[h!]
	\centering
	\includegraphics[width=0.7\linewidth]{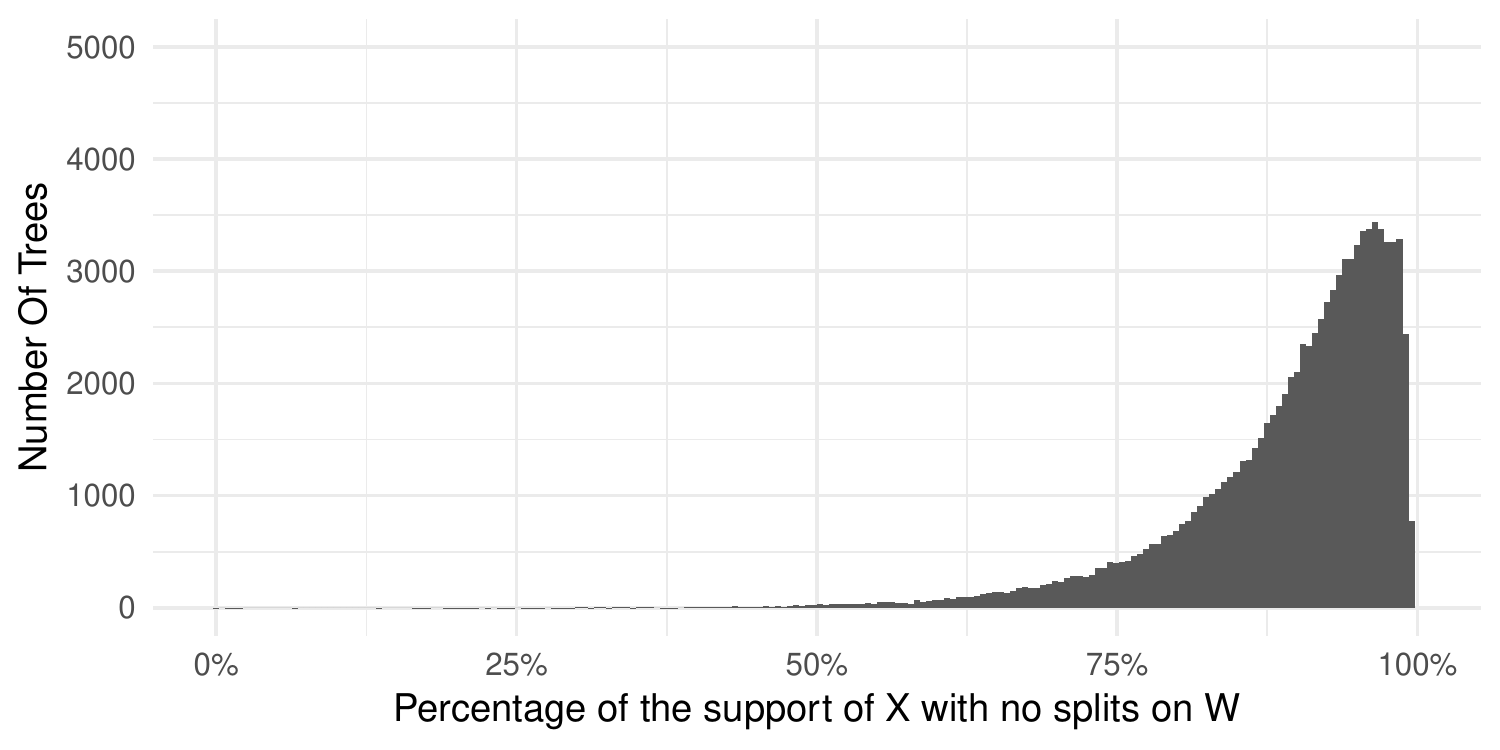}
	\caption{This figure is created from an S--RF learner to show that the S-learner often ignores the treatment effect entirely. It is based on 100,000 trees and it shows the histogram of trees by what percentage of the support of $X$ is not split on $W$.}
	\label{fig:snosplit}
\end{figure}

\section{Adaptivity to Different Settings and Tuning} \label{section:adaptivity}
Tuning the base learners to receive better CATE estimators or even selecting the best CATE estimator from a finite set of CATE estimators is very difficult, and our recent R package, \texttt{hte}, attempts to implement some tuning and selection methods. 
This is, however, very difficult and in the preceding sections, we did not tune our random forest algorithm or our BART estimators on the given data sets. Instead, we used fixed hyperparameters that were chosen in a different simulation study. 
In the sequel, we show that tuning the base learners and being able to select the best meta-learner can be very beneficial to constructing a good CATE estimator.

We conduct a simple experiment showing the potential benefits of hyperparameter tuning of the base learners. Specifically, we evaluate S--RF, T--RF, and X--RF in Simulations \ref{sim:globallin} and \ref{sim:complexlin}. We sample 1,000 hyperparameter settings for each of the learners and evaluate them in both simulations. In other words, for each hyperparameter setting, we obtain an MSE for Simulation \ref{sim:globallin} and an MSE for Simulation \ref{sim:complexlin}. 

Figure \ref{fig:tuningsituation} shows the MSE pairs. 
As expected, we observe that the T-learner generally does very well when the treatment effect is complex, while it does rather poorly when the treatment effect is simple. This was expected as the T-learner generally performs poorly compared to the S-learner when the treatment effect is simple or close to 0. 
Also as expected, the S-learner performs well when the treatment effect is simple, but it performs relatively poorly compared to the T-learner when the treatment effect is complex. 
The X-learner, on the other hand, is extremely adaptive. In fact, depending on the set of hyperparameters, the X-learner can perform as well as the T-learner or the S-learner.
However, there is not a single set of parameters that is optimal for both settings. In fact, the optimal settings almost describe a utility curve.

\begin{figure*}[h]
	\centering
	\includegraphics[width=.5\linewidth]{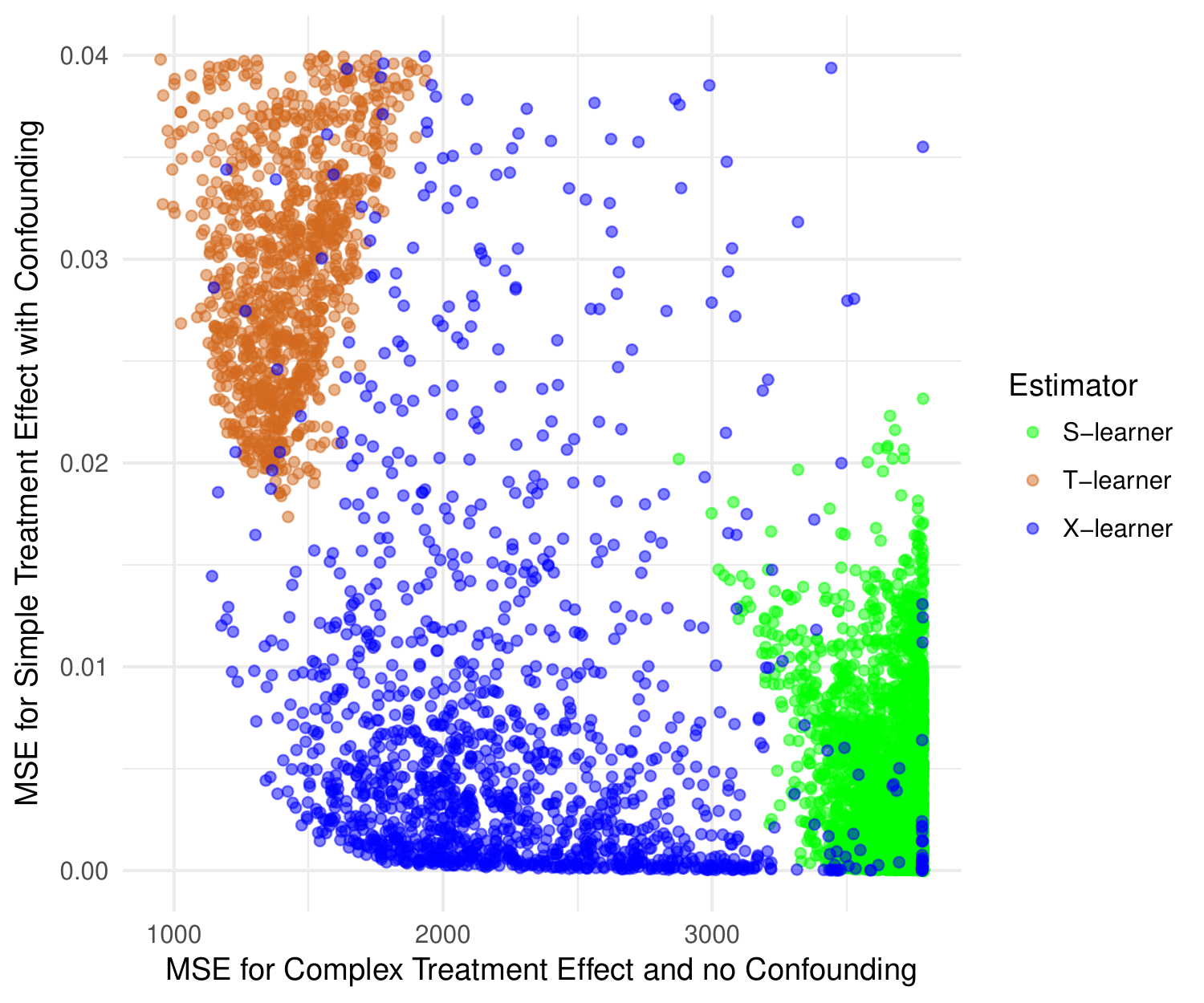}
	\caption{ 
		Each point corresponds to a different hyperparameter setting in random forests as the base learner in one of the S-, T-, or X-learners. The y-axis value is the MSE of Simulation \ref{sim:globallin} and the x-axis value is the MSE in Simulation \ref{sim:complexlin}. A perfect estimator that gets an MSE error of 0 in both simulations would thus correspond to a point at the origin (0,0). The training set size had 1,000 units and the test set that was used to estimate the MSE had 10,000 units.}
	\label{fig:tuningsituation}
\end{figure*}

\subsection{Setting the Tuning Parameters}
Since tuning each algorithm for each data set separately turns out to be very challenging, we decided to hold the hyperparameters fix for each algorithm. 
To chose those preset hyperparameters, we used the 2016 Atlantic Causal Inference Conference competition \cite{dorie2017automated}, and we chose the parameters in such a way that the algorithms perform very well in this competition. 
Specifically, we randomly generated for each algorithm 10,000 hyperparameters. We then evaluated the performance of these 10,000 hyperparameter settings on the 20 data sets of the ``Do it yourself!''-challenge, and we chose the hyperparameter combination which did best for that challenge.

\section{Conditioning on the Number of Treated Units} \label{section:conditioning_n}
In our theoretical analysis, we assume a superpopulation and we condition on the number of treated units both to avoid the problem that with a small but non-zero probability all units are in the treatment group or the control group and to be able to state the performance of different estimators in terms of $n$, the number of treated units, and $m$, the number of control units. This conditioning, however, leads to nonindependent samples.
The crucial step in dealing with this dependent structure is to condition on the treatment assignment, $W$.

Specifically, there are three models to be considered. 
\begin{enumerate}
	\item 
		The first one is defined by \refb{model:basic}. It specifies a distribution, $\Pcal$, of $(X, W, Y)$, and we assume to observe $N$ independent samples from this distribution, 
		$$
			\left(X_i, W_i, Y_i\right)_{i=1}^N \overset{iid}\sim \Pcal.
		$$
		We denote the joint distribution of $\left(X_i, W_i, Y_i\right)_{i=1}^N$ by $\Pcal^N$. 
	\item 
		We state our technical results in terms of a conditional distribution. For a fixed $n$ with $0 < n < N$, we consider the distribution of $\left(X_i, W_i, Y_i\right)_{i=1}^N$ given that we observe $n$ treated units and $m = N - n$ control units. We denote this distribution by $\Pcal^{nm}$.
		$$
			\left[\left(X_i, W_i, Y_i\right)_{i=1}^N \bigg | \sum_{i = 1}^N W_i = n \right] \sim \Pcal^{nm}.
		$$
		Note that under $\Pcal^{nm}$ the $\left(X_i, W_i, Y_i\right)$ are identical in distribution, but not independent.
	\item 	
		For technical reasons, we also introduce a third distribution, which we will use only in some of the proofs. Here, we condition on the vector of treatment assignments, $W$. 
		$$
			\left[\left(X_i, W_i, Y_i\right)_{i=1}^N \big | W = w \right] \sim \Pcal^{w}.			
		$$
		Under this distribution $W$ is non-random and $(X_i, Y_i)$ are not identical in distribution. However, within each treatment group the $(X_i, Y_i)$ tuples are independent and identical in distribution. 
		To make this more precise, define $\Pcal_1$ to be the conditional distribution of $(X, Y)$ given $W = 1$; then, under $\Pcal^w$, we have
		$$
			\left(X_i, Y_i\right)_{W_i = 1} \overset{iid} \sim \Pcal_1. 
		$$
		We prove these facts as follows. 
\end{enumerate}

\begin{theorem} \label{Theorem:CondIndi}
		Let $n$ and $N$ be such that $0 < n < N$ and let $w \in \{0,1\}^N$ with $\sum_{i=1}^N w_i = n$. 
		Then, under the distribution $\Pcal^w$, 
		$$
			(X_k, Y_k)_{W_k = 1} \overset{iid}{\sim} \Pcal_1.
		$$
\end{theorem}
We prove this in two steps. In Lemma \ref{muindi}, we prove that the distributions are independent and in Lemma \ref{iddis} we prove that they are identical. 

\begin{lemma}[independence] \label{muindi}
		Let $n$, $N$, and $w$ be as in Theorem \ref{Theorem:CondIndi} and define 
		$ S = \{j \in \N : w_j = 1\}$.
		Then for all $\emptyset \neq \I \subset S$, and all $(B_i)_{i \in \I}$ with $B_i \subset \R^p \times \R,$
		\begin{equation} \label{independence}
		\P\left(\bigcap_{i \in \I} \{(X_i, Y_i) \in B_i \}\bigg | W = w\right) 
		= 
		\prod_{i \in \I} \P\left((X_i, Y_i) \in B_i  \bigg | W = w\right).
		\end{equation}
\end{lemma}
Note that another way of writing \refb{independence} is 
\begin{equation}
	\P^w\left(\bigcap_{i \in \I} \{(X_i, Y_i) \in B_i \}\right) 
	= 
	\prod_{i \in \I} \P^w\left((X_i, Y_i) \in B_i\right).
\end{equation}

\begin{proof}[Proof of Lemma \ref{muindi}]
	\begin{align*}
	&\P\left(\bigcap_{i \in \I} \{(X_i, Y_i) \in B_i \} \bigg | W = w\right) 
	\\&=
	\P\left(\Big(\bigcap_{i \in \I} \{(X_i, Y_i) \in B_i \} \Big) 
	~  \cap ~ 
	\Big( \bigcap_{j \in S} \{W_j = 1\} \cap \bigcap_{k \in S^c} \{W_k = 0\} \Big) \right) 
	/~
	\P\Big(W = w \Big)
	\\&=
	\P\left(\Big(\bigcap_{i \in \I} \big\{(X_i, Y_i, W_i) \in B_i \times \{1\} \big\}\Big) 
	~  \cap ~ 
	\Big(\bigcap_{j \in S \setminus \I} \{W_j = 1\} \cap \bigcap_{k \in S^c} \{W_k = 0\} \Big) \right) 
	/
	\P\Big(W = w \Big)
	\\&=
	\prod_{i \in \I} \P\Big((X_i, Y_i, W_i) \in B_i \times \{1\}\Big) 
	\frac{
		\P \Big( \bigcap_{j \in S \setminus \I} \{W_j = 1\} \cap \bigcap_{k \in S^c} \{W_k = 0\} \Big)
	}{
		\P\Big(W = w \Big)
	} = (*).
	\end{align*}
	The last equality holds because $(X_i, Y_i, W_i)_{i = 1}^N$ are mutually independent. 
	The second term can be rewritten in the following way:
	\begin{align*}
	\frac{\P 
		\Big( \bigcap_{j \in S \setminus \I} \{W_j = 1\} \cap \bigcap_{k \in S^c} \{W_k = 0\} \Big)
	}{
		\P\Big(W = w \Big) 
	}
	&=
	\frac{
		\prod_{j \in S \setminus \I} \P(W_j = 1) \prod_{k \in S^c} \P(W_k = 0)
	}{\prod_{j \in S} \P(W_j = 1) \prod_{k \in S^c} \P(W_k = 0)}
	\\&=
	\prod_{j \in J}
	\frac{1
	}{\P(W_j = 1)}
	\\&=
	\prod_{j \in J}
	\frac{
		\prod_{j \in S \setminus \{j\}} \P(W_j = 1) \prod_{k \in S^c} \P(W_k = 0)
	}{\prod_{j \in S} \P(W_j = 1) \prod_{k \in S^c} \P(W_k = 0)}
	\\&=
	\prod_{i \in \I} 
	\frac{
		\P \bigg[
		\bigcap_{j \in S \setminus \{i \}} \{W_j = 1\} \cap \bigcap_{k \in S^c} \{W_k = 0\}
		\bigg]
	}{
		\P\Big[W = w \Big]
	}.
	\end{align*}

	Thus, 
	\begin{align*}
	(*)&=
	\prod_{i \in \I} 
	\P\bigg[
	(X_i, Y_i, W_i) \in B_i \times \{1\}  
	\bigg]
	\prod_{i \in \I} 
	\frac{
		\P \bigg[
		\bigcap_{j \in S \setminus \{i \}} \{W_j = 1\} \cap \bigcap_{k \in S^c} \{W_k = 0\}
		\bigg]
	}{
		\P\Big[W = w \Big]
	}
	\\&=
	\prod_{i \in \I} 
	\bigg(
	\P\bigg[
	(X_i, Y_i, W_i) \in B_i \times \{1\}  
	~  \cap ~ 
	\Big( \bigcap_{j \in S \setminus \{i \}} \{W_j = 1\} \cap \bigcap_{k \in S^c} \{W_k = 0\} \Big)
	\bigg]
	/~
	\P\Big[W = w \Big]
	\bigg)
	\\&=
	\prod_{i \in \I} 
	\bigg(
	\P\left((X_i, Y_i) \in B_i  ~  \cap ~ \Big\{W = w \Big\}\right)
	/~
	\P\Big(W = w \Big)
	\bigg)
	\\&
	\prod_{i \in \I} \P\left((X_i, Y_i) \in B_i  \bigg | W = w\right),
	\end{align*}
	which completes the proof.
\end{proof}

Next, we are concerned with showing that all treated units have the same distribution. 
\begin{lemma}[identical distribution] \label{iddis}
	Assume the same assumptions as in Lemma \ref{muindi} and let $i \neq j \in S$. Under the conditional distribution of $W = w$, $(X_i, Y_i)$ and $(X_j, Y_j)$ have the same distribution, $\Pcal_1$.
\end{lemma}
\begin{proof} 
	Let $B \subset \R^p \times \R$; then 
	\begin{align*}
	\P\left((X_i, Y_i) \in B \bigg| W = w\right) 
	&\overset{*}=
	\P\left((X_i, Y_i) \in B \big|W_i = 1\right) 
	\\&=
	\frac{
		\P\left((X_i, Y_i, W_i) \in B \times \{1\}\right) 
	}{
		\P(W_i = 1) 
	}
	\\&\overset{a}=
	\frac{
		\P\left((X_j, Y_j, W_j) \in B \times \{1\}\right) 
	}{
		\P(W_j = 1) 
	}
	\\& = 
	\P\left((X_j, Y_j) \in B \big|W_j = 1\right) 
	\\&\overset{*}=
	\P\left((X_j, Y_j) \in B \bigg| W = w\right).
	\end{align*}
	Here $\overset{*}~$ follows from $(X_i, Y_i, W_i)_{i = 1}^N$ being mutually independent, and $\overset{a}~$ follows from $(X_i, Y_i, W_i)_{i = 1}^N$ being identically distributed under $\Pcal$. 
\end{proof}

\section{Convergence Rate Results for the T-learner} \label{Conv:Tlearner}
In this section, we want to prove that 
$$
\EMSE(\hat\tau_T^{mn}, \Pcal) 
\le 
C(n^{-a_\mu} + m^{-a_\mu}).
$$
We start with a short lemma that will be useful for the proof of the theorem. 
\begin{lemma} \label{Lemma:DistChange}
	Let $\Pcal$ be defined as in \refb{model:basic} with $0 < \emin < e(x) < \emin < 1$. 
	Furthermore, let $X, W$ be distributed according to $\Pcal$, and let $g$ be a positive function such that the expectations below exist; then 
	\begin{align}
		\frac{\emin}{\emax}\E[g(X)] ~\le~ &\E[g(X) | W = 1] ~\le~ \frac{\emax}{\emin}\E[g(X)], \label{Ineq:W1}
	   	\\\frac{1 - \emax}{1 - \emin}\E[g(X)] ~\le~ &\E[g(X) | W = 0]  ~\le~ \frac{1 - \emin}{1-\emax}\E[g(X)]. \label{Ineq:W0}
	\end{align}
\end{lemma}
\begin{proof}[Proof of Lemma \ref{Lemma:DistChange}]
	Let us prove \refb{Ineq:W1} first. The lower bound follows from
	\begin{align*}
	\E[g(X) | W = 1] 
	\ge 
	\E[g(X)] \frac{\inf_x e(x)}{E[W]} 
	\ge 
	\frac{\emin}{E[W]} \E[g(X)] 
	\ge 
	\frac{\emin}{\emax} \E[g(X)],
	\end{align*}
	and the upper bound from
	\begin{align*}
		\E[g(X) | W = 1] 
		\le 
		\E[g(X)] \frac{\sup_x e(x)}{E[W]}
		\le 
		\frac{\emax}{\emin} \E[g(X)].
	\end{align*}
	\refb{Ineq:W0} follows from a symmetrical argument. 
\end{proof}

Let $m, n \in \mathbb{N}^+$ and $N = m+ n$ and let $\Pcal$ be a distribution of $(X, W, Y)$ according to \refb{model:basic} with the propensity score bounded away from 0 and 1. That is, there exists $\emin$ and $\emax$ such that $0 < \emin < e(x) < \emax < 1$. 
Furthermore, let $(X_i, W_i, Y_i)_{i = 1}^N$ be i.i.d. from $\Pcal$ and define $\Pcal^{nm}$ to be the conditional distribution of $(X_i, W_i, Y_i)_{i = 1}^N$ given that we observe $n$ treated units,
$
	\sum_{i = 1}^N W_i = n.
$

Note that $n$ and $m$ are not random under $\Pcal^{nm}$. 
We are interested in the performance of the T-learner, $\hat\tau_T^{mn}$, under $\Pcal^{nm}$ as measured by the EMSE,
$$
	\EMSE(\hat\tau_T^{mn}, \Pcal^{nm}) 
	\overset{\mbox{def}}=
	\E 
	\left[
	(\hat \tau_T^{mn}(\X) - \tau(\X))^2 \bigg| \sum_{i = 1}^N W_i = n
	\right].
$$
The expectation is here taken over the training data set $(X_i, W_i, Y_i)_{i=1}^N$, which is distributed according to $\Pcal^{nm}$, and $\X$, which is distributed according to the marginal distribution of $X$ in $\Pcal$.

For a family of superpopulations, $F \in S(a_\mu, a_\tau)$, we want to show that the T-learner with an optimal choice of base learners achieves a rate of  
$$
\mathcal{O}(m^{-a_\mu} + n^{-a_\mu}).
$$
An optimal choice of base learners is estimators that achieve the minimax rate of $n^{-a_\mu}$ and $m^{-a_\mu}$ in $F$. 

\begin{proof}[Conferegence Rate of the T-learner]
	The EMSE can be upper bounded by the errors of the single base learners:
	\begin{align*}
	\EMSE(\hat\tau_T^{mn}, \Pcal^{nm}) &=
	\E 
	\left[
	(\hat \tau_T^{mn}(\X) - \tau(\X))^2 \bigg| \sum_{i = 1}^N W_i = n
	\right]
	\\&\le
	2
	\underbrace{
		\E 
		\left[
		(\hat \mu^n_1(\X) - \mu_1(\X))^2 \bigg| \sum_{i = 1}^N W_i = n
		\right]
	}_A
	+ 
	2
	\underbrace{
		\E 
		\left[
		(\hat \mu^m_0(\X) - \mu_0(\X))^2 \bigg| \sum_{i = 1}^N W_i = n
		\right]
	}_B.
	\end{align*}
	Here we use the following inequality:
	$$
	(\hat \tau_T^{mn}(\X) - \tau(\X))^2 \le 2 (\hat \mu^n_1(\X) - \mu_1(\X))^2 + 2(\hat \mu^m_0(\X) - \mu_0(\X))^2.
	$$
	Let us look only at the first term. 
We can write
	\begin{align}
	A &= \E 
	\left[
	(\hat \mu^n_1(\X) - \mu_1(\X))^2 \bigg| \sum_{i = 1}^N W_i = n
	\right] \notag
	\\&=
	\E 
	\left[
	\E 
	\left[
	(\hat \mu^n_1(\X) - \mu_1(\X))^2 \bigg|  W, \sum_{i = 1}^N W_i = n
	\right]
	\bigg| \sum_{i = 1}^N W_i = n
	\right]. \label{iteratedExpectation}
	\end{align}
	It is of course not necessary to condition on $\sum_{i = 1}^N W_i = n$ in the inner expectation, and we only do so as a reminder that there are $n$ treated units. 
	
	For $i \in \{1, \ldots, n\}$, let $q_i$ be the $i^{\text{th}}$ smallest number in $\{k : W_k = 1\}$. That is, $\{q_i : i \in \{1, \ldots, n\}\}$ are the indexes of the treated units.
	To emphasize that $\hat \mu^n_1(\X)$ depends only on the treated observations, $(X_{q_i}, Y_{q_i})_{i= 1}^n$, we write
	$\hat \mu^n_1((X_{q_i}, Y_{q_i})_{i = 1}^n ,\X)$.
	Furthermore, we define $\Pcal_1$ to be the conditional distribution of $(X, Y)$ given $W=1$. Conditioning on $W$, Theorem \ref{Theorem:CondIndi} implies that $(X_{q_i}, Y_{q_i})_{i= 1}^n$ is i.i.d. from $\Pcal_1$.
	Let us define $\tilde \X$ to be distributed according to $\Pcal_1$. Then we can apply Lemma \ref{Lemma:DistChange} and use the definition of $S(a_\mu, a_\tau)$ to conclude that the inner expectation in \refb{iteratedExpectation} is in $\mathcal{O}(n^{-a_\mu})$:
	\begin{align*}
	\E 
	&\left[
	\hat \mu^n_1((X_{q_i}, Y_{q_i})_{i = 1}^n ,\X)- \mu_1(\X))^2 \bigg| W, \sum_{i = 1}^N W_i = n
	\right]
	\\&\le 
	\frac{\emax}{\emin}
	\E 
	\left[
	(\hat \mu^n_1((X_{q_i}, Y_{q_i})_{i = 1}^n ,\tilde\X) - \mu_1(\tilde\X))^2 \bigg| W, \sum_{i = 1}^n W_i = n
	\right]
	\\&\le 
	\frac{\emax}{\emin} C n^{-a_\mu}.
	\end{align*}
	Hence, it follows that
	$$
	A \le 
	2\E
	\left[
	\frac{\emax}{\emin} C n^{-a_\mu}
	\bigg| \sum_{i = 1}^n W_i = n
	\right] \le 
	2\frac{\emax}{\emin} C n^{-a_\mu}.
	$$
	
	By a symmetrical argument, it also holds that 
	$$
	B \le 
	2\frac{1 - \emin}{1 - \emax} C m^{-a_\mu},
	$$
	and we can conclude that
	$$
	\EMSE(\hat\tau_T^{mn}, \Pcal) 
	\le 
	2 C 
	\left[\frac{1 - \emin}{1 - \emax} +  \frac{\emax}{\emin} \right](n^{-a_\mu} + m^{-a_\mu}).
	$$
\end{proof}    

\section{Convergence Rate Results for the X-learner}
In this section, we are concerned with the convergence rate of the X-learner. Given our motivation of the X-learner in the main paper, we believe that $\hat \tau_0$ of the X-learner should achieve a rate of $\mathcal{O}(m^{-a_\tau} + n^{-a_\mu})$ and $\hat \tau_1$ should achieve a rate of $\mathcal{O}(m^{-a_\mu} + n^{-a_\tau})$. 
In what follows, we prove this for two cases, and we show that for those cases the rate is optimal. 
In the first case, we assume that the CATE is linear and thus $a_\tau = 1$. We don't assume any regularity conditions on the response functions, and we show that the X-learner with an OLS estimator in the second stage and an appropriate estimator in the first stage achieves the optimal convergence rate. We show this first for the MSE (Theorem \ref{theorem:unbalanced_ptwise})  and then for the EMSE (Theorem \ref{theorem:lineartau}).
We then focus on the case where we don't impose any additional regularity conditions on the CATE, but the response functions are Lipschitz continuous (Theorem \ref{Theorem:Lipschitz_Convergence}). 
The optimal convergence rate is here not obvious, and we will first prove a minimax lower bound for the EMSE, and we will then show that the X-learner with the KNN estimates achieves this optimal performance.

\subsection{MSE and EMSE convergence rate for the linear CATE} \label{App:Conv} \label{section:conv:linear}

\begin{theorem}[rate for the pointwise MSE] \label{theorem:unbalanced_ptwise}
	Assume that we observe $m$ control units and $n$ treated units from some superpopulation of independent and identically distributed observations $(Y(0), Y(1), X, W)$ coming from a distribution $\Pcal$ given in equation [\ref{model:basic}] and assume that the following assumptions are satisfied:
	\begin{enumerate}[label=B\arabic*]
		\item Ignorability holds.
		\item The treatment effect is linear, $\tau(x) = x^T \beta$, with $\beta \in \R^d$.
		\item  
		There exists an estimator $\hat \mu_0$ such that for all $x$, 
		$$
			\E\left[(\mu_0(x) - \hat \mu^m_0(x))^2 \bigg | \sum_{i=1}^N W_i = n \right] \le C^0 m^{-a}.
		$$
		\item The error terms $\varepsilon_i$ are independent given $X$, with $\E[\varepsilon_i|X = x] = 0$ and $\var[\varepsilon_i|X = x] \le \sigma^2 < \infty$.
		\item \label{theorem:lineartau_ptw:goodEV} The eigenvalues of the sample covariance matrix of the features of the treated units are well conditioned, in the sense that there exists an $n_0$, such that 
		\begin{align} 
		\sup_{n >n_0}\E\left[\mie^{-1}(\hat \Sigma_n)\bigg | \sum_{i=1}^N W_i = n\right]  < c_1& 
		&\mbox{and}& 
		&
		\sup_{n >n_0}
		\E\left[\mae(\hat \Sigma_n) / \mie^{2}(\hat \Sigma_n)\bigg | \sum_{i=1}^N W_i = n\right] < c_2,
		\end{align}
		where 
		$
			\hat \Sigma_n = \frac1n(X^1)'X^1
		$
		and $X^1$ is the matrix consisting of the features of the treated units.
	\end{enumerate}
	Then the X-learner with $\hat \mu_0$ in the first stage, OLS in the second stage, and weighting function $g\equiv 0$ has the following upper bound: for all $x \in \R^d$ and all $n> n_0$,
	\begin{equation} \label{theorem:lineartau_ptw:main}
		\E\left[\left(\tau(x) - \hat \tau_X(x)\right)^2 \bigg| \sum_{i=1}^N W_i = n\right] \le C_x \left(m^{-a} + n^{-1}\right)
	\end{equation}
	with $C_x = max(c_2 C^0, \sigma^2 d c_1) \|x\|^2$.

\end{theorem}

\begin{proof}[Proof of Theorem \ref{theorem:unbalanced_ptwise}]
	To simplify the notation, we write $X$ instead of $X^1$ for the observed features of the treated units.
	Furthermore, we denote that when $g \equiv 0$ in [\ref{equ:3rdStep}] in the main paper, the X-learner is equal to $\hat \tau_1$ and we only have to analyze the performance of $\hat \tau_1$.
	
	The imputed treatment effects for the treatment group can be written as
	$$
	D_i^1 = Y_i - \hat \mu_0(X_i) =  X_i\beta + \delta_i + \epsilon_i,
	$$
	with $\delta_i = \mu_0(X_i) - \hat \mu_0(X_i)$.
	In the second stage we estimate $\beta$ using an OLS estimator,
	$$
	\hat \beta = (X'X)^{-1} X' D^1.
	$$
	To simplify the notation, we define the event of observering $n$ treated units as $E_n = \{\sum_{i=1}^N W_i = n\}$.
	We decompose the MSE of $\hat \tau(x)$ into two orthogonal error terms:
	\begin{equation}
	\label{equ:MSE_xdecomp}
	\begin{aligned}
	\E\left[(\tau(x) - \hat \tau_X(x))^2\bigg | \sum_{i=1}^N W_i = n\right] 
	= 
	\E\left[(x' (\beta - \hat \beta))^2    \Big | E_n\right]
	&\le
	\|x\|^2 \E \left[
	\| (X' X)^{-1} X' \delta\|^2                
	+
	\| (X' X)^{-1} X' \varepsilon\|^2                
	\Big |  E_n\right]
	.
	\end{aligned}
	\end{equation}
	Throughout the proof, we assume that $n > n_0$ such assumption \ref{theorem:lineartau_ptw:goodEV} can be used.
	We will show that the second term decreases at the parametric rate, $n^{-1}$, while the first term decreases at a rate of $m^{-a}$:
	\begin{equation}\label{equ:bounding_epsilon}
	\begin{aligned}
	\E \left[\|(X' X)^{-1} X' \varepsilon\|^2\Big |  E_n\right]
	&= 
	\E\left[
	tr\left( X (X' X)^{-1} (X' X)^{-1} X' \E\left[\varepsilon \varepsilon'\big |X, E_n\right]\right)                
	\Big |  E_n\right]            
	\\&\le 
	\sigma^2 
	d
	\E\left[
	\mie^{-1}(\hat \Sigma_n)
	\Big |  E_n\right]
	n^{-1}
	\\&\le 
	\sigma^2 d c_1 n^{-1}.
	\end{aligned}
	\end{equation}
	For the last inequality we used assumption \ref{theorem:lineartau_ptw:goodEV}.
	Next, we are concerned with bounding the error coming from not perfectly predicting $\mu_0$:
	\begin{equation}\label{equ:bounding_delta}
	\begin{aligned}
	\E\left[   \|(X' X)^{-1} X' \delta\|_2^2 \Big |  E_n\right]   
	&\le
	\E\left[   \mae(\hat \Sigma_n)/ \mie^{2}(\hat \Sigma_n) \|\delta\|_2^2    \Big |  E_n\right]    n^{-1}
	\\&\le
	\E\left[\mae(\hat \Sigma_n)/ \mie^{2}(\hat \Sigma_n)\Big |  E_n\right]    C^0 m^{-a}
	\\&\le 
	c_2 C^0 m^{-a}.
	\end{aligned}
	\end{equation}
	Here we used that 
	$
	\mae(\hat \Sigma_n^{-2}) = \mie^{-2}(\hat \Sigma_n)
	$,
	and
	$
	\E\left[
	\|\delta\|_2^2
	\Big| X, E_n\right]
	= 
	\E\left[
	\sum_{i = 1}^{n}  
	\delta^2(X_i)
	\Big| X, E_n\right]
	\le n C^0 m^{-a}
	$. 
	For the last statement, we used assumption \ref{theorem:lineartau_ptw:goodEV}. 
	This leads to [\ref{theorem:lineartau_ptw:main}].
\end{proof}


\subsubsection*{Bounding the EMSE}

\begin{proof}[Proof of Theorem \ref{theorem:lineartau}]
	This proof is very similar to the proof of Theorem \ref{theorem:unbalanced_ptwise}. The difference is that here we bound the EMSE instead of the pointwise MSE, and we have a somewhat weaker assumption, because $\hat \mu_0$ only satisfies that its EMSE converges at a rate of $a$, but not necessarily the MSE at every $x$. 
	We introduce $\X$ here to be a random variable with the same distribution as the feature distribution such that the EMSE can be written as $\E[(\tau(\X) - \hat \tau_X(\X))^2|E_n]$. Recall that we use the notation that $E_n$ is the event that we observe exactly $n$ treated units and $m = N - n$ control units:
	$$
		E_n =\left\{\sum_{i = 1}^N W_i = n \right\}.
	$$
	We start with a similar decomposition as in [\ref{equ:MSE_xdecomp}]:
	\begin{equation} \label{equ:decomp}
		\begin{aligned}
			\E\left[(\tau(\X) - \hat \tau_X(\X))^2 \big | E_n \right] 
			&\le
			\E\left[ \|\X\|^2  \right]  
			\E\left[ \|\beta - \hat \beta\|^2 \big | E_n \right] 
			\\&=
			\E\left[ \|\X\|^2\right]  
			\E \left[\| (X' X)^{-1} X' \delta\|^2                
			+
			\| (X' X)^{-1} X' \varepsilon\|^2                
			\big | E_n \right].
		\end{aligned}
	\end{equation}
	Following exactly the same steps as in [\ref{equ:bounding_epsilon}], we get
	$$
		\E \left[\|(X' X)^{-1} X' \varepsilon\|^2\big | E_n \right]  \le \sigma^2 d C_\Sigma n^{-1}.
	$$
	Bounding 
	$
			\E\left[    \|(X' X)^{-1} X' \delta\|_2^2 \big | E_n \right] 
	$ 
	is now slightly different than in [\ref{equ:bounding_delta}]:
	\begin{equation} \label{eq:boundingdelta}
		\begin{aligned}
		\E\left[    \|(X' X)^{-1} X' \delta\|_2^2 \big | E_n \right] 
		&\le
		\E\left[ \mie^{-1}(X'X) \|X(X' X)^{-1} X' \delta\|_2^2 \big | E_n \right]             
		\\&\le
		\E\left[ \mie^{-1}(X'X) \|\delta\|_2^2 \big | E_n \right]         
		\\&\le
		\E\left[ \mie^{-1}(\Sigma_n) \frac1n \|\delta\|_2^2 \big | E_n \right]                     
        \\&\le
	    C_\Sigma \E\left[ \|\delta_1\|_2^2 \big | E_n \right].       
		\end{aligned}
	\end{equation}
	Here the last inequality follows from Condition \ref{theorem:lineartau:goodEV}.
	
	We now apply \refb{Ineq:W1}, \refb{Ineq:W0}, and Condition \ref{condition:controlrate} to conclude that
	\begin{align*}
		\E\left[ \|\delta_1\|_2^2 \big | E_n \right] 
		&= 
		\E\left[ \|\mu_0(X_1) - \hat\mu_0(X_1)\|_2^2 \big | E_n, W_1 = 1\right] 
		\\&\le 
		\frac{\emax - \emax\emin}{\emin - \emax\emin} \E\left[ \|\mu_0(X_1) - \hat\mu_0(X_1)\|_2^2 \big | E_n, W_1 = 0\right]
		\\& \le 
		\frac{\emax - \emax\emin}{\emin - \emax\emin} C_0 m^{-a_\mu}.
	\end{align*}
	
	Lastly, we use the assumption that $\E\left[\|\X\|^2\big | E_n \right]  \le C_\X$ and conclude that
	\begin{equation}\label{emse:rate}
		\E\left[(\tau(\X) - \hat \tau_X(\X))^2\big | E_n \right]  
		\le
		C_\X \left(\frac{\emax - \emax\emin}{\emin - \emax\emin} C_\Sigma C_0 m^{-a} + \sigma^2 d C_\Sigma n^{-1}\right).
	\end{equation}

\end{proof}


\subsection{Achieving the parametric rate}
	When there are a lot of control units, such that $m \geq n^{1/a}$, then we have seen that the X-learner achieves the parametric rate. 
	However, in some situations the X-learner also achieves the parametric rate even if the number of control units is of the same order as the number of treated units.
	To illustrate this, we consider an example in which the conditional average treatment effect and the response functions depend on disjoint and independent subsets of the features. 
	
	Specifically, we assume that we observe $m$ control units and $n$ treated units according to Model \ref{model:basic}. 
	We assume the same setup and the same conditions as in Theorem \ref{theorem:lineartau}.
	In particular, we assume that there exists an estimator $\hat \mu^m_0$ that depends only on the control observations and estimates the control response function at a rate of at most $m^{-a}$.
	In addition to these conditions we also assume the following independence condition.
	\setcounter{condition}{6}
	\begin{condition} \label{Cond:Ssparse}
		There exists subsets, $S, \Sb \subset \{1, \ldots, d\}$ with $S \cap \Sb = \emptyset$, such that
		\begin{itemize}
			\item $(X_i)_{i \in S}$ and $(X_i)_{i\in\Sb}$ are independent.
			\item For all $i \in S$, $E[X_i | W_i = 1] = 0$.
			\item There exist a function $\tilde \mu_0$, and a vector $\beta$ with $\mu_0(x) = \tilde \mu_0(x_\Sb)$ and $\tau(x) = x_S^T \tilde \beta$.
		\end{itemize}
	\end{condition}
	
	For technical reasons, we also need bounds on the fourth moments of the feature vector and the error of the estimator for the control response. 
	\begin{condition} \label{Cond:boundSigma}
		The fourth moments of the feature vector $X$ are bounded:
		$$
		\E[\|X\|_2^4| W = 1] \le C_X.
		$$
	\end{condition}
	
	\begin{condition} \label{Cond:boundeddelta}
		There exists an $m_0$ such that for all $m > m_0$, 
		$$\E\left[(\mu_0(X) - \hat \mu^m_0(X))^4 \bigg| W = 1 \right] \le C_\delta.$$
		Here $\hat \mu^m_0$ is defined as in Condition \ref{condition:controlrate}.
	\end{condition}	
	This condition is satisfied, for example, when $\mu_0$ is bounded. 
	
	Under these additional assumptions, the EMSE of the X-learner achieves the parametric rate in $n$, given that $m > m_0$.
	\begin{theorem} \label{Theo:Example_semipara}
		Assume that Conditions \ref{equation:conditionalUnconf}--\ref{Cond:boundeddelta} hold.
		Then the X-learner with $\hat \mu^m_0$ in the first stage and OLS in the second stage achieves the parametric rate in $n$. That is, there exists a constant $C$ such that for all $m > m_0$ and $n > 1$, 
		$$
		\E\left[(\tau(\X) - \hat \tau_X^{mn}(\X))^2 \Big | \sum_i W_i = n \right]  \le C n^{-1}.
		$$
	\end{theorem}
	We will prove the following lemma first, because it will be useful for the proof of Theorem \ref{Theo:Example_semipara}.
	\begin{lemma} \label{Lemma:bounding}
		Under the assmuption of Theorem \ref{Theo:Example_semipara}, there exists a constant C such that for all $n > n_0$, $m > m_0$, and $s > 0$,
		$$
		\P\left(n\|({X^1}' X^1)^{-1} {X^1}' \delta\|_2^2  \ge s \Big | \sum_i W_i = n \right) \le C\frac{1}{s^2},
		$$
		where $\delta_i = \mu_0(X^1_i) - \hat \mu^m_0(X^1_i)$.
	\end{lemma}
	
	\begin{proof}[Proof of Lemma \ref{Lemma:bounding}]
		To simplify the notation, we write $X$ instead of $X^1$ for the feature matrix of the treated units, and we define the event of observing exactly $n$ treated units as
		$$
		E_n = \left\{\sum_{i = 1}^n W_i = n\right\}.
		$$
		
		We use Condition \ref{theorem:lineartau:goodEV} and then Chebyshev's inequality to conclude that for all $n > n_0$ ($n_0$ is determined by Condition \ref{theorem:lineartau:goodEV}),
		\begin{align*}
		\P\left(n \|(X' X)^{-1} X' \delta\|_2^2  \ge s \Big| E_n \right) 
		&=
		\P\left(\frac1n \|\Sigma_n^{-1} X' \delta\|_2^2 \ge s\Big| E_n \right) 
		\\&\le 
		\P\left(\frac1n \gamma^{-2}_{\min}(\Sigma_n)  \| X' \delta\|_2^2 \ge s\Big| E_n \right) 
		\\& \le
		\E\left[\P\left(\frac1n C_\Sigma^2  \| X' \delta\|_2^2 \ge s\Big| E_n, \delta \right) \Big| E_n\right]
		\\&\le     
		\E\left[        
		\frac{C_\Sigma^4}{s^2 n^2}
		\var\left(\| X' \delta\|_2^2\Big| E_n, \delta \right) 
		\Big| E_n\right].
		\end{align*}
		
		Next we apply the Efron–Stein inequality to bound the variance term:
		$$
		\var\left(\| X' \delta\|_2^2\Big| E_n, \delta \right) 
		\le 
		\frac{1}{2} \sum_{i=1}^{n} \E\left[(f(X)-f(X^{(i)}))^2\Big| E_n, \delta \right].
		$$
		Here $f(x) = \|x' \delta\|_2^2 $, 
		$X^{(i)} = (X_1, \ldots, X_{i-1}, \tilde X_i, X_{i +1}, \ldots, X_n)$,
		and $\tilde X$ is an independent copy of $X$. 
		
		Let us now bound the summands:
		\begin{align*}
		&\E\left[(f(X)-f(X^{(i)}))^2\Big| E_n, \delta \right]
		\\=
		&\E\left[\left(\| X' \delta\|_2^2 - \| X' \delta - (X_i - \tilde X_i) \delta_i\|_2^2 \right)^2\Big| E_n, \delta \right]
		\\=
		&\E\bigg[
		\underbrace{\left(2 \delta' X (X_i - \tilde X_i) \delta_i\right)^2}_{A}
		+ 
		\underbrace{\|(X_i - \tilde X_i) \delta_i\|_2^4}_{B}
		- 
		\underbrace{4 \delta' X (X_i - \tilde X_i) \delta_i \|(X_i - \tilde X_i) \delta_i\|_2^2}_C 
		\Big| E_n, \delta \bigg].
		\end{align*}
		
		Let us first bound $\E[A|E_n, \delta]$:
		\begin{align*}
		\E\left[
		\left(2 \delta' X (X_i - \tilde X_i) \delta_i\right)^2
		\Big| E_n, \delta \right]
		&=
		\E\left[
		4 \sum_{j,k=1}^n 
		\delta_j X'_j (X_i - \tilde X_i) \delta_i
		\delta_k X'_k (X_i - \tilde X_i) \delta_i
		\Big| E_n, \delta \right]    
		\\& \overset{(a)}=
		\E\left[
		4 \sum_{j=1}^n 
		(\delta_j X'_j (X_i - \tilde X_i) \delta_i)^2
		\Big| E_n, \delta \right]    
		\\& \le
		4 \delta_i^4
		(n-1)
		\E\left[ 
		(X'_1 (X_2 - \tilde X_2))^2
		\Big| E_n, \delta 
		\right]    
		+
		4 \delta_i^4
		\E\left[ 
		(X'_1 (X_1 - \tilde X_1))^2
		\Big| E_n, \delta 
		\right]
		\\&\le C_A \delta_i^4 n.
		\end{align*}
		Here 
		$$C_A = 4 \max\left(    \E\left[
		(X'_1 (X_2- \tilde X_2))^2
		\Big| E_n
		\right], \E\left[
		(X'_1 (X_1  - \tilde X_1))^2
		\Big| E_n
		\right]
		\right),
		$$
		which is bounded by Condition \ref{Cond:boundSigma}.
		For equation $(a)$ we used that for $k \neq j$; therefore, we have that either $k$ or $j$ is not equal to $i$. Without loss of generality let $j \neq i$. Then
		\begin{equation}\label{eq:E[X|W=1]=0}
		\begin{aligned}
		E&\left[
		\delta_j X'_j (X_i - \tilde X_i) \delta_i
		\delta_k X'_k (X_i - \tilde X_i) \delta_i
		\Big| E_n, \delta 
		\right]        
		\\&=
		\delta_j
		\E\left[
		\E\left[
		X'_j 
		\Big| W, E_n, \delta 
		\right]    
		\E\left[
		(X_i - \tilde X_i) \delta_i
		\delta_k X'_k (X_i - \tilde X_i) \delta_i
		\Big| W, E_n, \delta 
		\right]    
		\Big| E_n, \delta 
		\right]        
		\\& =0,
		\end{aligned}
		\end{equation}
		because $ \E\left[ X'_j | W, E_n, \delta \right] = 0$ as per the assumption. 
		
		In order to bound $\E[B|E_n, \delta]$, note that all the fourth moments of $X$ are bounded and thus
		\begin{align*}
		\E\left[
		\|(X_i - \tilde X_i) \delta_i\|_2^4
		\Big| E_n, \delta \right]
		&\le
		C_B    
		\delta_i^4.
		\end{align*}
		
		Finally, we bound $\E[C|E_n, \delta]$:
		\begin{align*}
		\E \left[
		4 \delta' X (X_i - \tilde X_i) \delta_i \|(X_i - \tilde X_i) \delta_i\|_2^2 
		\Big| E_n, \delta 
		\right]
		& = 
		\E \left[
		\sum_{j=1}^n  \delta_j X'_{j} (X_i - \tilde X_i) \delta_i \|(X_i - \tilde X_i) \delta_i\|_2^2 
		\Big| E_n, \delta 
		\right]
		\\& = 
		\E \left[
		\delta_i^4 X'_{i} (X_i - \tilde X_i)  \|X_i - \tilde X_i\|_2^2 
		\Big| E_n, \delta 
		\Big| E_n, \delta 
		\right]
		\\&= C_C \delta_i^4,
		\end{align*}
		where the second equality follows from the same argument as in \refb{eq:E[X|W=1]=0}, and the last equality is implied by Condition \ref{Cond:boundSigma}.
		
		Plugging in terms A, B, and C, we have that for all $n > n_0$,
		$$
		\var\left(\| X' \delta\|_2^2 \Big| E_n, \delta \right) 
		\le 
		\frac{1}{2} \sum_{i=1}^{n} E[(f(X, \delta)-f(X^{(i)}, \delta^{(i)}))^2]
		\le C \delta^4 n^2,
		$$
		with $C = C_A + C_B + C_C$.    
		Thus for $n > n_0$,
		\begin{align*}
		\P\left(n \|(X' X)^{-1} X' \delta\|_2^2  \ge s \Big| E_n \right) 
		&\le
		\E\left[        
		\frac{C C_\Sigma^4}{s^2 }
		\delta^4
		\Big| E_n\right]
		\le 
		C C_\Sigma^4 C_\delta \frac{1}{s^2 }.
		\end{align*}
	\end{proof}
	
	\begin{proof}[Proof of Theorem \ref{Theo:Example_semipara}]
		We start with the same decomposition as in \refb{equ:decomp}:
		\begin{align*}
		\E\left[(\tau(\X) - \hat \tau_X^{mn}(\X))^2 \big | E_n \right] 
		&\le
		\E\left[ \|\X\|^2\right]  
		\E \left[\| (X' X)^{-1} X' \delta\|^2                
		+
		\| (X' X)^{-1} X' \varepsilon\|^2                
		\big | E_n \right],
		\end{align*}
		and we follow the same steps to conclude that 
		\begin{align*}
		&\E \left[\|(X' X)^{-1} X' \varepsilon\|^2\big | E_n \right]  \le \sigma^2 d C_\Sigma n^{-1}&
		&\mbox{and}&
		&\E\left[ \|\X\|^2\right] \le C_\X.&
		\end{align*}
		
		
		From Lemma \ref{Lemma:bounding}, we can conclude that there exists a constant $C$ such that
		\begin{align*}
		\lim_{n \rightarrow \infty}
		\E\left[n \|(X' X)^{-1} X' \delta\|_2^2 \big | E_n \right] 
		&=
		\lim_{n \rightarrow \infty, n > n_0} 
		\int_{0}^\infty
		\P\left(n \|(X' X)^{-1} X' \delta\|_2^2  \ge s \Big| E_n \right) 
		ds
		\\&\le
		\lim_{n \rightarrow \infty, n > n_0} 
		\int_{0}^\infty
		\max(1, C\frac{1}{s^2 })
		ds
		\\&\le
		1 + C.
		\end{align*}
		Thus there exists a $\tilde C$ such that for all $n > 1$,
		$$
		\E\left[\|(X' X)^{-1} X' \delta\|_2^2 \big | E_n \right] \le \tilde C n^{-1}.
		$$
	\end{proof}

\subsection{EMSE convergence rate for Lipschitz continuous response functions} \label{section:balanced}

In Section \ref{section:conv:linear}, we considered an example where the distribution of $(Y(0), Y(1), W, X)$ was assumed to be in some family $F \in S(a_\mu, a_\tau)$ with $a_\tau > a_\mu$, and we showed that one can expect the X-learner to outperform the T-learner in this case. 
Now we want to explore the case where $a_\tau \le a_\mu$. 

Let us first consider the case, where $a_\tau < a_\mu$. This is a somewhat artificial case, since having response functions that can be estimated at a rate of $N^{-a_\mu}$  implies that the CATE cannot be too complicated. For example, if $\mu_0$ and $\mu_1$ are Lipschitz continuous, then the CATE is Lipschitz continuous as well, and we would expect $a_\tau \approx a_\mu$. 
Even though it is hard to construct a case with $a_\tau < a_\mu$, we cannot exclude such a situation, and we would expect that in such a case the T-learner performs better than the X-learner.

We therefore believe that the case where $a_\tau \approx a_\mu$ is a more reasonable assumption than the case where $a_\tau < a_\mu$. In particular, we would expect the T- and X-learners to perform similarly when compared to their worst-case convergence rate.
Let us try to back up this intuition with a specific example.
Theorem \ref{theorem:lineartau} already confirms that $\hat \tau_1$ achieves the expected rate, 
$$ 
\mathcal{O}\left( m^{-a_\mu} +n^{-a_\tau} \right),
$$ 
for the case where the CATE is linear. 
Below, we consider another example, where the CATE is of the same order as the response functions. We assume some noise level $\sigma$ that is fixed, and we start by introducing a family $F^L$ of distributions with Lipschitz continuous regression functions.
\begin{definition}[Lipschitz continuous regression functions] \label{def:lipschitzRegression}
	Let $F^{L}$ be the class of distributions on
	$(X, Y) \in [0,1]^{d} \times \R$ such that:
	\begin{enumerate}
		\item
		The features, $X_i$, are i.i.d. uniformly distributed in $[0,1]^d$.
		\item The observed outcomes are given by
		$$
		Y_i = \mu(X_i) + \varepsilon_i,
		$$ 
		where the $\varepsilon_i$ is independent and normally distributed with mean 0 and variance $\sigma^2$.
		\item 
		$X_i$ and $\varepsilon_i$ are independent.
		\item
		The regression function
		$\mu$ is Lipschitz continuous with parameter $L$.
	\end{enumerate}
\end{definition}

\begin{remark}
	The optimal rate of convergence for the regression problem of estimating $x\mapsto \E[Y|X=x]$ in Definition \ref{def:lipschitzRegression} is $N^{-2/(2+d)}$. Furthermore, the KNN algorithm with the right choice of the number of  neighbors and the Nadaraya--Watson estimator with the right kernels achieve this rate, and they are thus minimax optimal for this regression problem.
\end{remark}
Now let's define a related distribution on $(Y(0), Y(1), W, X)$.
\begin{definition} \label{definition:Dfam}
	Let $\mathcal{D}^{L}_{mn}$ be the family of distributions of 
	$(Y(0), Y(1), W, X) \in \R ^ N \times \R ^ N \times \{0,1\} ^ N\times [0,1]^{d \times N}$ such that:
	\begin{enumerate}
		\item 
		$N = m + n$.
		\item
		The features, 
		$X_i$, are i.i.d. uniformly distributed in $[0,1]^d$.
		\item There are exactly $n$ treated units,
		$$\sum_i W_i = n.$$
		\item The observed outcomes are given by
		$$
		Y_i(w) = \mu_w(X_i) + \varepsilon_{wi},
		$$ 
		where $(\varepsilon_{0i}, \varepsilon_{1i})$ is independent normally distributed with mean 0 and marginal variances $\sigma^2$.\footnote{We do not assume that $\varepsilon_{0i} \perp \varepsilon_{1i}$.}
		\item 
		$X, W$ and $\varepsilon = (\varepsilon_{0i}, \varepsilon_{1i})$ are independent.
		\item
		The response functions
		$\mu_0, \mu_1$ are Lipschitz continuous with parameter $L$.
	\end{enumerate}
\end{definition}
Note that if $(Y(0), Y(1), W, X)$ is distributed according to a distribution in $D^L_{mn}$, then 
$(Y(0), X)$ given $W=0$ and $(Y(1), X)$ given $W=1$ have marginal distributions in $F^L$, and 
$(X, \mu_1(X) - Y(0))$ given $W = 0$ and $(X, Y(1) - \mu_0(X))$ given $W = 1$ have distributions in $F^{2L}$, and we therefore conclude that $D^L_{mn} \in S\left(\frac2{2+d},\frac2{2+d}\right)$.

We will first prove in Theorem \ref{theorem:optimalbalanced} that the best possible rate that can be uniformly achieved for distributions in this family is
$$\mathcal{O}(n^{2/(2+d)} + m^{2/(2+d)}).$$
This is precisely the rate the T-learner with the right base learners achieves. 
We will then show in Theorem \ref{Theorem:Lipschitz_Convergence} that the X-learner with the KNN estimator for both stages achieves this optimal rate as well, and conclude that both the T- and X-learners achieve the optimal minimax rate for this class of distributions.


\newcommand{\ld}{\D^L_{mn}}
\subsubsection*{Minimax lower bound}\label{App:Minimax}
In this section, we will derive a lower bound on the best possible rate for $\D^L_{mn}$. 
\begin{theorem}[Minimax Lower Bound] \label{theorem:optimalbalanced}
	Let $\hat \tau$ be an arbitrary estimator,  let 
	$a_1, a_2 > 0$, and let $c$ be such that for all $n,m\ge 1$,
	\begin{equation}\label{eq:minimaLB}
	\sup_{\Pcal \in \ld} \EMSE(\Pcal, \hat \tau^{mn})		
	\le
	c (m^{-a_0} + n^{-a_1});
	\end{equation}
	then $a_1$ and $a_2$ are at most $2/(2+d)$:
	$$a_0, a_1 \le 2/(2+d).$$
\end{theorem}

\begin{proof}[Proof of Theorem \ref{theorem:optimalbalanced}]
	To simplify the notation, we define $a = 2/(2+d)$.
	We will show by contradiction that $a_1 \le a$. The proof of $a_0$ is mathematically symmetric.
	We assume that $a_1$ is bigger than $a$, and we show that this implies that there exists a sequence of estimators
	$\hat \mu^n_{1}$, such that
	$$
		\sup_{\Pcal_1 \in F^L} 
		\E_{D_1^n \sim \Pcal_1^n}\Big[
		(
		\mu_1(\X) - \hat \mu^n_{1}(\X; \D_1^n)
		)^2
		\Big]
		\le 2c n^ {-a_1},
	$$
	which is a contradiction, since by the definition of $D_L^{mn}$, $\mu_{1}$ cannot be estimated at a rate faster than $n^{-a}$ (cf., \cite{gyorfi2006distribution}).
	Note that we write here $\hat \mu^n_{1}(\X; \D_1^n)$, because we want to be explicit that $\hat \mu_1^n$ depends only on the treated observations.
	
	Similiarly to $\hat \mu^n_{1}(\X; \D_1^n)$, we will use the notation $\hat\tau^{mn}(\X;\D_0^m, \D_1^n)$ to be explicit about the dependence of the estimator $\hat \tau^{mn}$ on the data in the control group, $\D_0^m$, and on the data in the treatment group, $\D_1^n$.
%
	Furthermore, note that in Definition \ref{definition:Dfam} each distribution in $\ld$ is fully specified by the distribution of $W$, $\varepsilon$, and the functions $\mu_1$ and $\mu_2$. Define $C_L$ to be the set of all functions $f:[0,1]^d \longrightarrow \R$ that are L-Lipschitz continuous. 
	For $f_1\in C_L$, define $\mathbb{D}(f_1)$ to be the distribution in $\ld$ with $\mu_0 = 0$, $\mu_1 = f_1$, $\varepsilon_{0} \perp \varepsilon_{1}$, and $W$ defined componentwise by 
			$$
				W_i = \begin{cases}
					1 \mbox{ if } i \le n \\
					0 \mbox{ otherwise.} 
				\end{cases}
			$$
	Then \refb{eq:minimaLB} implies that
	\begin{align*}
		c (m^{-a_0} + n^{-a_1}) 
		&\ge
		\sup_{\Pcal \in \ld} \E_{(\D_0^m \times \D_1^n) \sim \Pcal}
		\left[(\tau^\Pcal (\X) - \hat \tau^{mn}(\X;\D_0^m, \D_1^n))^2\right]
		\\&\ge
		\sup_{f_1 \in C_L} \E_{(\D_0^m \times \D_1^n) \sim \mathbb{D}(f_1)}
		\left[(\mu_1^{ \mathbb{D}(f_1)}(\X) - \hat \tau^{mn}(\X;\D_0^m, \D_1^n))^2\right].
	\end{align*}
	This follows, because in $\mathbb{D}(f_1)$, $\tau^{\mathbb{D}(f_1)} = \mu^{\mathbb{D}(f_1)}_1 = f_1$. We use here the notation $\tau^\Pcal$, $\tau^{\mathbb{D}(f_1)}$, and $\mu_1^{ \mathbb{D}(f_1)}$ to emphasize that those terms depend on the distribution of $\Pcal$ and ${\mathbb{D}(f_1)}$, respectively. 
	
	Let $\Pcal_0$ be the distribution of $\D_0^m = (X^0_i, Y^0_i)_{i=1}^N$ under $\mathbb{D}(f_1)$. Note that under $\Pcal_0$, $X_i \overset{iid}\sim [0,1]$, and $Y^0\overset{iid}\sim \N(0, \sigma^2)$, and $X^0$ and $Y^0$ are independent. In particular, $\Pcal_0$ does not depend on $f_1$. 
	We can thus write
	\begin{align*}
		c (m^{-a_0} + n^{-a_1}) 
		&\ge
		\sup_{f_1 \in C_L} \E_{(\D_0^m \times \D_1^n) \sim \mathbb{D}(f_1)}
		\left[\left(\mu_1^{ \mathbb{D}(f_1)}(\X) - \hat \tau^{mn}(\X;\D_0^m, \D_1^n)\right)^2\right]
		\\& =
		\sup_{f_1 \in C_L} \E_{\D_1^n \sim \mathbb{D}_1(f_1)}
		\E_{\D_0^n \sim \Pcal_0}
		\left[\left(\mu_1^{ \mathbb{D}_1(f_1)}(\X) - \hat \tau^{mn}(\X;\D_0^m, \D_1^n)\right)^2\right]
		\\& \ge
		\sup_{f_1 \in C_L} \E_{\D_1^n \sim \mathbb{D}_1(f_1)}
		\left[\left(\mu_1^{ \mathbb{D}_1(f_1)}(\X) - 	\E_{\D_0^n \sim \Pcal_0} \hat \tau^{mn}(\X;\D_0^m, \D_1^n)\right)^2\right].
	\end{align*}
	$\mathbb{D}_1(f_1)$ is here the distribution of $\D^n_1$ under $\mathbb{D}(f_1)$. For the last step we used Jensen's inequality. 
	
	Now choose a sequence $m_n$ in such a way that
		$m_n^{-a_1} + n^{-a_2} \le 2 n^{-a_1}$, and define 
	$$
		\hat \mu_{1}^n(x; \D_1^n) = 
		\E_{\D_0^{m_n} \sim \Pcal_0^{m_n}} \left[\hat \tau^{mn}(x;\D_0^{m_n}, \D_1^n)\right].
	$$
	Furthermore, note that 
	$$
		\left\{\mathbb{D}_1(f_1) ~:~ f_1 \in C_L \right\} = \left\{\Pcal_1 \in F^L\right\}
	$$
	in order to conclude that
	\begin{align*}
		2 c n^{-a_1} 
		\ge 
		c (m_n^{-a_0} + n^{-a_1}) 
		&\ge
		\sup_{f_1 \in C_L} \E_{\D_1^n \sim \mathbb{D}_1(f_1)}
		\left[\left(\mu_1^{\mathbb{D}_1(f_1)}(\X) - 	\hat \mu_1^{nm}(D_1^n; \X)\right)^2\right]
		\\&\ge
		 \sup_{\Pcal_1 \in F^L} \E_{D_1^n \sim \Pcal_1^n}
		\left[\left(\mu_1^{\Pcal_1^n}(\X) - 	\hat \mu_1^{nm}(D_1^n; \X)\right)^2\right].
	\end{align*}

	This is, however, a contradiction, because we assumed $a_1 > a$. 
\end{proof}

 
\subsubsection*{EMSE convergence of the X-learner} 

Finally, we can show that the X-learner with the right choice of base learners achieves this minimax lower bound.
\begin{theorem} \label{Theorem:Lipschitz_Convergence}
	Let $d > 2$ and assume 
	$(X, W, Y(0), Y(1))  \sim \Pcal \in \mathcal{D}_{mn}^{L}$. In particular,
	$\mu_0$ and $\mu_1$ are Lipschitz continuous with constant $L$, 
	$$
	|\mu_w(x) - \mu_w(z)| \le L \|x-z\|  ~~~\mbox{for $w \in \{0,1\}$},
	$$
	and $X \sim Unif([0,1]^d)$.
	
	Furthermore, let $\taun$ be the X-learner with 
	\begin{itemize}
		\item $g\equiv 0$,
		\item the base learner of the first stage for the control group $\hat \mu_0$, is a KNN estimator with constant $k_0=   \left\lceil(\sigma^2 / L^2)^{\frac{d}{2+d}} m^{\frac2{d+2}}\right\rceil$,
		\item the base learner of the second stage for the treatment group, $\hat \tau_1$, is a KNN estimator with constant $k_1 = \left\lceil(\sigma^2 / L^2)^{\frac{d}{2+d}} n^{\frac2{d+2}}\right\rceil$.
	\end{itemize}
	Then $\taun$ achieves the optimal rate	as given in Theorem \ref{theorem:optimalbalanced}. That is, there exists a constant $C$ such that
	\begin{equation}
	\E \|  \tau  -   \taun\|^2
	\le
	C 
	\sigma^{\frac4{d+2}} L^{\frac{2d}{2+d}}
	\left(m^{-2/ (2+d)} + n^{-2/ (2+d)} \right).
	\end{equation}
\end{theorem}

Note that in the third step of the X-learner, Equation [\ref{equ:3rdStep}], $\hat \tau_0$ and $\hat \tau_1$ are averaged:
$$    
	\taun(x) = g(x) \hat \tau_0^{mn}(x) + (1 - g(x)) \hat \tau_1^{mn}(x).
$$
By choosing $g\equiv 0$, we are analyzing $\hat\tau_1^{mn}$. By a symmetry argument it is straightforward to show that with the right choice of base learners, $\hat\tau_0^{mn}$ also achieves a rate of $\mathcal{O}\left(m^{-2/ (2+d)} + n^{-2/ (2+d)} \right)$. With this choice of base learners the X-learner achieves this optimal rate for every choice of $g$. 

We first state two useful lemmata that we will need in the proof of this theorem. 
\begin{lemma} \label{lemma::KNN rate}
	Let $\hat \mu_0^m$ be a KNN estimator based only on the control group with constant $k_0$, and let 
	      $\hat \mu_1^n$ be a KNN estinator based on the treatment group with constant $k_1$; then, by the assumption of Theorem \ref{Theorem:Lipschitz_Convergence},
	\begin{align*}
	\E [\|\hat \mu_0^{m} - \mu_0\|^2]
	&\le 
	\frac{\sigma^2}{k_0} + c L^2 \left(\frac{k_0}{{m}}\right)^{2/d},
	\\
	\E [\|\hat \mu_1^{n} - \mu_1\|^2]
	&\le 
	\frac{\sigma^2}{k_1} + c L^2 \left(\frac{k_1}{{n}}\right)^{2/d},
	\end{align*}
	for some constant $c$.
\end{lemma}
\begin{proof}[Proof of Lemma \ref{lemma::KNN rate}]
	This is a direct implication of Theorem 6.2 in \cite{gyorfi2006distribution}.
\end{proof}

\begin{lemma} \label{lemma:FirstNeighborBound}
	Let $x \in [0,1]^d$, $X_1, \ldots, X_n \overset{iid}{\sim} Unif([0,1]^d)$ and $d>2$. Define $\tilde{X}(x)$ to be the nearest neighbor of $x$; then there exists a constant $c$ such that for all $n > 0$,
	$$
	\E \|\tilde{X}(x) - x\|^2 \le \frac{c}{n^{2/d}}.
	$$
\end{lemma}
\begin{proof}[Proof of Lemma \ref{lemma:FirstNeighborBound}]
	First of all we consider 
	\begin{align*}
	\P( \| \tilde{X}(x) -  x \| \ge \delta ) 
	= (1 - \P(\|X_1 - x \| \le \delta)) ^ n \le (1 - \tilde c \delta^d) ^n
	\le e^{-\tilde c \delta^d n}.
	\end{align*}
	Now we can compute the expectation:
	\begin{align*}
	\E \|\tilde{X}(x)    & - x\|^2 
	=  
	\int_0^\infty  
	\P( \| \tilde{X}(x) -  x \| \ge \sqrt{\delta} ) 
	d\delta
	\le 
	\int_0^{d}
	e^{-\tilde c \delta^{d/2} {n}}
	d\delta    
	\le 
	\frac{1 - \frac{1}{-d/2 + 1}}{(\tilde c{n})^{2/d}}.
	\end{align*}
\end{proof}

\begin{proof}[Proof of Theorem \ref{Theorem:Lipschitz_Convergence}]
	Many ideas in this proof are motivated by \cite{gyorfi2006distribution} and \cite{bickel2015mathematical}.
	Furthermore, note that we restrict our analysis here only to $\hat \tau_1^{mn}$, but the analysis of $\hat \tau_0^{mn}$ follows the same steps. 
	
	We decompose $\hat \tau_1^{mn}$ into
	\begin{align*}
	\hat \tau_1^{mn}(x) 
	&= 
	\frac1{k_1}\sum_{i=1}^{k_1}
	\left[      
	Y^{1}_{(i,{n})}(x)
	- 
	\hat \mu^{m}_0 \left( X_{({i},{n})}^{1} (x) \right) 
	\right]
	=
	\hat\mu^{n}_1(x) - 
	\frac1{k_1}\sum_{i=1}^{k_1}
	\hat \mu^{m}_0 \left( X_{({i},{n})}^{1} (x) \right), 
	\end{align*}
	where the notation that 
	$
	\left(\big(X_{(1,{n_w})}^{w}(x) ,Y^{w}_{(1,{n_w})}(x)\big), \ldots, \big(X_{({n_w},{n_w})}^{w} (x), Y^{w}_{({n_w},{n_w})}(x)\big)\right)
	$
	is a reordering of the tuples $\big(X_{j}^{w}(x) ,Y^{w}_{j}(x)\big)$ such that $\|X_{(i,{n_w})}^{w} (x) - x\|$ is increasing in $i$. With this notation we can write the estimators of the first stage as
	\begin{align*}
	\hat\mu_0^{m}(x) = \frac1{k_0}\sum_{i=1}^{k_0}Y^{0}_{(i,{m})}(x),
	\quad\quad\mbox{and}\quad\quad
	\hat\mu_1^{n}(x) = \frac1{k_1}\sum_{i=1}^{k_1}Y^{1}_{(i,{n})}(x),
	\end{align*}
	and we can upper bound the EMSE with the following sum:
	\begin{align*}
	&\E [| \tau (\mathcal{X})- \hat \tau_1^{mn}(\mathcal{X}) |^2]
	\\=~&  
	\E \Big[\Big| 
	\mu_1(\mathcal{X}) 
	- \mu_0(\mathcal{X}) 
	- \hat \mu^{n}_1(\mathcal{X}) 
	+ \frac1{k_1}\sum_{i=1}^{k_1}\hat \mu^{m}_0(X^1_{(i,n)}(\mathcal{X})) 
	\Big|^2\Big]        
	\\\le~& 
	2\E \left[\left| 
	\mu_1(\mathcal{X}) 
	- \hat \mu^{n}_1(\mathcal{X})
	\right|^2\right]        
	+        
	2\E \Big[\Big| 
	\mu_0(\mathcal{X})  
	- \frac1{k_1}\sum_{i=1}^{k_1}\hat \mu^{m}_0(X^1_{(i,n)}(\mathcal{X})) 
	\Big|^2\Big].        
	\end{align*}
	
	The first term corresponds to the regression problem of estimating the treatment response function in the first step of the X-learner and we can 
	control this term with Lemma \ref{lemma::KNN rate}:
	\begin{equation*}
	\E [\|\mu_1 - \hat \mu_1^{n} \|^2]
	\le 
	\frac{\sigma^2}{k_1} + c_1 L^2 \left(\frac{k_1}{{n}}\right)^{2/d}.
	\end{equation*}

	The second term is more challenging: 
	\begin{align}
	& \frac12   \E \Big[\Big|
	\mu_0(\mathcal{X})  
	- \frac1{k_1}\sum_{i=1}^{k_1}\hat \mu^{m}_0(X^1_{(i,n)}(\mathcal{X})) 
	\Big|^2\Big]     \notag
	\\\le~&
	\E \Big[\Big|
	\mu_0(\mathcal{X})  
	-      \frac1{k_1k_0} \sum_{i=1}^{k_1} \sum_{j=1}^{k_0} \mu_0\Big(X^0_{(j,m)}(X^1_{(i,n)}(\mathcal{X}))\Big)
		\Big|^2\Big]   \label{eterm}
	\\+ ~&
	\E \Big[\Big|
	\frac1{k_1k_0} \sum_{i=1}^{k_1} \sum_{j=1}^{k_0} \mu_0\Big(X^0_{(j,m)}(X^1_{(i,n)}(\mathcal{X}))\Big)
	- \frac1{k_1}\sum_{i=1}^{k_1}\hat \mu^{m}_0(X^1_{(i,n)}(\mathcal{X})) 
		\Big|^2\Big].           \label{varterm}
	\end{align}
	
	\refb{varterm} can be bound as follows:
	\begin{align*}
	[\ref{varterm}] 
	= ~&
	\E \left(
	\frac1{k_1k_0} \sum_{i=1}^{k_1} \sum_{j=1}^{k_0} 
	\mu_0\Big(X^0_{(j,m)}(X^1_{(i,n)}(\mathcal{X}))\Big)
	- Y_{(j,m)}^0(X^1_{(i,n)}(\mathcal{X}))
	\right)^2
	\\ \le &
	\max_i \frac1{k^2_{m}} \sum_{j=1}^{k_0} \E \left(
	\mu_0\Big(X^0_{(j,m)}(X^1_{(i,n)}(\mathcal{X}))\Big)
	- Y_{(j,m)}^0(X^1_{(i,n)}(\mathcal{X}))
	\right)^2
	\\ = &
	\max_i \frac1{k^2_{m}} \sum_{j=1}^{k_0} \E \Bigg[ \E \bigg[\left( 
	\mu_0\Big(X^0_{(j,m)}(X^1_{(i,n)}(\mathcal{X}))\Big)
	- Y_{(j,m)}^0(X^1_{(i,n)}(\mathcal{X}))
	\right) ^2 \bigg\|\mathcal{D}, \mathcal{X}\bigg]\Bigg]
	\le 
	\frac{\sigma^2}{k_0}. 
	\end{align*}
	The last inequality follows from the assumption that, conditional on $\D$,
	$$
	Y^0_{(j,m)}(x) \sim \mathcal{N}\left(\mu_0\left(X^0_{(j,m)}(x)\right), \sigma^2\right).
	$$
	
	Next we find an upper bound for [\ref{eterm}]:
	\begin{align}
	[\ref{eterm}] 
	&\le          
	\E \Bigg( 
	\frac1{k_1k_0} \sum_{i=1}^{k_1} \sum_{j=1}^{k_0}
	\Big\|\mu_0(\mathcal{X})  
	-   \mu_0\Big(X^0_{(j,m)}(X^1_{(i,n)}(\mathcal{X}))\Big) 
	\Big\|
	\Bigg)^2 \notag
	\\& \le 
	\E \Bigg(
	\frac1{k_1k_0} \sum_{i=1}^{k_1} \sum_{j=1}^{k_0} 
	L \Big\| \mathcal{X} - X^0_{(j,m)}(X^1_{(i,n)}(\mathcal{X})) 
	\Big\|    
	\Bigg)^2 \notag
	\\& \le   
	L^2
	\frac1{k_1k_0} \sum_{i=1}^{k_1} \sum_{j=1}^{k_0} 
	\E  \Big\| \mathcal{X} - X^0_{(j,m)}(X^1_{(i,n)}(\mathcal{X})) 
	\Big\|^2  \label{Jensen}
	\\& \le   
	L^2
	\frac1{k_1} \sum_{i=1}^{k_1} 
	\E  \Big\| 
	\mathcal{X} - X^1_{(i,n)}(\mathcal{X}) 
	\Big\|^2 \label{simpleKNNdist}
	\\& +
	L^2
	\frac1{k_1k_0} \sum_{i=1}^{k_1} \sum_{j=1}^{k_0} 
	\E  \Big\| 
	X^1_{(i,n)}(\mathcal{X}) - X^0_{(j,m)}(X^1_{(i,n)}(\mathcal{X})) 
	\Big\|^2 \label{sndOrdKNNdist}
	\end{align}
	where [\ref{Jensen}] follows from Jensen's inequality.
	
	Let's consider [\ref{simpleKNNdist}].
	We partition the data into $A_1, \ldots, A_{k_1}$ sets, where the first $k_1 - 1$ sets have $\lfloor\frac{n}{k_1}\rfloor$
	elements and we define $\tilde X_{i,1}(x)$ to be the nearest neighbor of $x$ in $A_i$. Then we can conclude that
	\begin{align*}
	\frac1{k_1} \sum_{i=1}^{k_1} 
	\E  \Big\| 
	\mathcal{X} - X^1_{(i,n)}(\mathcal{X}) 
	\Big\|^2 
	&\le
	\frac1{k_1} \sum_{i=1}^{k_1} 
	\E  \Big\| 
	\mathcal{X} - \tilde X_{i,1}(\mathcal{X}) 
	\Big\|^2 
	\\ & =
	\frac1{k_1} \sum_{i=1}^{k_1} 
	\E \bigg [ \E\Big[\Big\| 
	\mathcal{X} - \tilde X_{i,1}(\mathcal{X}) 
	\Big\|^2 \Big | \mathcal{X} \Big] \bigg]
	\le
	\frac{\tilde c}{\lfloor\frac{n}{k_1}\rfloor^{2/d}}.
	\end{align*}
	Here the last inequality follows from Lemma \ref{lemma:FirstNeighborBound}. With exactly the same argument, we can bound [\ref{sndOrdKNNdist}] and we thus have
	$$
	[\ref{eterm}] 
	\le 
	L^2 \tilde c * 
	\left(
	\frac{1}{\lfloor\frac{n}{k_1}\rfloor^{2/d}} 
	+ 
	\frac{1}{\lfloor\frac{n_2}{k_2}\rfloor^{2/d}}
	\right) 
	\le
	2 \tilde c L^2 * 
	\left(
	\left(\frac{k_1}{n}\right)^{2/d} 
	+ 
	\left(\frac{k_0}{m}\right)^{2/d}
	\right).     
	$$
	
	Plugging everything in, we have
	\begin{align*}
	\E [| \tau (\mathcal{X})- \hat \tau_1^{mn}(\mathcal{X}) |^2] 
	&\le 
	2\frac{\sigma^2}{k_1} 
	+ 
	2(c_2 + 2\tilde c) L^2 \left(\frac{k_1}{{n}}\right)^{2/d}
	+
	2\frac{\sigma^2}{k_0} 
	+
	4 \tilde c  L^2  
	\left(\frac{k_0}{m}\right)^{2/d}
	\\&\le
	C\left(
	\frac{\sigma^2}{k_1} 
	+ 
	L^2 \left(\frac{k_1}{{n}}\right)^{2/d}
	+
	\frac{\sigma^2}{k_0} 
	+
	\left(\frac{k_0}{m}\right)^{2/d}
	\right)
	\end{align*}
	with $C = 2 \max(1, c_2 + 2\tilde c, 2 \tilde c)$.
\end{proof}
	
\section{Pseudocode}
\label{app:learners}

In this section, we  present pseudocode for the algorithms in this paper. 
We denote by $Y^0$ and $Y^1$ the observed outcomes for the control group and the treatment group, respectively. For example, $Y^1_i$ is the observed outcome of the $i$th unit in the treatment group. $X^0$ and $X^1$ are the features of the control units and the treated units, and hence $X^1_i$  corresponds to the feature vector of the $i$th unit in the treatment group. $M_k(Y\sim X)$ is  the notation for a regression estimator, which estimates $x \mapsto \E[Y|X=x]$. It can be any regression/machine learning estimator. In particular, it can be a black box algorithm.

\begin{algorithm}[h!] 
	\begin{algorithmic}[1]
		\Procedure{T-learner}{$X, \Yobs, W$}
		
		\State $\hat \mu_0 	= M_0(Y^0 \sim X^0)$
		\State $\hat \mu_1	 = M_1(Y^1 \sim X^1)$
		
		\item[]
		\State $\hat \tau(x) = \hat \mu_1(x) - \hat \mu_0(x)$ 
		\EndProcedure
	\end{algorithmic}
	\caption{T-learner}
	
	\label{algo:Tlearner}
		\algcomment{$M_0$ and $M_1$ are here some, possibly different, machine-learning/regression algorithms.}	
\end{algorithm}

\begin{algorithm}[h!] 
	\begin{algorithmic}[1]
		\Procedure{S-learner}{$X, \Yobs, W$}
		
		\State $\hat \mu 	= M(\Yobs \sim (X, W))$
						
		\State $\hat \tau(x) = \hat \mu(x,1) - \hat \mu(x,0)$ 
		\EndProcedure
	\end{algorithmic}
	\caption{S-learner}
	
	\label{algo:Slearner}
	\algcomment{$M(\Yobs \sim (X, W))$ is the notation for estimating $(x,w) \mapsto \E[Y|X=x, W=w]$ while treating $W$ as a 0,1--valued feature.}
\end{algorithm}

\begin{algorithm}[h!]
	\begin{algorithmic}[1]	
		\Procedure{X-learner}{$X, \Yobs, W, g$}

		\item[]
		\State $\hat \mu_0	= M_1(Y^0 \sim X^0)$ 
				\Comment{Estimate response function}
		\State $\hat \mu_1	 = M_2(Y^1 \sim X^1)$
		
		\item[]
		\State $\Dt^1_{i} =         Y^1_{i}                         -     \hat \mu_0(X^1_{i})$
				\Comment{Compute imputed treatment effects}
		\State $\Dt^0_{i} =     \hat \mu_1(X^0_{i})     -     Y^0_{i}$
		
		\item[]
		\State $\hat \tau_1      = M_3(\Dt^{1}      \sim     X^1)$
				\Comment{Estimate CATE in two ways}
		\State $\hat \tau_0     = M_4(\Dt^0     \sim     X^0)$
		
		\item[]
		\State $\hat \tau(x) = g(x) \hat \tau_0(x) + (1- g(x)) \hat \tau_1(x)$ 
				\Comment{Average the estimates}
		
		\EndProcedure
	\end{algorithmic}
	\caption[test]{X-learner}

	\label{aglo:Xlearner}
	
	\algcomment{$g(x) \in [0,1]$ is a weighting function that is chosen to minimize the variance of $\hat \tau(x)$. It is sometimes possible to estimate $\cov(\tau_0(x) , \tau_1(x) )$, and compute the best $g$ based on this estimate. However, we have made good experiences by choosing $g$ to be an estimate of the propensity score. }	
\end{algorithm}

\begin{algorithm}[h!]
	\begin{algorithmic}[1]
		\Procedure{F-learner}{$X, \Yobs, W$}
		
		\State 	$\hat e = M_e[W \sim X]$
		
		\State  $Y^*_i = \Yobs_i  \frac{W_i - \hat e(X_i)}{\hat e(X_i)(1- \hat e(X_i))}$
		
		\State 	$\hat \tau = M_{\tau}(Y^* \sim X)$
		
		\EndProcedure
	\end{algorithmic}
	\caption{F-learner}
	
	\label{algo:Flearner}
	
\end{algorithm}

\begin{algorithm}[h!]
	\begin{algorithmic}[1]
		\Procedure{U-learner}{$X, \Yobs, W$}
		
		\State 	$\hat \mu_{obs} = M_{obs}(Y^{obs} \sim X)$

		\State 		$\hat e = M_e[W \sim X]$
				
		\State 	$R_i = (\Yobs_i - \hat \mu_{obs}(X_i)) / (W_i - \hat e(X_i))$

		\State 		$\hat \tau = M_{\tau}(R \sim X)$
		\EndProcedure
	\end{algorithmic}
	\caption{U-learner}
	
	\label{algo:Ulearner}
\end{algorithm}

\begin{algorithm}[h!]
	\caption{Bootstrap Confidence Intervals 1}\label{alg:computeCI}
	\begin{algorithmic}[1]

		\Procedure{computeCI}{
			\newline$x$: features of the training data, 
			\newline $w$: treatment assignments of the training data, 
			\newline $y$: observed outcomes of the training data,
			\newline $p$: point of interest}

		\State $S_0 = \{i ~ : ~ w_i = 0 \}$\; 
		\State $S_1 = \{i ~ : ~ w_i = 1 \}$\;
		\State $n_0 = \# S_0$\;
		\State $n_1 = \# S_1$\;
			
		\For{$b$ in $\{1, \ldots, B\}$}
			\State $s_b^* = c($sample($S_0$, replace = T, size = $n_0$), sample($S_1$, replace = T, size = $n_1$)$)$\;		
			
			\State $x^*_b = x[s_b^*]$ \;
			\State $w^*_b = w[s_b^*]$ \;
			\State $y^*_b = y[s_b^*]$\;
			\State $\hat \tau^*_b(p) = $ learner$(x^*_b, w^*_b, y^*_b)(p)$\;
		\EndFor

		\State 	$ \hat \tau(p) =$ learner$(x, w, y)(p)$\;
		
		\State 	$\sigma = sd(\{\hat \tau^*_b(p)\}_{b=1}^B)$

		\State \Return{$(\hat \tau(p) - q_{\alpha/2} \sigma,  \hat \tau(p) + q_{1 - \alpha/2} \sigma)$}
		
		\EndProcedure
	\end{algorithmic}
	\algcomment{For this pseudo code we use \texttt{R} notation. For example, $c()$ is here a function that combines its arguments to form a vector.}
\end{algorithm}

\begin{algorithm}[h!]
			
	\caption{Bootstrap Confidence Intervals 2}\label{alg:computeCI2}
	\begin{algorithmic}[1]

		\Procedure{computeCI}{
			\newline$x$: features of the training data, 
			\newline $w$: treatment assignments of the training data, 
			\newline $y$: observed outcomes of the training data,
			\newline $p$: point of interest}
		\State $S_0 = \{i ~ : ~ w_i = 0 \}$\;
		\State $S_1 = \{i ~ : ~ w_i = 1 \}$\;
		\State $n_0 = \# S_0$\;
		\State $n_1 = \# S_1$\;
		
		\For{$b$ in $\{1, \ldots, B\}$}
		\State $s_b^* = c($sample($S_0$, replace = T, size = $n_0$), sample($S_1$, replace = T, size = $n_1$)$)$\;

		\State $x^*_b = x[s_b^*]$ \;
		\State $w^*_b = w[s_b^*]$ \;
		\State $y^*_b = y[s_b^*]$\;
		\State $\hat \tau^*_b(p) = $ learner$(x^*_b, w^*_b, y^*_b)(p)$\;
		\EndFor

		\State 	$ \tilde \tau(p) = \frac1B \sum_{b=1}^B \hat \tau^*_b(p)$\;
		
		\State For all $b$ in $\{1, \ldots, B\}$ and $j$ in $\{1,\ldots, n\}$ define 
		$$S^*_{bj} = \# \{k : s^*_{b}[k] = j \}$$\;
		\State For all $j$ in $\{1,\ldots, n\}$ define $\overline {S^*_{\cdot j}} = \frac1B \sum_{b=1}^B S^*_{bj}$ and 
		$$\cov_j = 	\frac1B \sum_{b=1}^B  (\hat \tau^*_b(p) -\tilde \tau(p)) (S^*_{bj} - \overline {S^*_{\cdot j}} )$$\;
		\State $\sigma =  \left(\sum_{j=1}^n Cov_j^2\right)^{0.5}$ \;
		\State \Return{$(\tilde \tau(p) - q_{\alpha/2} \sigma,  \tilde \tau(p) + q_{1 - \alpha/2} \sigma)$}
		
		\EndProcedure
	\end{algorithmic}
	\algcomment{This version of the bootstrap was proposed in \cite{efron2014estimation}.}
\end{algorithm}

\begin{algorithm}
	\caption{Monte Carlos Bias Approximation}\label{alg:computeBIAS}
	\begin{algorithmic}[1]
		
		\Procedure{approximateBIAS}{
			\newline$x$: features of the full data set, 
			\newline $w$: treatment assignments of the full data set, 
			\newline $y(0)$: potential outcome under control of the full data set,
			\newline $y(1)$: potential outcome under treatment of the full data set,
			\newline $S$: indices of observations that are not in the test set,
			\newline $S_T$: indices of the training set,
			\newline $p$: point of interest,
			\newline $\tau(p)$: the true CATE at p}
		
		\For{$i$ in $\{1, \ldots, 1000\}$}
		\State 	Create a new treatment assignment by permuting the original one,
		$$
		w_i = \mbox{sample}(w, \mbox{replace} = F).
		$$	
		\State Define the observed outcome, 
		$$
		y_i = y(1) w_i + y(0) (1 - w_i).
		$$
		\State 	Sample uniformly a training set of $50,000$ observations,
		\begin{align*}
		s^*_i &= \mbox{sample}(S, \mbox{replace} = F, \mbox{size} = 50,000),\\
		w^*_i &= w_i[s^*_i],\\
		x^*_i & = x[s^*_i],\\
		y^*_i & = y_i[s^*_i].
		\end{align*}				
		\State 	Estimate the CATE, 
		$$
		\hat \tau^*_i(p) =  \mbox{learner}(x^*_i, w^*_i, y^*_i)(p).
		$$
		\EndFor
		
		\State 	$ \bar \tau^*(p) =\frac1{1000} \sum_{i = 1}^{1000} \hat \tau^*_i(p)$\;	
		
		\State \Return{$\bar \tau^*(p) -  \tau(p)$}
		
		\EndProcedure
	\end{algorithmic}
	\algcomment{ This algorithm is used to compute the bias in a simulation study where the potential outcomes and the CATE function are known. $S$, the indices of the units that are not in the test set and $S_T$, the indices of the units in the training set are not the same, because the training set is in this case a subset of 50,000 units of the full data set.}
\end{algorithm}

\begin{algorithm}[h!]
	\caption{Bootstrap Bias}\label{alg:estimateBIAS}
	\begin{algorithmic}[1]

		\Procedure{estimateBIAS}{
			\newline$x$: features of the training data, 
			\newline $w$: treatment assignments of the training data, 
			\newline $y$: observed outcomes of the training data,
			\newline $p$: point of interest}
		
		\State $S_0 = \{i ~ : ~ w_i = 0 \}$\; 
		\State $S_1 = \{i ~ : ~ w_i = 1 \}$\;
		\State $n_0 = \# S_0$\;
		\State $n_1 = \# S_1$\;
		
		\For{$b$ in $\{1, \ldots, B\}$}
		\State $s_b^* = c($sample($S_0$, replace = T, size = $n_0$), sample($S_1$, replace = T, size = $n_1$)$)$\;		
		\State $x^*_b = x[s_b^*]$ \;
		\State $w^*_b = w[s_b^*]$ \;
		\State $y^*_b = y[s_b^*]$\;
		\State $\hat \tau^*_b(p) = $ learner$(x^*_b, w^*_b, y^*_b)(p)$\;
		\EndFor

		\State 	$ \hat \tau(p) =$ learner$(x, w, y)(p)$\;
		
		\State 	$ \bar \tau^*(p) =\frac1B \sum_{i = 1}^B \hat \tau^*_i(p)$\;	
		
		\State \Return{$\bar \tau^*(p) -  \hat \tau(p)$}
		
		\EndProcedure
	\end{algorithmic}
\end{algorithm}

 

\begin{thebibliography}{10}

\bibitem{GerberGreenLarimer}
Gerber AS, Green DP, Larimer CW (2008) Social pressure and voter turnout:
  Evidence from a large-scale field experiment.
\newblock {\em American Political Science Review} 102(1):33--48.

\bibitem{transphobia}
Broockman D, Kalla J (2016) Durably reducing transphobia: A field experiment on
  door-to-door canvassing.
\newblock {\em Science} 352(6282):220--224.

\bibitem{foster2013subgroup}
Foster JC (2013) Ph.D. thesis (The University of Michigan).

\bibitem{athey2015machine}
Athey S, Imbens GW (2015) Machine learning methods for estimating heterogeneous
  causal effects.
\newblock {\em stat} 1050(5).

\bibitem{hill2011bayesian}
Hill JL (2011) Bayesian nonparametric modeling for causal inference.
\newblock {\em Journal of Computational and Graphical Statistics}
  20(1):217--240.

\bibitem{green2012modeling}
Green DP, Kern HL (2012) {Modeling heterogeneous treatment effects in survey
  experiments with Bayesian additive regression trees}.
\newblock {\em Public Opinion Quarterly} 76(3):491--511.

\bibitem{hansen2009attributing}
Hansen BB, Bowers J (2009) Attributing effects to a cluster-randomized
  get-out-the-vote campaign.
\newblock {\em Journal of the American Statistical Association}
  104(487):873--885.

\bibitem{wager2015estimation}
Wager S, Athey S (2017) Estimation and inference of heterogeneous treatment
  effects using random forests.
\newblock {\em Journal of the American Statistical Association}.

\bibitem{Kalla2017}
Kalla JL, Broockman DE (2018) The minimal persuasive effects of campaign
  contact in general elections: Evidence from 49 field experiments.
\newblock {\em American Political Science Review} 112(1):148--166.

\bibitem{SekhonSATO17}
Sekhon JS, Shem-Tov Y (2017) Inference on a new class of sample average
  treatment effects.
\newblock {\em arXiv preprint arXiv:1708.02140}.

\bibitem{rubin1974estimating}
Rubin DB (1974) Estimating causal effects of treatments in randomized and
  nonrandomized studies.
\newblock {\em Journal of Educational Psychology} 66(5):688.

\bibitem{splawa1990application}
Splawa-Neyman J, Dabrowska DM, Speed T (1990) On the application of probability
  theory to agricultural experiments.
\newblock {\em Statistical Science} 5(4):465--472.

\bibitem{rosenbaum1983central}
Rosenbaum PR, Rubin DB (1983) The central role of the propensity score in
  observational studies for causal effects.
\newblock {\em Biometrika} 70(1):41--55.

\bibitem{tian2014simple}
Tian L, Alizadeh AA, Gentles AJ, Tibshirani R (2014) A simple method for
  estimating interactions between a treatment and a large number of covariates.
\newblock {\em Journal of the American Statistical Association}
  109(508):1517--1532.

\bibitem{Powers2017}
Powers S, et~al. (2017) Some methods for heterogeneous treatment effect
  estimation in high-dimensions.
\newblock {\em arXiv preprint arXiv:1707.00102}.

\bibitem{KuenzelHTE}
K\"unzel S, Tang A, Bickel P, Yu B, Sekhon J (2017) hte: An implementation of
  heterogeneous treatment effect estimators and honest random forests in c++
  and r.
\newblock {\em https://github.com/soerenkuenzel/hte}.

\bibitem{scornet2015consistency}
Scornet E, Biau G, Vert JP (2015) Consistency of random forests.
\newblock {\em The Annals of Statistics} 43(4):1716--1741.

\bibitem{Hughbart}
Chipman HA, George EI, {E. McCulloch} R (2010) Bart: Bayesian additive
  regression trees.
\newblock {\em The Annals of Applied Statistics} 4(1):266--298.

\bibitem{stone1982optimal}
Stone CJ (1982) Optimal global rates of convergence for nonparametric
  regression.
\newblock {\em The Annals of Statistics} 10(4):1040--1053.

\bibitem{birge1983approximation}
Birg{\'e} L (1983) Approximation dans les espaces m{\'e}triques et th{\'e}orie
  de l'estimation.
\newblock {\em Probability Theory and Related Fields} 65(2):181--237.

\bibitem{gyorfi2006distribution}
Gy{\"o}rfi L, Kohler M, Krzyzak A, Walk H (2006) {\em A distribution-free
  theory of nonparametric regression}.
\newblock (Springer Science \& Business Media).

\bibitem{Tsybakov2009}
Tsybakov AB (2009) {\em Introduction to nonparametric estimation}.
\newblock (Springer Series in Statistics).

\bibitem{bickel2015mathematical}
Bickel PJ, Doksum KA (2015) {\em Mathematical statistics: Basic ideas and
  selected topics}.
\newblock (CRC Press) Vol.{}~2.

\bibitem{hajek1967basic}
H{\'a}jek J (1967) On basic concepts of statistics in {\em Proceedings of the
  Fifth Berkeley Symposium on Mathematical Statistics and Probabilities}.
\newblock Vol.{}~1, pp. 139--162.

\bibitem{le1956asymptotic}
Le~Cam L (1956) On the asymptotic theory of estimation and testing hypotheses
  in {\em Proceedings of the Third Berkeley Symposium on Mathematical
  Statistics and Probability}.
\newblock Vol.{}~1.

\bibitem{mann2010there}
Mann CB (2010) {Is there backlash to social pressure? A large-scale field
  experiment on voter mobilization}.
\newblock {\em Political Behavior} 32(3):387--407.

\bibitem{michelson2016risk}
Michelson MR (2016) The risk of over-reliance on the institutional review
  board: An approved project is not always an ethical project.
\newblock {\em PS: Political Science \& Politics} 49(02):299--303.

\bibitem{ScienceReal}
Bohannon J (2016) {For real this time: Talking to people about gay and
  transgender issues can change their prejudices}.
\newblock {\em Science}.

\bibitem{broockman2015irregularities}
Broockman D, Kalla J, Aronow P (2015) {Irregularities in LaCour (2014)}.
\newblock {\em Work. pap., Stanford Univ. http://stanford. edu/{\~{}}
  dbroock/broockman{\_}kalla{\_}aronow{\_}lg{\_}irregularities. pdf}.

\bibitem{BKS2017design}
Broockman DE, Kalla JL, Sekhon JS (2017) The design of field experiments with
  survey outcomes: A framework for selecting more efficient, robust, and
  ethical designs.
\newblock {\em Political Analysis} 25:435--464.

\bibitem{breiman2001random}
Breiman L (2001) Random forests.
\newblock {\em Machine Learning} 45(1):5--32.

\bibitem{heckman1997making}
Heckman JJ, Smith J, Clements N (1997) Making the most out of programme
  evaluations and social experiments: Accounting for heterogeneity in programme
  impacts.
\newblock {\em The Review of Economic Studies} 64(4):487--535.

\bibitem{Lewandowski2009}
Lewandowski D, Kurowicka D, Joe H (2009) Generating random correlation matrices
  based on vines and extended onion method.
\newblock {\em Journal of Multivariate Analysis} 100(9):1989--2001.

\bibitem{dorie2017automated}
Dorie V, Hill J, Shalit U, Scott M, Cervone D (2017) Automated versus
  do-it-yourself methods for causal inference: Lessons learned from a data
  analysis competition.
\newblock {\em arXiv preprint arXiv:1707.02641}.

\bibitem{liu2013asymptotic}
Liu H, Yu B (2013) Asymptotic properties of lasso+mls and lasso+ridge in sparse
  high-dimensional linear regression.
\newblock {\em Electronic Journal of Statistics} 7:3124--3169.

\bibitem{efron2014estimation}
Efron B (2014) Estimation and accuracy after model selection.
\newblock {\em Journal of the American Statistical Association}
  109(507):991--1007.

\bibitem{putter2012resampling}
Putter H, Van~Zwet WR (2012) Resampling: consistency of substitution estimators
  in {\em Selected Works of Willem van Zwet}.
\newblock (Springer), pp. 245--266.

\end{thebibliography}
\end{document}